\documentclass[11pt]{article}

\newcommand{\de}{\,\mathrm{d}}                               
\newcommand{\e}{\operatorname{e}}                               
                     
\newcommand{\inc}{\mathrm{inc}}

\newcommand{\p}{\partial}

\newcommand{\imag}{\mathrm{Im}\,}

\newcommand{\R}{\mathbb{R}}       
\newcommand{\C}{\mathbb{C}}       
\newcommand{\N}{\mathbb{N}}

\newcommand{\der}{{\rm D}}
\newcommand{\erf}{\operatorname{erf}}

\newcommand{\erfc}{\operatorname{erfc}}
\usepackage{xcolor}
\definecolor{capa}{RGB}{50, 90, 160}
\definecolor{delftblue}{RGB}{0,61,165}
\usepackage[colorlinks=true,
  linkcolor=delftblue,
  urlcolor=delftblue,
  citecolor=delftblue]{hyperref}

\usepackage{enumitem}

\usepackage{multirow}
\usepackage{amsmath}
\usepackage{amsfonts}
\usepackage{amsmath}
\usepackage{amssymb}  
\usepackage{graphicx}
\usepackage{caption}
\usepackage{mathrsfs}
\usepackage{upgreek}
\usepackage{amsthm}
\usepackage{subfig}
\usepackage{booktabs}
\usepackage{authblk}
\usepackage{cite}
\usepackage{bbm}
\usepackage{url}

\usepackage[]{algorithm2e}

\newtheorem{theorem}{Theorem}[section]

\newtheorem{proposition}[theorem]{Proposition}

\newtheorem{remark}[theorem]{Remark}

\usepackage{framed}
\newcommand{\rev}[1]{#1}

\definecolor{revffcolor}{rgb}{0,0,0}
\newcommand{\qord}{\mathfrak{r}}

\title{High-order kernel regularization of singular and hypersingular \\Helmholtz boundary integral operators\thanks{C.P.-A.\ and S.T.\ gratefully acknowledge the organizers of the program \emph{Interfaces and Unfitted Discretization Methods} at the Institut Mittag-Leffler, where this collaboration was initiated.}}

\author[3]{Luiz M. Faria\thanks{\href{mailto:luiz.maltez-faria@inria.fr}{luiz.maltez-faria@inria.fr}}}
\author[1]{Carlos P\'erez-Arancibia\thanks{\href{mailto:c.a.perezarancibia@utwente.nl}{c.a.perezarancibia@utwente.nl}}}
\author[2]{Svetlana Tlupova\thanks{\href{mailto:tlupovs@farmingdale.edu}{tlupovs@farmingdale.edu}}}
\affil[3]{\small{POEMS, CNRS, Inria, ENSTA Paris, Institut Polytechnique de Paris, France}}
\affil[1]{\small{Department of Applied Mathematics, University of Twente, Enschede, the Netherlands}}
\affil[2]{\small{Department of Mathematics, Farmingdale State College, SUNY, Farmingdale, NY 11735, USA}}
\date{\today}

\begin{document}

\maketitle
\begin{abstract}
This paper extends and analyzes the high-order kernel regularization framework of Beale \& Tlupova~\cite{beale2025high} to all four boundary integral operators of the Helmholtz Calder\'{o}n calculus in three dimensions: the single-layer, double-layer, adjoint double-layer, and hypersingular operators. To the best of our knowledge, this work provides the first high-order kernel regularization of the hypersingular operator for both the Helmholtz and Laplace equations in three dimensions. The regularization replaces each singular kernel with a smooth modification constructed from error functions together with a polynomial correction whose coefficients are determined through moment conditions. Alongside the derivation of the regularizing functions, the paper provides a unified error analysis of the combined regularization and quadrature discretization procedure. By coupling the regularization parameter to the mesh size, the two error contributions can be balanced, leading to explicit overall convergence rates that depend jointly on the order of the regularization and the degree of exactness of the surface quadrature rule. A key practical feature of the method is its implementation simplicity. Once the regularizing functions are determined, the numerical task reduces entirely to the evaluation of smooth surface integrals using standard quadrature, without the need for element-local solves, singularity-specific precomputations, or specialized quadrature rules. Although the modified kernel is generally incompatible with kernel-specific fast methods, this limitation is addressed through $\mathcal{H}$-matrix acceleration, applicable in a black-box manner. Numerical examples---including verification of the predicted convergence rates and solution of sound-soft and sound-hard scattering problems by smooth obstacles---demonstrate the accuracy and practicality of the proposed methodology.
\end{abstract}
\section{Introduction}
Boundary integral equation (BIE) methods constitute a powerful numerical framework for the solution of boundary value problems arising in a wide range of applications. Many of the main challenges associated with the numerical evaluation of boundary integral operators over smooth surfaces are by now well understood and effectively addressed, including the nonlocal nature of the kernels---typically handled through fast algorithms---and the development of accurate quadrature rules and regularization techniques for singular and nearly singular integrals. In this context, several well-established approaches achieve high-order or even spectral convergence in three dimensions. While such high-order accuracy is highly desirable, as it allows coarse discretizations to reach a prescribed accuracy, these methods are often technically involved and computationally expensive, frequently constituting the most costly component of the numerical discretization of BIEs. Moreover, and perhaps more importantly, the complexity of their implementation can hinder the broader adoption of BIE methods in applied settings, where the primary focus lies on solving practical problems rather than on the development and maintenance of sophisticated numerical techniques. 

A broad range of strategies has been proposed to address the accurate high-order evaluation of singular and nearly singular surface integrals, which is a central challenge in BIE methods. Classical singularity-cancellation methods apply local coordinate transformations---such as polar or Duffy-type changes of variables---to algebraically remove the integrand singularity prior to numerical integration; see, e.g.,~\cite{Sauter2010} for the Galerkin setting and~\cite{Bruno:2001ima} for high-order Nystr\"{o}m discretizations using partition-of-unity decompositions on surface patches. A separate family of approaches exploits global smooth parametrizations of the surface to obtain spectrally accurate methods~\cite{GaneshGraham2004,GaneshHawkins2006}. More recently, a combination of singularity-centered mesh refinement, precomputations, and a Chebyshev-mesh discretization of a non-overlapping representation of the surface was introduced in~\cite{BrunoGarza:20}, yielding high-order Nystr\"{o}m methods for both Helmholtz and Maxwell~\cite{hu2021chebyshev} scattering problems.
Locally corrected trapezoidal rules have emerged as an alternative, achieving high-order accuracy by adding local corrections near the singularity to an otherwise uniform trapezoidal rule, including unified treatments of singular and hypersingular operators on curved surfaces in three dimensions~\cite{WuMartinsson2021,WuMartinsson2023}. Quadrature by Expansion (QBX)~\cite{klockner2013quadrature} takes a different approach by expanding the layer potential in a local Taylor or multipole series centered at an off-surface point and evaluating the resulting smooth series on the boundary.\rev{~QBX has since been extended to large collections of axisymmetric scatterers~\cite{afklinteberg2016fast}, where the rotational symmetry of each particle is exploited to reduce the cost of the expansion, as well as to general three-dimensional geometries~\cite{wala2019fast,wala2020optimization,siegel2018local}.} The Density Interpolation Method (DIM)~\cite{faria2021general,perez2019planewave,perez2019harmonic} pursues a related goal by constructing, for each target point, a local PDE-compatible high-order interpolant of the density; Green's formula is then applied to this interpolant to produce an analytically computable correction that is subtracted from the original operator, thereby lowering the strength of the kernel singularity and reducing the integration to a smooth problem. While all these methods can achieve high-order or spectral accuracy in three dimensions, they generally involve specialized quadrature precomputations, element-local solves, or other structural features that add to the implementation complexity, and are typically tied to a specific surface discretization framework.

A class of kernel regularization techniques for the numerical evaluation of singular and nearly singular surface integrals in three dimensions was introduced by Beale and collaborators~\cite{beale2016simple}. These methods rely on mollifying the singular kernels through the use of error functions. While the original formulations were limited to low-order accuracy, recent developments have extended the approach to high order for the Laplace and Stokes equations: first through Richardson extrapolation~\cite{beale2024extrapolated}, and subsequently through suitable modifications of the regularizing functions that introduce additional degrees of freedom to enforce moment conditions, thereby achieving high-order convergence of the regularization error~\cite{beale2025high}.

One of the principal advantages of this class of techniques is their remarkable simplicity and flexibility. The regularization is carried out analytically at the level of the kernel: the singular kernel is replaced by a smooth modification that can be integrated to high order using standard, off-the-shelf surface quadrature rules, with no special treatment of the singularity required. This stands in contrast to other high-order singular integration methods of comparable accuracy and generality. Moreover, the approach is agnostic of the surface discretization: once the regularizing functions have been determined, any sufficiently accurate quadrature rule for smooth integrands can be employed, making the method applicable in a variety of discretization frameworks. \rev{This simplicity comes, however, with a genuine difficulty that is addressed in the present work. The accuracy of the overall procedure results from two competing error contributions. The regularization error decreases as the regularization parameter $\delta\to0$, whereas the regularized integrand becomes increasingly peaked in that regime, so that the quadrature error produced by a fixed mesh of size $h$ grows as $\delta\to0$. High-order accuracy is therefore attainable only if $\delta$ is coupled to $h$ in such a way that the two contributions are balanced---which requires a quantitative understanding of how the quadrature error depends on $\delta$, and is precisely what the analysis in Section~\ref{sec:discretization_error} provides.} 

In this work, we extend the methodology to all four boundary integral operators of the Helmholtz Calder\'{o}n calculus: the single-layer, double-layer, adjoint double-layer, and hypersingular operators, the latter understood in the sense of Hadamard finite-part integrals. A particularly novel aspect of this extension is the derivation of regularizing functions for the hypersingular operator, which, to the best of our knowledge,  constitutes the first high-order kernel regularization of the hypersingular operator for both the Helmholtz and Laplace equations, with the Laplace case following as a natural specialization of the present framework.
\rev{The construction of these regularizing functions follows the same general framework as for the weaker singularities---mollify the radial factor of the kernel and fix the free coefficients of a polynomial correction by imposing radial moment conditions---but the $\mathcal{O}(|x-y|^{-3})$ singularity of the hypersingular kernel, where $x$ and $y$ denote the target and source points on the surface~$\Gamma$, makes each of these steps more demanding. The polynomial correction requires two additional terms in order for the regularized integrand to remain bounded across the diagonal, the moment conditions must be interpreted in the Hadamard finite-part sense, and the smooth factor of the kernel must be split into two contributions of different radial structure, each requiring its own regularizing function.}

A central contribution of this work is a rigorous, unified error analysis of the combined regularization and quadrature discretization procedure (Section~\ref{sec:discretization_error}). The analysis accounts for both sources of error: (i) the regularization error, of order $\mathcal{O}(\delta^\mathfrak{m})$, governed by the number of moment conditions imposed on the regularizing function and thus controllable to arbitrary order; and (ii) the quadrature discretization error arising from numerical integration of the smooth regularized integrand, which depends on the degree of exactness $\mathfrak{q}$ of the surface quadrature rule and on $\delta$. By optimally coupling $\delta$ to the mesh size $h$ via a derived power law $\delta \propto h^{\mu^\star}$, the two error contributions are balanced, yielding a combined convergence rate $\mathfrak{o}^\star$ that depends on both $\mathfrak{m}$ and $\mathfrak{q}$. This rate can in principle be made arbitrarily large by increasing either parameter; it approaches $\mathfrak{q}+1$ as $\mathfrak{m}\to\infty$ and approaches $\mathfrak{m}$ as $\mathfrak{q}\to\infty$. The analysis is developed for all four operators, is valid for all positive wavenumbers $k$, and applies to flexible discretizations of the surface into non-overlapping curved patches equipped with interpolatory quadrature rules. 

A practical consideration concerns the interaction of the kernel modification with fast algorithms for matrix-vector products. The regularized kernel differs from the original Helmholtz kernel by a correction that decays as a Gaussian modulated by a polynomial factor whose degree grows with the regularization order $\mathfrak{m}$. This correction is not compactly supported and is not sufficiently confined to a neighborhood of the singularity: it introduces a rather non-local modification of the kernel that persists throughout the surface. As a consequence, unfortunately, off-the-shelf fast multipole method (FMM) implementations~\cite{greengard1997new,gumerov2005fast}---which are specifically tailored to the original Laplace/Helmholtz kernel---cannot be directly applied to the modified kernel. Instead, we employ $\mathcal{H}$-matrix compression~\cite{Hackbusch:2015}, which operates in a black-box, kernel-agnostic fashion and is therefore fully compatible with the modified kernel regardless of its specific form.

The paper is organized as follows. Section~\ref{sec:general} presents the general framework of the proposed methodology and the main theoretical results that support it. Sections~\ref{sec:single_layer},~\ref{sec:double_layer}, and~\ref{sec:hyper} then specialize the approach to the single-layer, (adjoint) double-layer, and hypersingular operators, respectively. Section~\ref{sec:discretization_error} presents the analysis of the discretization error. Section~\ref{sec:numerics} reports numerical results validating the proposed regularization technique. Finally, Section~\ref{sec:conclusion} summarizes the conclusions and outlines directions for future work.

\section{Kernel regularization framework}\label{sec:general}
To treat all four boundary integral operators of the Helmholtz Calder\'{o}n calculus---the single-layer $\mathsf{S}$, double-layer $\mathsf{K}$, adjoint double-layer $\mathsf{K}^\top$, and hypersingular $\mathsf{T}$ operators---within a common framework, we develop the regularization theory for a generic operator of the form~\eqref{eq:BIOP} below; the operator-specific details are presented in Sections~\ref{sec:single_layer},~\ref{sec:double_layer}, and~\ref{sec:hyper}.

Consider a generic boundary integral operator of the form
\begin{equation}\label{eq:BIOP}
\mathsf B[\varphi](x) = {\rm p.f.}\!\int_{\Gamma}\frac{\Phi(k|x-y|)}{|x-y|^{2p+1}}\Psi(x,y)\varphi (y)\de s(y),\quad x\in\Gamma,\quad p\in\{0,1,2\},
\end{equation}
where $\Phi:\R\to\C$ is a smooth and even function \rev{that grows at most algebraically at infinity, i.e. such that $|\Phi(\tau)|\leq C\,(1+|\tau|)^{\lambda}$, $\tau\in\R$, for some constants $C>0$ and $\lambda\geq0$}, and $\Psi:\Gamma\times\Gamma\to \R$ is smooth in both arguments. \rev{As shown in Sections~\ref{sec:single_layer}--\ref{sec:hyper}, the four operators considered here correspond to $\lambda=0$, $1$ and $2$.}
The exact expressions for these functions depend on the specific operator under consideration and will be defined in detail in the subsequent sections. The initials “p.f.” indicate that the integral in~\eqref{eq:BIOP} is to be understood as a Hadamard finite-part integral~\cite[sec.~3.2]{hsiao2021boundary}. When $p=2$ we have  $\Psi(x,y) = \mathcal{O}(|x-y|^2)$ as $|x-y| \to 0$, so the integral in~\eqref{eq:BIOP} is a Hadamard finite-part integral with an $\mathcal{O}(|x-y|^{-3})$--type singularity at $x = y \in \Gamma$, as occurs in the case of the hypersingular operator. Both the density $\varphi:\Gamma \to \C$ and the surface $\Gamma\subset\R^3$ are assumed to be smooth. 

\begin{remark}
  Both the density $\varphi:\Gamma \to \C$ and the surface $\Gamma\subset\R^3$ are required to possess a sufficient number of derivatives. \rev{Specifically, the regularization error estimate of order $\mathcal{O}(\delta^\mathfrak{m})$ relies on the expansion of Proposition~\ref{prop:expansion} with $2m=\mathfrak m+2p-1$, which holds for $\varphi \in C^{\mathfrak{m}+2p-1}(\Gamma)$ and $\Gamma \in C^{\mathfrak{m}+2p+1}$. Analyticity is not needed, since the Lagrange inversion theorem used in its proof extends to functions with finitely many derivatives~\cite{beale2026lagrange}. The quadrature error analysis with degree of exactness $\mathfrak{q}$ of Section~\ref{sec:discretization_error} additionally requires $\varphi \in C^{\mathfrak{q}+1}(\Gamma)$ and $\Gamma \in C^{\mathfrak{q}+2}$, the extra derivative of $\Gamma$ accounting for the normal vectors entering $\Psi$ and for the surface Jacobian. It therefore suffices that $\varphi \in C^{\max(\mathfrak{m}+2p-1,\mathfrak{q}+1)}(\Gamma)$ and $\Gamma \in C^{\max(\mathfrak{m}+2p+1,\mathfrak{q}+2)}$. These conditions are sufficient but not necessarily sharp.}

\end{remark}

The sought regularized operator is defined by
\begin{equation}\label{eq:reg_operator}
\mathsf B_\delta[\varphi](x) = \int_{\Gamma}\frac{\Phi(k|x-y|)}{|x-y|^{2p+1}}\sigma_p\left(\frac{|x-y|}{\delta}\right)\Psi(x,y)\varphi (y)\de s(y),\quad x\in\Gamma,
\end{equation}
where  $\sigma_p$ is the regularization function and the parameter $\delta > 0$ will be referred to as the regularization parameter. Following~\cite{beale2024extrapolated,beale2025high},  $\sigma_p$ is assumed to take the form
\begin{subequations}\label{eq:reg_funcs}\begin{equation}
\sigma_p(t) :=\erf(t) +\frac{2}{\sqrt{\pi}}\e^{-t^2}P_p(t),
\end{equation}
where
\begin{equation}\label{eq:poly_term}
P_p(t):=t^{2p}\sum_{\ell=1}^n a_\ell t^{2\ell-1}+\begin{cases}
0,&p=0,\\
 -t,&p=1,\\ 
\displaystyle -t-\frac{2}{3}t^3,&p=2.\end{cases}
\end{equation}\end{subequations}
The number of terms $n \in \mathbb{N}$ in the sums above, on which $P_p$ and $\sigma_p$ implicitly depend, controls the order of the regularization and is chosen based on the prescribed accuracy and the singularity strength parameter~$p$.

The choice of the error function as the base for $\sigma_p$ is motivated by two key properties. First, 
\rev{$\erf(t)$ vanishes linearly as $t\to0^+$, and more generally admits the odd Maclaurin expansion recalled in~\eqref{eq:error_fun_asym} below. For $p=0$ the linear vanishing already matches the strength of the singularity, and $\erf(t)$ by itself makes the ratio $\sigma_0(r/\delta)/r$ tend to a finite limit as $r\to0^+$. For $p\in\{1,2\}$ this is no longer the case: $\erf(t)/t^{2p+1}$ diverges as $t\to0^+$, and the first $p$ terms of the expansion of $\erf$ must be cancelled explicitly. This is precisely the role of the additional terms $-t$ (for $p=1$) and $-t-\tfrac23 t^3$ (for $p=2$) that are built into the polynomial factor $P_p$ in~\eqref{eq:poly_term}. Indeed, only with those terms included does the ratio $\sigma_p(r/\delta)/r^{2p+1}$ possess the finite limit~\eqref{eq:limits} below, so that the regularized integrand~\eqref{eq:reg_operator} extends smoothly across the diagonal $x=y$.} Second, since $\erf(t) \to 1$ exponentially fast as $t \to \infty$, the correction $1 - \sigma_p(t)$ decays as a Gaussian modulated by a polynomial, ensuring that the regularization error is well-localized in a neighborhood of the target point $x$ of radius $\mathcal{O}(\delta)$. The polynomial factor $P_p$ introduces $n$ free coefficients $\{a_\ell\}_{\ell=1}^n$ that are determined by enforcing moment conditions (see Section~\ref{sec:single_layer} and subsequent sections), i.e., conditions that cancel successive leading-order terms in the asymptotic expansion of the regularization error as $\delta\to 0$, thereby raising the convergence order to $\mathcal{O}(\delta^\mathfrak{m})$ for any prescribed $\mathfrak{m}$.

We recall that the error function, 
\begin{equation}\label{eq:error_fun}
\erf(t):= \frac{2}{\sqrt{\pi}} \int_0^{t} \e^{-\tau^2} \de \tau,\quad t\in\R,
\end{equation}
satisfies 
\begin{equation}\label{eq:error_fun_asym}
\erf(t) =\frac{2}{\sqrt{\pi}}\e^{-t^2}\left(t+\frac{2}{3}t^3+\mathcal O\left(t^5\right)\right)\text{ as }t\to0\quad\text{and}\quad \erf(t)=1+\e^{-t^2} \mathcal O\left(t^{-1}\right) \text{ as } t\to\infty. 
\end{equation}

\rev{Explicitly, the limit}
\begin{equation}\label{eq:limits}
\lim_{r\to 0+}\frac{\sigma_p(r/\delta)}{r^{2p+1}}=\begin{cases}\displaystyle\frac{2}{\sqrt{\pi }\delta }(1+a_1),&p=0,\medskip\\
\displaystyle\frac{2}{3 \sqrt{\pi } \delta^3}(2+3a_1),&p=1,\medskip\\
\displaystyle\frac{2}{15 \sqrt{\pi } \delta^5}(4+15a_1),&p=2,
\end{cases}
\end{equation} 
\rev{holds, and $\sigma_p(t) = 1 + \mathcal{O}(P_p(t)\e^{-t^2})$ as $t \to \infty$, where $P_p$ is the polynomial of degree $2(p+n)-1$ defined in~\eqref{eq:poly_term}.}

Note that since $\Phi$ is an even function and~\eqref{eq:limits} holds, the integrand in~\eqref{eq:reg_operator} can be extended to a smooth function on the entire surface $\Gamma$. Consequently, the surface integral~\eqref{eq:reg_operator} in the definition of the regularized operator $\mathsf B_\delta$ can be evaluated with high accuracy using standard numerical integration techniques for smooth surfaces.
 
To analyze the regularization error, we look at the difference:
\begin{equation*}
\begin{split}
\mathsf B[\varphi](x)-\mathsf B_\delta[\varphi](x) =&~{\rm p.f.}\!\! \int_\Gamma \frac{\Phi(k|x-y|)}{|x-y|^{2p+1}}\left\{1-\sigma_p\left(\frac{|x-y|}{\delta}\right)\right\}\Psi(x,y)\varphi(y)\de s(y),\quad x\in\Gamma.    
\end{split}
\end{equation*}
The goal is to derive conditions on $\sigma_p$ that ensure this quantity exhibits $\mathcal{O}(\delta^\mathfrak{m})$, $\mathfrak m\in\N$, asymptotic behavior as $\delta \to 0$. These conditions will lead to a system of equations from which the coefficients $\{a_\ell\}_{\ell=1}^n$ of the polynomial $P_p$ in~\eqref{eq:poly_term} can be computed as functions of the wavenumber $k\geq0$ and the regularization parameter $\delta>0$.

To conduct the analysis locally on the surface, we consider  a Monge parametrization centered at the target point $x\in\Gamma$. For the target $x$ on  the smooth surface $\Gamma$, with tangent vectors $e_1,e_2\in T_x\Gamma$ and unit normal $\nu(x)$, we define the Monge map around $x$ as the smooth function
\begin{equation}\label{eq:monge_patch}
\chi : U \subset \mathbb{R}^2 \to  \Gamma,
\qquad 
\chi(\alpha) := x + \alpha_1 e_1 + \alpha_2 e_2 + h(\alpha)\,\nu(x), \quad \alpha=(\alpha_1,\alpha_2),
\end{equation}
where $h:U\to\R$ is smooth and satisfies $h(0)=0$ and $\nabla h(0)=0$, i.e., the plane spanned by $e_1$ and $e_2$ coincides with $T_x\Gamma$, the tangent to the surface at $x$. Without loss of generality we assume $e_i\cdot e_j=\delta_{i,j}$, $i,j\in\{1,2\}$.
The domain $U$ is chosen such that the mapping $\chi:U\to\Gamma$ is a diffeomorphism onto its image, which defines the Monge patch $\chi(U)=\mathcal{P} \subset \Gamma$.

Furthermore, we define $w_p := 1 - \sigma_p$ and let
\rev{$d = d(x) := \operatorname{dist}\left(x,\p\mathcal P\right) = \inf_{z\in\p\mathcal P}|x-z| >0$ denote the distance from the target point $x$ to the boundary of the Monge patch $\mathcal P=\chi(U)$. Note that $d(x)>0$ because $x=\chi(0)$ lies in the interior of $\mathcal P$, and that $d(x)$ is bounded below uniformly in $x\in\Gamma$ once a fixed finite atlas of Monge patches is chosen. Since $w_p(t)=\mathcal O\big(P_p(t)\e^{-t^2}\big)$ as $t\to\infty$ and $|x-y|\geq d$ for every $y\in\Gamma\setminus\mathcal P$, the contribution of $\Gamma\setminus\mathcal P$ to the regularization error is bounded by $\mathcal O\big(P_p(d/\delta)\e^{-(d/\delta)^2}\big)$. This bound decays faster than any power of $\delta$ as $\delta\to0$, and is therefore negligible relative to the $\mathcal O(\delta^{\mathfrak m})$ accuracy targeted here. Accordingly, the regularization error can be approximated as}
\begin{equation}\label{eq:reg_error_int}
\mathsf B[\varphi](x)-\mathsf B_\delta[\varphi](x) \approx{\rm p.f.}\!\int_{U}\frac{\Phi(k|x-\chi(\alpha)|)}{|x-\chi(\alpha)|^{2p+1}}w_p\left(\frac{|x-\chi(\alpha)|}{\delta}\right)\phi(\alpha)\de \alpha ,
\end{equation}
where 
\begin{equation}\label{eq:smooth_dens}
\phi(\alpha):=\varphi(\chi(\alpha))\Psi(x,\chi(\alpha)) \left|\frac{\p\chi(\alpha)}{\p {\alpha_1}}\times\frac{\p\chi(\alpha)}{\p{\alpha_2}}\right|.
\end{equation}
\rev{The symbol $\approx$ in~\eqref{eq:reg_error_int}, and in~\eqref{eq:error_integral} and the expressions derived from it below, thus denotes equality up to the exponentially small remainder $\mathcal O\big(P_p(d/\delta)\e^{-(d/\delta)^2}\big)$ that results from discarding $\Gamma\setminus\mathcal P$. No other approximation is involved, since the passage from the surface integral over $\mathcal P$ to the parametric integral over $U$ is exact.}

To estimate the integral on the right-hand side of~\eqref{eq:reg_error_int}, we further localize the analysis by introducing a suitable reparametrization in terms of the target–source distance $|x-\chi (\alpha)|$, following the approach introduced in~\cite{beale2024extrapolated}. The following proposition defines this reparametrization, establishes its invertibility, and allows it to be used as a change of variables in the error integral~\eqref{eq:reg_error_int}.

\begin{proposition}\label{prop:mapping}Consider the mapping 
\begin{equation}\label{eq:change_variables}
\xi(\alpha):=\begin{cases}\displaystyle |x-\chi(\alpha)| \frac{\alpha}{|\alpha|},&\alpha\in U\setminus\{0\},\\
 0,&\alpha=0, \end{cases}   
\end{equation}
where $\chi $ denotes the parametrization introduced in~\eqref{eq:monge_patch}.
Then, there exists an open set $U_0 \subset U$ with $0\in U_0$, such that $\xi|_{U_0}:U_0\to \xi(U_0)$ is a $C^1$ diffeomorphism.
\end{proposition}
\begin{proof}
The mapping under consideration can be expressed as
$$
\xi(\alpha) =   q(\alpha) \alpha,\quad q(\alpha):=\sqrt{1+\frac{h^2(\alpha)}{|\alpha|^2}},\quad \alpha\in U\setminus\{0\},
$$
where we have used the fact that $|x-\chi(\alpha)| =  \sqrt{|\alpha|^2+h^2(\alpha)}.$  
A second-order Taylor expansion about $\alpha = 0$ of the height function in~\eqref{eq:monge_patch} yields 
 $h(\alpha) = \mathcal{O}(|\alpha|^2)$ as $|\alpha|\to 0$, and similarly $\nabla h(\alpha) = \mathcal{O}(\alpha)$ as $|\alpha|\to 0$. It then follows that $q$ admits a $C^1$ extension to $U$, and so does $\xi$. Moreover, a direct computation yields 
 $$
 \operatorname{det}\der\xi(\alpha) =q^2(\alpha)+q(\alpha)\nabla q(\alpha)\cdot\alpha=
  1+\frac{h(\alpha)}{|\alpha|^2}\nabla h(\alpha)\cdot\alpha =1+\mathcal{O}(|\alpha|^2)\quad \text{as}\quad |\alpha|\to0
$$
establishing, via the application of the standard inverse function theorem~\cite{krantz2002implicit}, the existence of an open set $U_0\subset U$, $0\in U_0$, on which $\xi$ is a $C^1$ diffeomorphism.
\end{proof}

\begin{remark}In what follows, we assume that the domains of definition of both the Monge parametrization~\eqref{eq:monge_patch} and the reparametrization~\eqref{eq:change_variables} coincide, that is, $U=U_0$, where $U_0$ is the open set introduced in Proposition~\ref{prop:mapping}. This assumption entails no loss of generality, since the surface atlas and the parameter $\delta>0$ can be suitably redefined to ensure both this condition and the approximability requirement~\eqref{eq:reg_error_int}.
\end{remark} 

It follows from Proposition~\ref{prop:mapping} that we can perform the change of variables $\alpha=\alpha(\xi)$ in~\eqref{eq:reg_error_int}, which yields
\begin{equation}\label{eq:error_integral}
\mathsf B[\varphi](x)-\mathsf  B_\delta[\varphi](x) \approx {\rm p.f.}\!\!\int_{\xi(U)}\frac{\Phi(k|\xi|)}{|\xi|^{2p+1}}w_p\left(\frac{|\xi|}{\delta}\right)\tilde\phi(\xi)\de \xi,    
\end{equation}
where $\tilde\phi$ is the function defined below in~\eqref{eq:non_smooth_dens}. Note that, in deriving \eqref{eq:error_integral}, we also used the identity $|\xi(\alpha)| = r(\alpha)$.

Next we establish a key power-series expansion, whose proof is adapted from~\cite{beale2024extrapolated}. This expansion, that is here presented for arbitrarily high-order, will then be used to derive the equations that determine the unknown coefficients defining the desired regularizing function $\sigma_p$.
\begin{proposition}\label{prop:expansion}
Let $\tilde\phi:\xi(U_0)\to \C$ be defined by
\begin{equation}\label{eq:non_smooth_dens}
\tilde\phi(\xi) :=\phi(\alpha(\xi))\left|\operatorname{det}\der \alpha(\xi)\right|,
\end{equation}
where $\phi$ is the smooth function introduced in~\eqref{eq:smooth_dens}, and $\alpha:\xi(U_0)\to U_0$ denotes the inverse of the mapping~$\xi$ defined in~\eqref{eq:change_variables}. \rev{Let $m\in\mathbb{N}_{\ge 2}$, and assume that the height function $h$ in~\eqref{eq:monge_patch} belongs to $C^{2m+2}(U_0)$ and $\phi$ to $C^{2m}(U_0)$.} Then $\tilde\phi$ admits the expansion
\begin{equation}\label{eq:imp_expansion}
\tilde\phi(\xi)=\phi(0)+\nabla\phi(0)\cdot\xi+\sum_{j=2}^{2m-1} |\xi|^j\sum_{\beta : |\beta|=j}c_{\beta,j}u^{\beta} +\mathcal{O}(|\xi|^{2m})\quad\text{as}\quad |\xi|\to0,
\end{equation}
for suitable coefficients $\{c_{\beta,j}\}$. Here, $u := \xi/|\xi|$, and multi-indices $\beta\in\mathbb{N}_0^2$ are understood in the standard sense.\end{proposition}
\begin{proof}
We first examine in more detail the invertibility of the mapping~\eqref{eq:change_variables} to then develop the series expansion~\eqref{eq:imp_expansion}.  
In view of the fact that $u=\xi/|\xi| =\alpha/|\alpha|$, we only need to establish the invertibility of the radial part of the transformation. We then fix $u$ and consider the function $\tilde q(|\alpha|;u) = q(\alpha)$ given by
$$\tilde q(|\alpha|;u) := \frac{|\xi(\alpha)|}{|\alpha|}= \sqrt{1+\left(\frac{h(|\alpha|u)}{|\alpha|}\right)^2}.$$

We let $f(|\alpha|;u) = |\xi(\alpha)| = |\alpha|\tilde q(|\alpha|;u)$ and construct the \rev{expansion} of the inverse function $g(|\xi|;u) = f^{-1}(|\xi|;u)$ (in the first variable) by applying the Lagrange inversion theorem~\cite{krantz2002implicit}\rev{, which extends to functions with finitely many derivatives, with an expansion of finite order~\cite{beale2026lagrange}}. Note that the sought inverse mapping yields $\alpha(\xi) = g(|\xi|;u)u$. 

\rev{Since $h\in C^{2m+2}$, the function $h(|\alpha|u)/|\alpha|=\int_0^1\nabla h(t|\alpha|u)\cdot u\de t$ is of class $C^{2m+1}$ in $|\alpha|$, and so are $\tilde q(\cdot\,;u)$ and $f(\cdot\,;u)$, with derivatives bounded uniformly in $u$. The Lagrange inversion theorem for functions of finite smoothness~\cite[Theorem~2]{beale2026lagrange} then yields}
\begin{equation}\label{eq:inv_func}
g(|\xi|;u) = \sum_{j = 1}^{\rev{2m}} g_j(u) |\xi|^j\rev{+\mathcal O(|\xi|^{2m+1})\quad\text{as}\quad|\xi|\to0},\qquad g_j(u)=\frac{1}{j}\left[|\alpha|^{j-1}\right] \tilde q(|\alpha|;u)^{-j},
\end{equation}
\rev{uniformly in $u$. Moreover, $g(\cdot\,;u)\in C^{2m+1}$ by the inverse function theorem, so that the expansion~\eqref{eq:inv_func} may be differentiated once term by term, with remainder $\mathcal O(|\xi|^{2m})$.}
(Here, $\left[x^{j-1}\right] f(x)$ denotes the coefficient of $x^{j-1}$ in the Taylor expansion of $f$ around $x=0$.)

To determine the coefficients $\{g_j(u)\}_{j\in\N}$, we need to look into the \rev{Taylor} expansion of $\tilde q(|\alpha|;u)^{-j}$, for $j\in\N_0$, around $\alpha=0$.  Leveraging the assumed smoothness of the surface $\Gamma$ and the fact that both $h(0)$ and $\nabla h(0)$ vanish, we  have\rev{, as $|\alpha|\to0$,}
\begin{align}
h(|\alpha| u) =&~ \sum_{\beta:\rev{2\leq|\beta|\leq 2m+1}}\frac{\der^\beta h(0) }{\beta !}u^\beta |\alpha|^{|\beta|}\rev{+\mathcal O(|\alpha|^{2m+2})}, \nonumber\\
\left(\frac{h(|\alpha| u)}{|\alpha|}\right)^2=&~\sum_{\beta:\rev{4\leq|\beta|\leq 2m+2}} c^{(0)}_{\beta}u^\beta  |\alpha|^{|\beta|-2}\rev{+\mathcal O(|\alpha|^{2m+1})}, \nonumber\\
\tilde q(|\alpha|;u) = \left(1+\left(\frac{h(|\alpha| u)}{|\alpha|}\right)^2\right)^{\frac12}=&~1+\sum_{\beta:\rev{4\leq|\beta|\leq 2m+2}} c^{(1)}_{\beta}u^\beta  |\alpha|^{|\beta|-2}\rev{+\mathcal O(|\alpha|^{2m+1})}, \nonumber\\
\tilde q(|\alpha|;u)^{-1}=\left(1+\left(\frac{h(|\alpha| u)}{|\alpha|}\right)^2\right)^{-\frac12}=&~1+\sum_{\beta:\rev{4\leq|\beta|\leq 2m+2}} c^{(2)}_{\beta}u^\beta  |\alpha|^{|\beta|-2} \rev{+\mathcal O(|\alpha|^{2m+1})},\nonumber\\
\tilde q(|\alpha|;u)^{-j}=\left(1+\left(\frac{h(|\alpha| u)}{|\alpha|}\right)^2\right)^{-\frac{j}{2}}=&~1+\sum_{\beta:\rev{4\leq|\beta|\leq 2m+2}} c^{(3)}_{\beta,j}u^\beta  |\alpha|^{|\beta|-2} \rev{+\mathcal O(|\alpha|^{2m+1})},\quad j\in\N,\label{eq:psi_j}
\end{align}
 where the coefficients $\{c^{(l)}_\beta\}$, $l\in\{0,1,2,3\}$, depend on $\{\der^\beta h(0)/\beta!\}$ and the derivatives of the (real-analytic) functions  $\sqrt{1+t^2}$ and $1/\sqrt{1+t^2}$ at $t=0$. 
In view of~\eqref{eq:psi_j}, the coefficients in the expansion~\eqref{eq:inv_func} are  given\rev{, for $j\leq 2m$,} by
$$
g_j(u)=\begin{cases} 1,&j=1,\\
0,&j=2,\\
\displaystyle \sum_{\beta:|\beta|=j+1} \frac{c^{(3)}_{\beta,j}}{j} u^\beta,&j\geq 3.
\end{cases}
$$
Replacing the coefficients in~\eqref{eq:inv_func}, we  conclude that the inverse mapping can be expressed as
$
 \alpha(\xi) =g(|\xi|;u)u = \mu(|\xi|,u)\,\xi$
where 
\begin{equation}\label{eq:exp_mu}\begin{split}
\mu(|\xi|,u) :=&~ 1+ \sum_{j= 3}^{\rev{2m}}\sum_{\beta:|\beta|= j+1}\frac{c^{(3)}_{\beta,j}}{j}u^\beta |\xi|^{|\beta|-2}\rev{+\mathcal O(|\xi|^{2m})}\\
=&~1+ \sum_{\beta:\rev{4\leq|\beta|\leq 2m+1}}c^{(4)}_\beta u^\beta |\xi|^{|\beta|-2}\rev{+\mathcal O(|\xi|^{2m})}\quad\rev{\text{as}\quad|\xi|\to0}.
\end{split}\end{equation}
In the last identity above we have combined the double sum into one with new coefficients $\{c^{(4)}_\beta\}$.

On the other hand, a direct computation yields that the Jacobian determinant of the inverse transformation is given by 
$$
\operatorname{det}\der\alpha(\xi)=\mu^2(|\xi|,u) +\mu(|\xi|,u)\nabla\mu(|\xi|,u)\cdot \xi = \mu(|\xi|,u)\left(\mu(|\xi|,u) +|\xi|\frac{\p\mu(|\xi|,u)}{\p|\xi|}\right).
$$
In view of the last identity above and~\eqref{eq:exp_mu}\rev{, which may be differentiated once term by term,} we obtain
\begin{align}\label{eq:jac_det}
\operatorname{det}\der\alpha(\xi) =&~ 1+\sum_{\beta:\rev{4\leq|\beta|\leq 2m+1}}c^{(5)}_\beta u^\beta |\xi|^{|\beta|-2}\rev{+\mathcal O(|\xi|^{2m})},
\end{align}
\rev{as $|\xi|\to0$}.  Then, \rev{for $|\xi|$} sufficiently small so that $\operatorname{det}\der\alpha(\xi)$ in~\eqref{eq:jac_det} remains positive, we obtain---using again the expansion~\eqref{eq:exp_mu}---that for every multi-index $\gamma\in\N^2_0$ it holds
$$
\mu(|\xi|,u)^{|\gamma|}|\operatorname{det}\der\alpha(\xi)| =~ 1+\sum_{\beta:\rev{4\leq|\beta|\leq 2m+1}}c^{(6)}_{\beta,\gamma} u^\beta |\xi|^{|\beta|-2}\rev{+\mathcal O(|\xi|^{2m})},\nonumber\\
$$
and hence
\begin{align}
\alpha(\xi)^\gamma |\operatorname{det}\der\alpha(\xi)|=&~ \xi^\gamma+\sum_{\beta:\rev{4\leq|\beta|\leq 2m+1}}c^{(6)}_{\beta,\gamma} u^{\beta+\gamma} |\xi|^{|\beta|+|\gamma|-2}\rev{+\mathcal O(|\xi|^{2m+|\gamma|})},\label{eq:series_expan_2}
\end{align}
\rev{as $|\xi|\to0$}. Clearly, the coefficients $\{c^{(5)}_\beta\}$ and $\{c_{\beta,\gamma}^{(6)}\}$ in~\eqref{eq:jac_det} and~\eqref{eq:series_expan_2} depend on the coefficients $\{c_\beta^{(4)}\}$ in~\eqref{eq:exp_mu} and the multi-index $\gamma$.

Therefore, from the Taylor expansion
$$
\phi(\xi)  = \sum_{\gamma:|\gamma|\leq 2m-1} \frac{\der^{\gamma}{\phi}(0)}{\gamma!}\alpha^{\gamma} + \mathcal{O}(|\alpha|^{2m})\quad\text{as}\quad |\alpha|\to0,
$$
of the smooth function defined in~\eqref{eq:smooth_dens},  we obtain from~\eqref{eq:series_expan_2} that the function in~\eqref{eq:non_smooth_dens} admits the expansion
\begin{align}
\tilde\phi(\xi)  =&  \sum_{\gamma:|\gamma|\leq 2m-1} \frac{\der^{\gamma}{\phi}(0)}{\gamma!}\left( \xi^\gamma+\sum_{\beta:\rev{4\leq|\beta|\leq 2m+1}}c^{(6)}_{\beta,\gamma} u^{\beta+\gamma} |\xi|^{|\beta|+|\gamma|-2}\right) + \mathcal{O}(|\xi|^{2m})\nonumber\\
 =&  \sum_{\gamma:|\gamma|\leq 2m-1} \frac{\der^{\gamma}{\phi}(0)}{\gamma!}\xi^\gamma+ \sum_{\gamma:|\gamma|\leq 2m-1} \frac{\der^{\gamma}{\phi}(0)}{\gamma!}\sum_{\beta:\rev{4\leq|\beta|\leq 2m+1}}c^{(6)}_{\beta,\gamma} u^{\beta+\gamma} |\xi|^{|\beta|+|\gamma|-2} + \mathcal{O}(|\xi|^{2m})\nonumber\\
 =&\sum_{\gamma: |\gamma|\leq 2m-1} \frac{\der^{\gamma}{\phi}(0)}{\gamma!} \xi^\gamma+\frac{1}{|\xi|^2}\sum_{\beta:\rev{4\leq|\beta|\leq 4m}}c^{(7)}_\beta \xi^{\beta}+ \mathcal{O}(|\xi|^{2m})\nonumber\\
 =&\sum_{\gamma:|\gamma|\leq 2m-1} \frac{\der^{\gamma}{\phi}(0)}{\gamma!} \xi^\gamma+\frac{1}{|\xi|^2}\sum_{\beta:4\leq|\beta|\leq 2m+1} c^{(7)}_\beta \xi^{\beta} + \mathcal{O}(|\xi|^{2m})\quad\text{as}\quad |\xi|\to0, \label{eq:proto_series}
\end{align}
where the coefficients $\{c^{(7)}_\beta\}$ depend on $\{c^{(6)}_{\beta,\gamma}\}$ in~\eqref{eq:series_expan_2}.

Finally, to complete the proof, we replace $\xi = |\xi|u$ in~\eqref{eq:proto_series}, to obtain that the first sum can be expressed as 
$$
\sum_{\gamma:|\gamma|\leq 2m-1} \frac{\der^{\gamma}{\phi}(0)}{\gamma!} \xi^\gamma =\phi(0)+\nabla\phi(0)\cdot\xi + \sum_{j=2}^{2m-1}|\xi|^j\sum_{\beta:|\beta|=j}c^{(8)}_{\beta,j} u^\beta,
$$
while, similarly,  the second sum can be written as
$$
\frac{1}{|\xi|^2}\sum_{\beta:4\leq|\beta|\leq 2m+1} c^{(7)}_\beta \xi^{\beta} = \sum_{j=2}^{2m-1}|\xi|^j\sum_{\beta:|\beta|=j+2} c^{(9)}_{\beta,j} u^{\beta},
$$
for suitable coefficients $\{c^{(8)}_{\beta,j}\}$ and $\{c^{(9)}_{\beta,j}\}$ depending on $\{\der^{\gamma}\phi(0)/\gamma!\}$ and $\{c^{(7)}_\beta\}$, respectively. The proof is now complete.
\end{proof}

Substituting the expansion~\eqref{eq:imp_expansion} into~\eqref{eq:error_integral} and rewriting the resulting expression in polar coordinates,
$$
\rho := |\xi|, \qquad 
\theta := \operatorname{atan}\!\left(\frac{\xi_2}{\xi_1}\right),
$$
we can approximate the error integral~\eqref{eq:error_integral} by
$$
\mathsf B[\varphi](x)-\mathsf B_\delta[\varphi](x) \approx E_r(\delta,k),
$$
where
\begin{equation}\label{eq:error_polar}\begin{split}
    E_r(\delta,k)= \sum_{j=0}^{2m-1}\left\{{\rm p.f.}\!\! \int_{0}^\infty \Phi(k\rho)w_p\left(\frac{\rho}{\delta}\right)\rho^{j-2p}\de \rho\right\}\left\{\sum_{\beta:|\beta|=j}c_{\beta,j}\int_{0}^{2\pi}\cos^{\beta_1}(\theta)\sin^{\beta_2}(\theta)\de\theta\right\}\\
+R_{2m}(\delta,k),\quad  m\in\N,
\end{split}    
\end{equation}
where we have extended the definition of  coefficients $\{c_{\beta,j}\}$  to $j\in\{0,1\}$ in a suitable manner. 

The residual term satisfies the estimate
\begin{equation}\label{eq:res_bound}
|R_{2m}(\delta,k)| \lesssim \left|\int_{0}^\infty\Phi(k\rho)w_p\left(\frac{\rho}{\delta}\right)\rho^{2(m-p)}\de \rho\right|=\left|I_{p,m}(\varkappa)\right|\delta^{2(m-p)+1},
\end{equation}
where
\begin{equation}\label{eq:scaled_wn}
\varkappa := \delta k
\end{equation} is the scaled wavenumber and $I_{p,j}$ is the moment integral  defined by
\begin{equation}\label{eq:important_integral_gen}
I_{p,j}(\varkappa) :={\rm p.f.}\!\int_0^\infty\Phi(\varkappa t)w_p(t)t^{2(j-p)}\de t,\quad j\in\N_0,\quad p\in\{0,1,2\}.
\end{equation}

\rev{Note that the radial integrals in~\eqref{eq:error_polar} and~\eqref{eq:res_bound} are written over the unbounded interval $[0,\infty)$, whereas the polar change of variables above produces, for each fixed $\theta$, a radial integral over the finite range $\rho\in[0,\rho_{\max}(\theta)]$ determined by the parameter domain $U$. Since $\chi$ is a diffeomorphism onto $\mathcal P$, $\rho_{\max}(\theta)$ is bounded below, uniformly in $\theta$, by a positive multiple of the distance $d$ from the target $x$ to $\p\mathcal P$. Then, in view of the Gaussian decay $w_p(t)=\mathcal O\big(P_p(t)\e^{-t^2}\big)$ as $t\to\infty$, which dominates the algebraic growth of $\Phi$ assumed in Section~\ref{sec:general}, extending each radial integral to $[0,\infty)$ introduces an error $\mathcal O\big(\e^{-c(d/\delta)^2}\big)$, uniformly in $\theta$ and for some constant $c\in(0,1)$ independent of $\delta$, which decays faster than any power of $\delta$ and is therefore negligible relative to the desired $\mathcal{O}(\delta^{\mathfrak m})$ accuracy.}

Since the integrals with respect to $\theta$ in~\eqref{eq:error_polar} vanish identically whenever 
$|\beta| = \beta_1 + \beta_2$ is odd, only the even-indexed moments contribute. Hence, we may rewrite
the  error as
\begin{equation}\label{eq:expan_error}
E_r(\delta,k) = \sum_{j=0}^{m-1} \mu_j\delta^{2(j-p)+1} I_{p,j}(\varkappa) + R_{2m}(\delta,k),\quad m\in \N,
\end{equation} for some coefficients $\{\mu_j\}_{j=0}^{m-1}\subset\C$.

Our goal is to determine the coefficients $\big\{a_\ell(\varkappa)\big\}_{\ell=1}^n$ in the definition of $\sigma_p$ in~\eqref{eq:reg_funcs} in such a way that
\begin{equation}\label{eq:system_gen}
I_{p,j}(\varkappa) = {\rm p.f.}\!\int_0^\infty\Phi(\varkappa t)w_p(t)t^{2(j-p)}\de t=0,\quad \text{for}\quad  j\in\{0,1,\ldots,m-1\},   
\end{equation}
which, in view of~\eqref{eq:res_bound}, yields a regularization error $E_r(\delta, k)$ that satisfies the bound
\begin{equation}\label{eq:bound_reg_error}
|E_r(\delta, k)| \lesssim  |I_{p,m}(\varkappa)|\delta^{\mathfrak m},\qquad \mathfrak m:=2(m-p)+1,
\end{equation}
with an implicit constant independent of $\varkappa$ and $\delta$.

To that end, by substituting $w_p = 1 - \sigma_p$ into~\eqref{eq:important_integral_gen}, we can rewrite~\eqref{eq:system_gen} as the following $m\times n$ system  of equations:
\begin{equation}\label{eq:good_system}
  \sum_{\ell=1}^{n} a_\ell(\varkappa) A_{j+\ell}(\varkappa)  = b_{j}(\varkappa),\qquad j=0,\ldots,m-1    
\end{equation}
for the unknown coefficients $\big\{a_\ell\big\}_{\ell=1}^{n}$ in~\eqref{eq:reg_funcs}, where the matrix and vector coefficients are given by
\begin{subequations}\begin{align}
 A_{j+\ell}(\varkappa) :=&~\frac{2}{\sqrt{\pi}} \int_0^\infty \Phi(\varkappa t)\e^{-t^2}t^{2(j+\ell)-1}\de t\quad\text{and}    \label{eq:coeff_gen}\\
b_{j}(\varkappa) :=&~{\rm p.f.}\!\int_0^\infty \Phi(\varkappa t)\left\{\operatorname{erfc}(t)+\frac{2}{\sqrt{\pi}}\e^{-t^2}\left((\delta_{1,p}+\delta_{2,p})t+\frac{2}{3}\delta_{2,p}t^3\right)\right\}t^{2(j-p)}\de t.\label{eq:coeff_gen_formula}
\end{align}\label{eq:generic_coeff}\end{subequations}
Here, $\erfc := 1-\erf$ is the complementary error function\rev{, with $\erf$ as in~\eqref{eq:error_fun},} and $\delta_{p,q}$ is the Kronecker delta. \rev{The integrals in~\eqref{eq:generic_coeff} are absolutely convergent at infinity for the same reason as the moment integrals~\eqref{eq:important_integral_gen}. At the origin, however, the integrand of $b_j$ in~\eqref{eq:coeff_gen_formula} is not integrable for $j<p$, and the integral is understood, as indicated, in the Hadamard finite-part sense. The integrals~\eqref{eq:coeff_gen} defining $A_{j+\ell}$ are instead integrable there, since $2(j+\ell)-1\geq1$.}

\begin{remark}\label{rem:ill-cond_mat}The explicit expressions for the coefficients~\eqref{eq:generic_coeff} and the selection of the parameters $n$ and $m$ (for a given regularization error order $\mathfrak m$) that lead to a solvable linear system~\eqref{eq:good_system} depend on the specific boundary integral operator being regularized. The natural choice of selecting $n$ and $m$ such that the resulting linear system is square leads to a symmetric Hankel system matrix $\mathrm{A}(\varkappa)$. Unfortunately, $\mathrm{A}(\varkappa)$ becomes singular for a finite number of values of $\varkappa$. Although this occurs only for a limited set of parameter values, it leads to severe ill-conditioning in their vicinity. As a result, the computed coefficients $\big\{a_\ell(\varkappa)\big\}_{\ell=1}^{n}$ may become inaccurate, potentially degrading the overall accuracy of the kernel regularization procedure. A simple remedy is to choose $n$ and $m$  so that the matrix rows are linearly independent. In practice, this is accomplished by enlarging the system to a rectangular form---with more columns than rows---and solving it via the pseudoinverse. \rev{This is the approach taken in the numerical examples presented below in Section~\ref{sec:numerics}, with the single exception of the convergence study of Figure~\ref{fig:h_convergence_torus_k_1}, for which the square systems were used.}

\rev{The mechanism behind the degeneracy of the square system, the system sizes, and the conditioning and accuracy of the computed coefficients are discussed in Appendix~\ref{sec:coeff_systems}.}
\end{remark}

 In the sequel, we detail the construction of regularizations for the four boundary integral operators.

\subsection{Single-layer operator}\label{sec:single_layer}
We begin with the single-layer operator, defined by
\begin{equation}\label{eq:single_layer}
\mathsf S[\varphi](x) := \frac{1}{4\pi}\int_{\Gamma}\frac{\e^{ik|x-y|}}{|x-y|}\varphi (y)\de s(y),\quad x\in\Gamma.
\end{equation}

Since $\imag\!\left\{\e^{ir}/r\right\} = \sin(r)/r$ admits a direct analytic extension to $\R$, only the real part of the kernel requires regularization. By interpreting the single-layer operator~\eqref{eq:single_layer} as a member of the generic class of singular operators defined in~\eqref{eq:BIOP}, and recalling that~\eqref{eq:single_layer} is weakly singular with singularity order $p=0$, the regularized single-layer operator we consider can therefore be written in the form
\begin{equation}\label{eq:reg_single_layer}
\mathsf S_\delta[\varphi](x) := \frac{1}{4\pi}\int_{\Gamma}\frac{\Phi^{(\mathsf S)}(k|x-y|)}{|x-y|}\sigma^{(\mathsf S)}_0\left(\frac{|x-y|}{\delta}\right)\varphi (y)\de s(y)+\frac{i}{4\pi}\int_{\Gamma}\frac{\sin(k|x-y|)}{|x-y|}\varphi(y)\de s(y),\quad x\in\Gamma,
\end{equation}
where 
\begin{equation}\label{eq:funcs_S}\Phi^{(\mathsf S)}(r) := \cos(r),\quad r\in\R.
\end{equation}

To determine the coefficients $\big\{a_\ell^{(\mathsf S)}(\varkappa)\big\}_{\ell=1}^m$ defining a suitable regularizing function $\sigma_0^{(\mathsf S)}$ of the form~\eqref{eq:reg_funcs}, tailored to the single-layer operator, 
it suffices to enforce the conditions~\eqref{eq:system_gen} on the corresponding moment integrals given by:
$$
I^{(\mathsf S)}_{0,j}(\varkappa) := \int_{0}^\infty \cos(\varkappa t)w^{(\mathsf S)}_0(t)t^{2j}\de t,\quad j\in\N_0.
$$
The resulting $m$ conditions lead to the linear system in the form of~\eqref{eq:good_system} for the coefficients  $\big\{a_\ell^{(\mathsf S)}(\varkappa)\big\}_{\ell=1}^n$, which in this case is given by:
\begin{equation}\label{eq:system_SL_coeff}
  \sum_{\ell=1}^{n} a^{(\mathsf S)}_\ell(\varkappa) A^{(\mathsf S)}_{j+\ell}(\varkappa)  = b^{(\mathsf S)}_{j}(\varkappa),\qquad j\in\{0,\ldots,m-1\},   
\end{equation}
 where 
\begin{subequations}\label{eq:coeff_SL}\begin{align}
 A^{(\mathsf S)}_{j+\ell}(\varkappa) :=&~\frac{2}{\sqrt{\pi}} \int_0^\infty \cos(\varkappa t)\e^{-t^2}t^{2(j+\ell)-1}\de t\quad\text{and}\quad    \label{eq:mat_coeff_S}\\
b^{(\mathsf S)}_{j}(\varkappa) :=&~\int_0^\infty \cos(\varkappa t)\erfc(t)t^{2j}\de t ,\label{eq:coeff_gen_formula_S}
\end{align}\end{subequations}
for $j\in\{0,1,\ldots,m-1\}$ and $\ell\in\{1,\ldots,n\}$.  

This construction of $\sigma_0^{(\mathsf S)}$ guarantees that the regularization error for the single-layer operator, $E^{(\mathsf S)}_r(\delta, k)$, satisfies
\begin{equation}\label{eq:reg_error_S}
\big|E_r^{(\mathsf S)}(\delta, k)\big| \lesssim \big|I^{(\mathsf S)}_{0,m}(\varkappa)\big| \delta^{2m+1}.
\end{equation}
Clearly, for a fixed $\varkappa\geq0$,  the resulting regularization error in this case scales as
$\mathcal{O}(\delta^{\mathfrak m})$  as $\delta \to 0$ with $\mathfrak m=2m+1$.

\rev{The integrals in~\eqref{eq:coeff_SL} can be expressed in terms of the functions introduced in Appendix~\ref{sec:coeff}.} Specifically, one has
\begin{equation*}
A_j^{(\mathsf S)}(\varkappa) = C_{j-1}(\varkappa),\quad j\in\N,
\qquad\text{and}\qquad
b_j^{(\mathsf S)}(\varkappa) = \widetilde{C}_{j}(\varkappa),\quad j\in\N_0,
\end{equation*}
where the sets $\{C_j(\varkappa)\}_{j\in\N_0}$ and $\{\widetilde{C}_j(\varkappa)\}_{j\in\N_0}$ are introduced in  Propositions~\ref{prop:prop_1} and~\ref{prop:prop_2} in Appendix~\ref{sec:coeff},  which also establish the corresponding recurrence relations.

\begin{remark}\label{rem:laplace_SL} In view of the identities established in Propositions~\ref{prop:prop_1} and~\ref{prop:prop_2}, we readily obtain that, at $\varkappa = 0$, it holds that
$$
A_{j}^{(\mathsf S)}(0)= \frac{(j-1)!}{\sqrt{\pi}},\quad j \in\N,\qquad\text{and}\qquad b_{j}^{(\mathsf S)}(0)= \frac{1}{\sqrt{\pi}}\frac{j!}{2j+1},\quad j\in\N_0.
$$
These formulae enable to explicitly write down the linear system for the computation of the regularization function $\sigma_{0}^{(\mathsf S)}$ to arbitrarily high order. It can be shown, computing the determinant of the system matrix $\mathrm A^{(\mathsf S)}(0)\in\R^{n\times n}$, $[\mathrm A^{(\mathsf S)}(0)]_{i,j}= A_{i+j-1}(0)$, that the linear system for the coefficients $\big\{a^{(\mathsf S)}_\ell(0)\big\}_{\ell=1}^{n}$ is non-singular for all $m = n$. In fact, 
\begin{equation*}
\operatorname{det}\mathrm A^{(\mathsf S)}(0) = 
\frac{1}{\pi^{n / 2}} \prod_{j=0}^{n-1}(j!)^2\neq 0.
\end{equation*}
Although the matrix remains invertible, the rapid growth of its coefficients makes it severely ill-conditioned as its dimension increases, thereby limiting the achievable order of the method. However, for moderate sizes (e.g., $n = m$ with $n \leq 7$), the system can be solved symbolically, effectively avoiding ill-conditioning issues altogether. 

For example, to achieve a regularization error that scales as $\mathcal{O}(\delta^7)$ as $\delta \to 0$, we can choose $m = n = 3$.  
This yields a $3 \times 3$ linear system~\eqref{eq:system_SL_coeff}, whose unique  solution at $\varkappa = 0$ is
$$
a^{(\mathsf S)}_1(0)=\frac{11}{5},\qquad a^{(\mathsf S)}_2(0)=-\frac{26}{15},\qquad a^{(\mathsf S)}_3(0)=\frac{4}{15}.
$$
The corresponding regularization function then takes the form
$$
\sigma^{(\mathsf S)}_{0}(t) =\erf(t)+\frac{2}{\sqrt{\pi}}\e^{-t^2}\left\{\frac{11}{5}t-\frac{26}{15}t^3+\frac{4}{15}t^5\right\},
$$
which agrees with the one presented in~\cite{beale2025high} for the Laplace single-layer operator.

\end{remark}

\subsection{Double-layer and adjoint double-layer operators}\label{sec:double_layer}
Next we consider the double-layer and adjoint double-layer operators, which are defined by
$$
\mathsf K[\varphi](x) := \frac{1}{4\pi}\int_{\Gamma}\frac{(1-ik|x-y|)\e^{ik|x-y|}}{ |x-y|^3}\nu(y)\cdot(x-y)\varphi (y)\de s(y)
$$
and
$$
\mathsf K^\top[\varphi](x) :=\frac{1}{4\pi} \int_{\Gamma}\frac{(1-ik|x-y|)\e^{ik|x-y|}}{ |x-y|^3}\nu(x)\cdot(y-x)\varphi (y)\de s(y),
$$
for $x\in\Gamma$, respectively, where $\nu$ denotes the outward unit normal to the surface. 

By decomposing the kernels into singular and smooth components, 
and identifying the singular contributions with the generic operator in~\eqref{eq:BIOP}, 
we obtain the regularized double-layer and adjoint double-layer operators in the form
\begin{equation}\label{eq:reg_K}\begin{split}
\mathsf K_\delta[\varphi](x) :=&\frac{1}{4\pi}\int_{\Gamma}\frac{\Phi^{(\mathsf K)}(k|x-y|)}{|x-y|^3}\sigma^{(\mathsf K)}_1\left(\frac{|x-y|}{\delta}\right)\Psi^{(\mathsf K)}(x,y)\varphi (y)\de s(y)+ \\
&\frac{i}{4\pi}\int_{\Gamma}\frac{\sin(k|x-y|)-k|x-y| \cos (k|x-y|)}{|x-y|^3}\Psi^{(\mathsf K)}(x,y)\varphi (y)\de s(y)
\end{split}\end{equation}
and
\begin{equation}\label{eq:reg_Kp}\begin{split}
\mathsf K^\top_\delta[\varphi](x) :=& \frac{1}{4\pi} \int_{\Gamma}\frac{\Phi^{({\mathsf K})}(k|x-y|)}{|x-y|^3}\sigma^{(\mathsf K)}_1\left(\frac{|x-y|}{\delta}\right)\Psi^{(\mathsf K)}(y,x)\varphi (y)\de s(y)+ \\
&\frac{i}{4\pi}\int_{\Gamma}\frac{\sin(k|x-y|)-k|x-y| \cos (k |x-y|)}{ |x-y|^3}\Psi^{(\mathsf K)}(y,x)\varphi (y)\de s(y),
\end{split}\end{equation}
for $x\in\Gamma$, where
\begin{equation}\label{eq:funcs_double_layer}
\Phi^{(\mathsf K)}(r) := \cos (r)+r \sin (r)\quad (r\in\R)\quad\text{and}\quad \Psi^{(\mathsf K)}(x,y) :=\nu(y)\cdot (x-y)\quad  (x,y\in\Gamma). 
\end{equation}

The moment integrals~\eqref{eq:important_integral_gen} associated with both the double-layer and the adjoint double-layer kernels coincide, and are given by
$$
I^{(\mathsf K)}_{1,j}(\varkappa) := {\rm p.f.}\!\int_{0}^\infty \{\cos(\varkappa t)+\varkappa t\sin(\varkappa t)\}w^{(\mathsf K)}_1(t)t^{2(j-1)}\de t,\quad j\in\N_0.$$
This shows that the required regularizing function $\sigma_1^{(\mathsf K)}$ can be chosen identically for both operators.

Since $\Psi^{(\mathsf K)}(x,y) = \mathcal{O}(|x-y|^2)$ as $|x-y| \to 0$, the expansion~\eqref{eq:expan_error} of the regularization errors $E^{(\mathsf K)}_r(\delta,k)$ and $E_r^{(\mathsf K')}(\delta,k)$, for the double-layer and adjoint double-layer operators respectively, starts at $j=1$ rather than $j=0$.

Therefore, we obtain the following  linear system for the coefficients 
$\big\{a_\ell^{(\mathsf K)}(\varkappa)\big\}_{\ell=1}^n$:
\begin{equation}\label{eq:system_K}
\sum_{\ell=1}^{n} a_\ell^{(\mathsf K)}(\varkappa)A^{(\mathsf K)}_{j+\ell}(\varkappa) = b_j^{(\mathsf K)}(\varkappa),\quad j\in\{1,2,\ldots,m-1\},
\end{equation}
where the system matrix entries are given by
\begin{subequations}\begin{eqnarray}
A^{(\mathsf K)}_{j+\ell}(\varkappa) &:=&\frac{2}{\sqrt{\pi}} \int_0^\infty\left\{\cos(\varkappa t)+\varkappa t \sin (\varkappa t)\right\} \e^{-t^2}t^{2(j+\ell)-1}\de t\quad\text{and}\quad  \nonumber  \\
b^{(\mathsf K)}_{j}(\varkappa) &:=& \int_0^\infty \left\{\cos(\varkappa t)+ \varkappa t \sin (\varkappa t)\right\}\left\{\erfc(t)+\frac{2}{\sqrt{\pi }} t\e^{-t^2}\right\} t^{2(j-1)}\de t,\nonumber
\end{eqnarray}\end{subequations}
for $j\in\{1,2,\ldots,m-1\}$ and $\ell\in\{1,2,\ldots,n\}$.  

This construction of the regularizing function leads to the following bounds on the regularization errors:
\begin{equation}\label{eq:reg_error_K}
\big|E_r^{(\mathsf K)}(\delta,k)\big|\lesssim \big|I^{(\mathsf K)}_{1,m}(\varkappa)\big|\delta^{2m-1}\quad\text{and}\quad \big|E_r^{(\mathsf K')}(\delta,k)\big|\lesssim \big|I^{(\mathsf K)}_{1,m}(\varkappa)\big|\delta^{2m-1}.
\end{equation}

The following recursions for the coefficients involved in the linear system~\eqref{eq:system_K} can be obtained by applying Propositions~\ref{prop:prop_1} and~\ref{prop:prop_2} in Appendix~\ref{sec:coeff}. Indeed, since 
\begin{equation}\label{eq:coeff_DL_rec}
A^{(\mathsf K)}_j(\varkappa) = C_{j-1}(\varkappa)+\varkappa S_{j}(\varkappa),\ j\in\N_{\geq 2}, \quad\text{and}\quad b_j^{(\mathsf K)} = 2j\left\{\widetilde C_{j-1}(\varkappa)+\varkappa\widetilde S_{j-1}(\varkappa)\right\}+\varkappa^2\widetilde C_{j}(\varkappa),\  j\in \N,
\end{equation}
where $\{C_j(\varkappa)\}_{j\in\N_0}$, $\{S_j(\varkappa)\}_{j\in\N_0}$, $\{\widetilde{C}_j(\varkappa)\}_{j\in\N_0}$, and $\{\widetilde{S}_j(\varkappa)\}_{j\in\N_0}$ are introduced in  Propositions~\ref{prop:prop_1} and~\ref{prop:prop_2} in Appendix~\ref{sec:coeff},  which also establish the corresponding recurrence relations.

\begin{remark}\label{rem:laplace_DL_coeff}  In view of~\eqref{eq:exact_coeff_0} and~\eqref{eq:zero_erf}, at $\varkappa = 0$, the coefficients in~\eqref{eq:coeff_DL_rec} become
$$
A_{j}^{(\mathsf K)}(0)= C_{j-1}(0)= \frac{(j-1)!}{\sqrt{\pi}},\quad j\in\N_{\geq 2},\qquad\text{and}\qquad b_{j}^{(\mathsf K)}(0)=2j\widetilde C_{j-1}(0)=\frac{2}{\sqrt{\pi}}\frac{j!}{2j-1},\quad j \in\N.
$$
To achieve a regularization error that scales as $\mathcal{O}(\delta^7)$ as $\delta \to 0$, for example, we must choose $m = n+1 = 4$, which leads to a $3 \times 3$ linear system~\eqref{eq:system_K} for the sought coefficients. In this case, we obtain$$
a^{(\mathsf K)}_1(0)=\frac{118}{15},\qquad a^{(\mathsf K)}_2(0)=-\frac{68}{15},\qquad a^{(\mathsf K)}_3(0)=\frac{8}{15},
$$
which yield the regularization function
$$
\sigma^{(\mathsf K)}_1(t) =\erf(t)+\frac{2}{\sqrt{\pi}}\e^{-t^2}\left\{-t+\frac{118}{15}t^3-\frac{68}{15}t^5+\frac{8}{15}t^7\right\}
$$
that agrees with the one derived in~\cite{beale2025high} for the Laplace double-layer operator. 

As in the case of the single-layer operator, the system matrix $\mathrm A^{\mathsf K}(0)\in\R^{n\times n}$ for the coefficients is invertible since it can be shown that
\begin{equation*}
\operatorname{det}\mathrm A^{(\mathsf K)}(0) = 
\frac{n!}{\pi^{n / 2}} \prod_{j=0}^{n-1}(j!)^2\neq 0.
\end{equation*}  
\end{remark}

\subsection{Hypersingular operator}\label{sec:hyper}
We begin by expressing the hypersingular operator as the sum
$$
\mathsf T := \mathsf H +\mathsf W,
$$
which involves a hypersingular part and a weakly singular part, respectively defined as follows:
\begin{subequations}\begin{align}
\mathsf H[\varphi](x) :=&~ \frac{1}{4\pi}{\rm p.f.}\int_{\Gamma} \frac{\left(1-ik|x-y|\right)\e^{ik|x-y|}}{|x-y|^3}\nu(y)\cdot\nu(x)\varphi(y)\de s(y)\quad\text{and}\\
\mathsf W[\varphi](x) :=&~ \frac{1}{4\pi}\int_{\Gamma}\frac{(k^2|x-y|^2+3ik|x-y|-3)\e^{ik|x-y|}}{|x-y|^5}(x-y)\cdot\nu(y)(x-y)\cdot\nu(x)\varphi(y)\de s(y),
\end{align}\end{subequations}
for $x\in\Gamma$. 

Given the distinct singular nature of these two operators, we have to construct different regularization functions, $\sigma_2^{(\mathsf H)}$ and $\sigma_2^{(\mathsf W)}$, of the form~\eqref{eq:reg_funcs},  for each operator term, $\mathsf H$ and $\mathsf W$, respectively. In detail, the regularized operator components that we consider are given by 
\begin{subequations}\begin{align}\begin{split}
\mathsf H_\delta[\varphi](x) :=&\frac{1}{4\pi }\int_{\Gamma}\frac{\Phi^{(\mathsf H)}(k|x-y|)}{|x-y|^5}\sigma^{(\mathsf H)}_2\left(\frac{|x-y|}{\delta}\right)\Psi^{(\mathsf H)}(x,y)\varphi(y)\de s(y) +\\
&\frac{i}{4\pi}\int_{\Gamma}\frac{\sin(k|x-y|)-k|x-y|\cos(k|x-y|)}{|x-y|^3}\nu(y)\cdot\nu(x)\varphi(y)\de s(y)\label{eq:reg_H}
\end{split}\\
\begin{split}\mathsf W_\delta[\varphi](x):=&\frac{1}{4\pi}\int_{\Gamma}\frac{\Phi^{(\mathsf W)}(k|x-y|)}{|x-y|^5}\sigma^{(\mathsf W)}_2\left(\frac{|x-y|}{\delta}\right)\varphi(y)\de s(y)+\\
&\frac{i}{4\pi}\int_{\Gamma}\frac{3k|x-y|\cos(k|x-y|)+(k^2|x-y|^2-3)\sin(k|x-y|)}{ |x-y|^5}\Psi^{(\mathsf W)}(x,y)\varphi(y)\de s(y)
\end{split}\label{eq:reg_W}
\end{align}\end{subequations}
for $x\in\Gamma$, where
\begin{subequations}\begin{align}
\Phi^{(\mathsf H)}(r) :=&~\cos(r)+r\sin(r),&\qquad \Psi^{(\mathsf H)}(x,y):=&~[(x-y)\cdot(x-y)]\, [\nu(y)\cdot\nu(x)],\label{eq:smooth_H}\\
\Phi^{(\mathsf W)}(r) :=&~(r^2-3)\cos(r)-3r\sin(r),&\qquad \Psi^{(\mathsf W)}(x,y):=&~[(x-y)\cdot\nu(y)]\, [(x-y)\cdot\nu(x)],\label{eq:smooth_W}
\end{align}\label{eq:funcs_H}\end{subequations}
for $r\in\R$ and $x,y\in\Gamma$.
\begin{remark}\label{rem:about_H}
Note that the kernel associated with the operator $\mathsf H$ can be expressed with a weaker singularity, due to the factor $|x-y|^2 = (x-y)\cdot(x-y)$ appearing in the smooth function $\Psi^{(\mathsf H)}$. In this case, the limiting value of the regularized integrand as $|x-y| \to 0$ is proportional to $\delta^{-3}$ (see~\eqref{eq:limits}). In order to mitigate such potentially large values arising for small $\delta$, the kernel is written in the present form so as to ensure that the regularized integrand vanishes at $x = y$.
\end{remark}
In view of~\eqref{eq:important_integral_gen}, we have that the corresponding moment integrals are  given by 
\begin{subequations}\label{eq:mom_int_hyper}\begin{align}
I_{2,j}^{(\mathsf H)}(\varkappa):=&~{\rm p.f.}\!\int_{0}^\infty \{\cos(\varkappa t)+\varkappa t\sin(\varkappa t)\}w^{(\mathsf H)}_2(t)t^{2(j-2)}\de t,\\
I_{2,j}^{(\mathsf W)}(\varkappa):=&~{\rm p.f.}\!\int_{0}^\infty \{((\varkappa t)^2-3)\cos(\varkappa t)-3\varkappa t\sin(\varkappa t)\}w^{(\mathsf W)}_2(t)t^{2(j-2)}\de t,\quad j\in\N_{0}.
\end{align}\end{subequations}

As for the other boundary integral operators, we have decomposed each kernel into a singular component and a smooth remainder. The singular parts were then identified within the generic singular operator framework introduced in Section~\ref{sec:general}. Note that, since $\Psi^{(\mathsf W)}(x,y)=\mathcal{O}(|x-y|^4)$ as $|x-y|\to 0$,  $\mathsf W$ is indeed weakly singular.  Consequently, the corresponding expansion~\eqref{eq:expan_error} of the regularization error $E_r^{(\mathsf W)}(\delta,k)$ for $\mathsf W$ starts at $j=2$. On the other hand, since $\Psi^{(\mathsf H)}(x,y) = |x-y|^2$ as $|x-y|\to 0 $, the corresponding expansion of $E_r^{(\mathsf H)}(\delta,k)$  starts at $j=1$.

Imposing the conditions in~\eqref{eq:system_gen} on the moment integrals~\eqref{eq:mom_int_hyper} we arrive at the following linear systems:\begin{subequations}\begin{align}
\sum_{\ell=1}^{n} a_\ell^{(\mathsf H)}(\varkappa)A^{(\mathsf H)}_{j+\ell}(\varkappa) =&~ b_j^{(\mathsf H)}(\varkappa),\quad j\in\{1,2,\ldots,m_H-1\}\quad\text{and}\label{eq:lin_sym_H}\\
\sum_{\ell=1}^{n} a_\ell^{(\mathsf W)}(\varkappa)A^{(\mathsf W)}_{j+\ell}(\varkappa) =&~ b_j^{(\mathsf W)}(\varkappa),\quad j\in\{2,3,\ldots,m_W-1\},\label{eq:lin_sym_W}
\end{align}\label{eq:system_for_hyper}\end{subequations}
 for the coefficients $\big\{a_\ell^{(\mathsf H)}(\varkappa)\big\}_{\ell=1}^n$ and $\big\{a_\ell^{(\mathsf W)}(\varkappa)\big\}_{\ell=1}^n$ that define the regularizing functions $\sigma_2^{(\mathsf H)}$ and $\sigma_2^{(\mathsf W)}$.
Here, the linear-system coefficients are given by 
\begin{subequations}\begin{align}\label{eq:coeff_gen_formula_H}
A^{(\mathsf H)}_{j+\ell}(\varkappa) :=&~\frac{2}{\sqrt{\pi}} \int_0^\infty\left\{\cos(\varkappa t)+\varkappa t \sin (\varkappa t)\right\} \e^{-t^2}t^{2(j+\ell)-1}\de t\quad\text{and}\quad \\    
b^{(\mathsf H)}_{j}(\varkappa) :=&~{\rm p.f.}\! \int_0^\infty\left\{\cos(\varkappa t)+ \varkappa t \sin (\varkappa t)\right\}\left\{\erfc(t)+\frac{2}{\sqrt{\pi }} \e^{-t^2}\left(t+\frac{2}{3}t^3\right)\right\} t^{2(j-2)}\de t,
\end{align}\end{subequations}
for $j\in\{1,\ldots,m_H-1\}$ and $\ell\in\{1,\ldots,n\}$,
and 
\begin{subequations}\begin{align}\label{eq:coeff_gen_formula_W}
A^{(\mathsf W)}_{j+\ell}(\varkappa) :=&~\frac{2}{\sqrt{\pi}} \int_0^\infty\left\{((\varkappa t)^2-3)\cos(\varkappa t)-3\varkappa t\sin(\varkappa t)\right\} \e^{-t^2}t^{2(j+\ell)-1}\de t\quad\text{and}\quad \\    
b^{(\mathsf W)}_{j}(\varkappa) :=&~\int_0^\infty\left\{((\varkappa t)^2-3)\cos(\varkappa t)-3\varkappa t\sin(\varkappa t)\right\}\left\{\erfc(t)+\frac{2}{\sqrt{\pi }} \e^{-t^2}\left(t+\frac{2}{3}t^3\right)\right\} t^{2(j-2)}\de t,
\end{align}\end{subequations}
for $j\in\{2,\ldots,m_W-1\}$ and $\ell\in\{1,\ldots,n\}$.
The resulting regularization errors satisfy 
\begin{equation}\label{eq:reg_error_T}
\big|E_r^{(\mathsf H)}(\delta,k)\big| \lesssim \big|I^{(\mathsf H)}_{2,m_H}(\varkappa)\big| \delta^{2m_H-3}\quad\text{and}\quad\big|E_r^{(\mathsf W)}(\delta,k)\big| \lesssim \big|I^{(\mathsf W)}_{2,m_W}(\varkappa)\big| \delta^{2m_W-3},
\end{equation}
respectively.

The integrals in~\eqref{eq:coeff_gen_formula_H} and \eqref{eq:coeff_gen_formula_W} can further  be expressed in terms of the set of functions introduced in the Appendix~\ref{sec:coeff}. Indeed, applying the definitions in Propositions~\ref{prop:prop_1},~\ref{prop:prop_2}, and~\ref{prop:prop_3}, we readily obtain 
\begin{subequations}\begin{align}
A^{(\mathsf H)}_{j}(\varkappa) =&~C_{j-1}(\varkappa) + \varkappa S_{j}(\varkappa),\quad j\in\N_{\geq 2},\\
b^{(\mathsf H)}_{j}(\varkappa) =&~\widetilde C_{j-2}(\varkappa)+C_{j-2}(\varkappa)+\frac{2}{3}C_{j-1}(\varkappa)+\varkappa \left\{\widetilde S_{j-2}(\varkappa)+ S_{j-1}(\varkappa)+\frac{2}{3}S_{j}(\varkappa)\right\},\quad j\in\N.\label{eq:sing_rhs}
\end{align}
and
\begin{align}
A^{(\mathsf W)}_{j}(\varkappa) =&~\varkappa^2 C_{j}(\varkappa)-3[C_{j-1}(\varkappa) + \varkappa S_{j}(\varkappa)],\quad j\in\N_{\geq 3},\\
b^{(\mathsf W)}_{j}(\varkappa) =&~\varkappa^2\widetilde C_{j-1}(\varkappa)-3[\widetilde C_{j-2} (\varkappa)+ \varkappa \widetilde S_{j-2}(\varkappa)] + \varkappa^2 C_{j-1}(\varkappa)-3[C_{j-2}(\varkappa)+\varkappa S_{j-1}(\varkappa)]+\nonumber\\
&\frac{2}{3}\varkappa^2 C_{j}(\varkappa)-2[C_{j-1}(\varkappa)+\varkappa S_{j}(\varkappa)],\quad j\in\N_{\geq 2}.
\end{align}\label{eq:coeff_hyper}\end{subequations}
In the case of the operator $\mathsf{H}$, the negative-index coefficients $\widetilde{C}_{-1}(\varkappa)$, $C_{-1}(\varkappa)$, and $\widetilde S_{-1}(\varkappa)$---required in~\eqref{eq:sing_rhs}---are provided explicitly in Proposition~\ref{prop:prop_3} in terms of improper integrals.

\begin{remark} At $\varkappa=0$ the relations~\eqref{eq:coeff_hyper} simplify to
\begin{align*}
A^{(\mathsf H)}_{j}(0) =&~C_{j-1}(0),\quad j\in\N_{\geq 2}&b^{(\mathsf  H)}_{j}(0) =&~\widetilde C_{j-2}(0)+C_{j-2}(0)+\frac{2}{3}C_{j-1}(0),\quad j\in \N,\\
A^{(\mathsf W)}_{j}(0) =&-3C_{j-1}(0),\quad j\in\N_{\geq 3}&b^{(\mathsf W)}_{j}(0) =&-3[\widetilde C_{j-2}(0) +C_{j-2}(0)]-2C_{j-1}(0),\quad j\in\N_{\geq 2}.
\end{align*} 
Applying formulae~\eqref{eq:exact_coeff_0},~\eqref{eq:zero_erf}, and~\eqref{eq:sing_case}, we arrive at 
\begin{align*}
A^{(\mathsf H)}_{j}(0) =&\frac{(j-1)!}{\sqrt{\pi}},\quad j\in\N_{\geq 2}, & b_j^{(\mathsf H)}(0) =& \frac{4}{3\sqrt{\pi}}\frac{j!}{2j-3},\quad j\in\N,\\
A^{(\mathsf W)}_{j}(0) =&-\frac{3}{\sqrt\pi}(j-1)!,\quad j\in\N_{\geq 3},&b^{(\mathsf W)}_j(0)=&-\frac{4}{\sqrt{\pi}}\frac{j!}{2 j-3},\quad j\in\N_{\geq 2}.
\end{align*} 

Then, in the case $m_H = m_W = 5$, for which the regularization error scales as $\mathcal{O}(\delta^{7})$, we solve the corresponding linear systems~\eqref{eq:system_for_hyper} and obtain the following set of coefficients for the regularization functions:
$$
a^{(\mathsf H)}_1(0)=  -\frac{172}{5}, \qquad
a^{(\mathsf H)}_2(0)=  \frac{584}{15}, \qquad
a^{(\mathsf H)}_3(0)=  -\frac{464}{45}, \qquad
a^{(\mathsf H)}_4(0)=  \frac{32}{45}, 
$$
and
$$
a^{(\mathsf W)}_1(0)= \frac{124}{15},\qquad a^{(\mathsf W)}_2(0)= -\frac{56}{15},\qquad a^{(\mathsf W)}_3(0)= \frac{16}{45},
$$
which yield
$$
\sigma^{(\mathsf H)}_2(t) =  \erf(t)+\frac{2}{\sqrt{\pi}}\e^{-t^2}\left\{-t-\frac23t^3-\frac{172}{5}t^5+\frac{584}{15}t^7-\frac{464}{45}t^9+\frac{32}{45}t^{11}\right\}
$$
and
$$
\sigma^{(\mathsf W)}_2(t) =\erf(t)+\frac{2}{\sqrt{\pi}}\e^{-t^2}\left\{-t-\frac23t^3+\frac{124}{15}t^5-\frac{56}{15}t^7+\frac{16}{45}t^9\right\}.
$$
Unlike the previous cases, the present results constitute a significant extension of the work in~\cite{beale2025high}, where the Laplace hypersingular operator is not considered.

As in the case of the Laplace single- and double-layer operators, the matrices $\mathrm A^{(\mathsf H)}(0)\in\R^{(n+1)\times(n+1)}$ and $\mathrm A^{(\mathsf W)}(0)\in \R^{n\times n}$ for the coefficients of the operators making up $\mathsf T$ satisfy
$$
\operatorname{det} \mathrm{A}^{(\mathsf{H})}(0)=\frac{(n+1)!}{\pi^{(n+1) / 2}} \prod_{j=0}^{n}(j!)^2 \neq 0 
\quad\text{and}\quad
\operatorname{det} \mathrm{A}^{(\mathsf{W})}(0)=\frac{(-3)^n}{\pi^{n / 2}} \prod_{j=0}^{n-1} j!(j+2)!\neq 0.
$$
\end{remark}

\section{Quadrature and discretization error}\label{sec:discretization_error}

In this section, we present a numerical integration procedure for the evaluation of the regularized boundary integral operators considered in the previous sections. Since the surface integrands are smooth in all cases, we develop a composite quadrature rule for smooth integrals over the surface $\Gamma$. As the derivatives of the resulting integrands are controlled by the regularization parameter, as will be shown in this section, the quadrature discretization error depends on this parameter as well. We derive bounds for the  error and obtain expressions for the regularization parameter $\delta$ in terms of the discretization parameter $h$, to ensure convergence of the overall procedure and to balance the discretization and regularization errors, thereby guaranteeing high-order convergence. \rev{Throughout this section, $\qord\in\{0,\ldots,\mathfrak q+1\}$ denotes an integer smoothness index---the order of the H\"older/$C^{\qord}$ norms entering the quadrature error estimates.}

Following~\cite[Sec. 4.1.2]{Sauter2010}, we consider a family of surface meshes $\left\{ \mathcal{T}_h \right\}_{h>0}$, where each $\mathcal{T}_h$ provides a covering of $\Gamma$; i.e. 
$$
\Gamma=\bigcup_{T \in \mathcal{T}_h} T.
$$
Furthermore, each element $T \in \mathcal{T}_h$ is the image of a reference element $\widehat T$ under a smooth parametrization $F_T: \widehat{T} \rightarrow T \subset \Gamma$. We define the mesh size
$$
h:=\max _{T \in \mathcal{T}_h} \operatorname{diam}(T),
$$
where $\operatorname{diam}(T)$ denotes the diameter of the element $T$.
We assume that the family of tessellation $\left\{\mathcal{T}_h\right\}_{h>0}$ is shape-regular and quasi-uniform in the sense~\cite[Sec. 4.1.2]{Sauter2010}. 

A quadrature rule on $T$ is defined by the pushforward of a fixed quadrature rule on $\widehat{T}$.
Specifically,
$$
Q_T[\psi]:=\sum_{j=1}^{N_Q} w_j \psi\big(F_T\left(\widehat{x}_j\right)\big) J_T\left(\widehat{x}_j\right),
$$
where $\left\{\hat{x}_j, w_j\right\}_{j=1}^{N_Q}$ are the nodes and weights of a quadrature rule on $\widehat{T}$, and $J_T=\left|\p_{\hat{x}_1} F_T \times \p_{\hat{x}_2} F_T\right|$ is the surface Jacobian. A technical assumption given by~\cite[Assumption 5.3.5]{Sauter2010}) is also used; in the case of an affine mapping, this means that there exist constants $c, C>0$, independent of $h$ and $T$, such that
$c h^2 \leq J_T(\widehat{x}) \leq C h^2$ for all $\widehat{x} \in \widehat{T}.$

In what follows, we consider a quadrature rule with positive weights and degree of exactness $\mathfrak{q}$, that is,
$$
\sum_{j=1}^{N_Q} w_j P\left(\hat{x}_j\right)= \int_{\widehat{T}} P(\hat{x}) \de \hat{x}\quad\text{for all polynomials $P$ of total degree $\leq \mathfrak{q}.$}
$$
For such quadrature rule, the following local error estimate holds for integration over a single (curved) element $T$~\cite[Corollary 5.3.12]{Sauter2010}:
\begin{equation}\label{eq:loc_error_estimate}
\left|\int_{T} \psi\de s -Q_{ T}[\psi]\right|  \lesssim  h_T^{\qord+2}\|\psi\|_{C^{\qord}(T)}\lesssim  h_T^{\qord}\|\psi\|_{C^{\qord}(T)}\int_{T}\de s,\quad \qord\in\{0,\ldots,\mathfrak{q}+1\},
\end{equation}
where $\psi$ is assumed to be sufficiently smooth and $h_T=\operatorname{diam}(T)$, with the implicit constant depending on the quadrature rule and the element chart $F_T$.

Let $\Sigma \subset \Gamma$ be a tessellated surface patch, i.e., a union of surface elements $T \in \mathcal{T}_h$. Slightly abusing the notation, we define the composite quadrature rule over $\Sigma$ by
\begin{equation}\label{eq:composite_quadrature}
Q_{\Sigma}[\psi] :=\sum_{T\in\mathcal T_h:T\subset\Sigma} Q_{T}[\psi].
\end{equation}
Summing~\eqref{eq:loc_error_estimate} over all elements $T \subset \Sigma$ yields the  estimate
\begin{equation}\label{eq:global_estimate}
\left|\int_{\Sigma} \psi\de s -Q_{\Sigma}[\psi]\right|\lesssim  h^{\qord}\|\psi\|_{C^{\qord}(\Sigma)} \int_{\Sigma}\de s,\quad \qord\in\{0,\ldots,\mathfrak{q}+1\},
\end{equation}
with the implicit constant depending on the quadrature rule and the surface. The independence from $h>0$ follows from the shape-regularity assumption on the tessellation.

Let $x\in\Gamma$ and $k>0$ be fixed and consider 
\begin{equation}\label{eq:reg_integral_analysis}
f_\delta(y) = \frac{\Phi(k|x-y|)}{|x-y|^{2p+1}}\sigma_p\left(\frac{|x-y|}{\delta}\right)\Psi(x,y)\varphi(y),\quad y\in\Gamma,
\end{equation} the surface integrand in~\eqref{eq:reg_operator} arising from the regularization  
of the generic operator $\mathsf B[\varphi](x)$~\eqref{eq:BIOP},  for a smooth density $\varphi$ and a target point 
$x\in\Gamma$. The proposed method approximates $\mathsf B[\varphi](x)$ by the composite quadrature $Q_\Gamma\big[f_\delta\big]$~\eqref{eq:composite_quadrature}. The expressions for the functions $\Phi$, $\Psi$, and $\sigma_p$ associated with each of the boundary integral operators considered can be found in Sections~\ref{sec:single_layer}, \ref{sec:double_layer}, and~\ref{sec:hyper}, corresponding respectively to the single-layer operator, the double- and adjoint double-layer operators, and the two terms composing the hypersingular operator. 

The overall error of the proposed approach, after numerical integration, can be expressed as
\begin{equation}\label{eq:error}
E=\underbrace{{\rm f.p.}\int_{\Gamma}\left(f- f_\delta\right)\de s}_{\text {regularization error $E_r(\delta,k)$}}+\underbrace{\int_{\Gamma}f_\delta \de s-Q_\Gamma\big[f_\delta\big]}_{\text{discretization error $E_d(h,k)$}},
\end{equation}
where $f$ denotes the (singular) surface integrand corresponding to $\mathsf B[\varphi](x)$ in~\eqref{eq:BIOP}.

Borrowing the proof idea from~\cite{anderson2024fast}, to analyze the discretization error we decompose the surface integral over $\Gamma$ into local (near-field) and far-field contributions by writing
\begin{equation*}
\int_{\Gamma} f_\delta \de s=\sum_{j=1}^{N_\varrho}\int_{\Gamma_{j\varrho}(x)} f_\delta \de s
\end{equation*}
where, for $j=1$,
\begin{subequations}\begin{equation}\label{eq:1st_ring}
\Gamma_\varrho(x):=\bigcup_{T\in\mathcal T_h:T\cap B_\varrho(x)\neq\emptyset}T
\quad\text{and} \quad B_\varrho(x):=\{y\in\R^3:|x-y|<\varrho\},\quad \varrho\geq\delta\geq h,\end{equation}
and, for $j\geq 2$,
\begin{equation}\label{eq:other_rings}
\Gamma_{j\varrho}(x):=\left(\bigcup_{T\in\mathcal T_h:T\cap B_{j\varrho}(x)\neq\emptyset}T\right)\setminus \bigcup_{i=1}^{j-1}\Gamma_{i\varrho}(x).\end{equation}\end{subequations}
Here, $N_\varrho$ must be large enough to ensure $\Gamma\subset B_{N_\varrho\varrho}(x)$; see Figure~\ref{fig:rings}. Letting $d_\Gamma=\operatorname{diam}\Gamma$, this is satisfied by taking $N_\varrho = \lceil d_\Gamma/\varrho\rceil$. The parameter $\varrho>0$ is chosen appropriately for each operator considered.

\begin{figure}[htbp]
  \centering
  \includegraphics[width=0.7\textwidth]{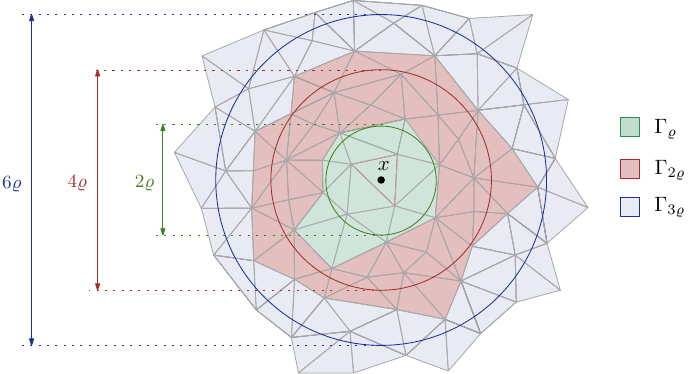} 
  \caption{Illustration of the decomposition of the surface $\Gamma$ into annular regions $\Gamma_{j\varrho}(x)$, $j=1,\dots,N_\varrho$, around a point $x\in\Gamma$. The innermost region $\Gamma_\varrho(x)$ defined in~\eqref{eq:1st_ring} consists of all mesh elements intersecting the ball $B_\varrho(x)$, and each subsequent ring $\Gamma_{j\varrho}(x)$, defined in~\eqref{eq:other_rings}, contains the elements intersecting $B_{j\varrho}(x)$ but not covered by the previous regions.}
  \label{fig:rings}
\end{figure}

Let $\qord \in \{0,\ldots,\mathfrak{q}+1\}$ be fixed for the time being. The discretization error for each operator can be analyzed within the same framework. We begin by expressing the regularized integrand~\eqref{eq:reg_integral_analysis} as
\begin{subequations}\label{eq:int_form_1}\begin{equation}
f_\delta(y)  := \frac{1}{\delta^{2(p-s)+1}}u_p\left(\frac{|x-y|}{\delta}\right)\Phi(k|x-y|)\widetilde\Psi(x,y)\varphi(y),
\end{equation}
where 
\begin{equation}
u_p(t) := \frac{\sigma_p\left(t\right)}{ t^{2(p-s)+1}}\quad\text{and}\quad \widetilde\Psi(x,y) := \frac{\Psi(x,y)}{|x-y|^{2s}}.
\end{equation}\end{subequations}
The exponent $s\in(0,p]$ is chosen as large as possible while ensuring that $\widetilde\Psi$ remains sufficiently smooth, at least of class $C^{\mathfrak q+1}(\Gamma\times\Gamma)$. In principle, $s$ may be selected depending on $\mathfrak q$, particularly for small values of $\mathfrak q$. For simplicity, however, and for general surfaces (with the exception of the sphere; see Appendix~\ref{sec:sphere}), we set $s=0$ for all operators except $\mathsf{H}$, for which we safely take $s=1$ (see Remark~\ref{rem:about_H}). In the sequel, we assume that the regularization parameter satisfies $\delta \in (0,1)$. 

By construction of the regularizing function $\sigma_p$, the function $u_p$ is a smooth (analytic) even function of $t\in\mathbb{R}$, and so is $\Phi$. From Proposition~\ref{lem:rad_funcs}, it is obtained  that both radial factors in~\eqref{eq:int_form_1} are smooth functions of $y\in\Gamma$, and therefore $f_\delta$ is a smooth function on $\Gamma$.
Indeed, from Proposition~\ref{lem:rad_funcs} and the chain rule, we get
$$
\|\Phi(k |x-\,\cdot\,|)\|_{C^{\qord}(\Gamma)}\lesssim \max\{1,k^{\qord}\}\|\Phi(|\,\cdot\,|)\|_{C^{\qord}(\overline{B_{kd_\Gamma}(0)})}\lesssim\max\{1,k^{\qord}\}\|\Phi\|_{C^{2\qord}([0,kd_\Gamma])},
$$
 where the implicit constant depends only on $\qord$ and $\Gamma$.   Here and in the sequel, the notation $\lesssim$ hides constants that may depend on $\qord$ and $\Gamma$, but are independent of $h$, $\delta$, and $k$. Note that the evenness of $\Phi$ is used here.

Moreover, from the definition of the specific function $\Phi$ associated with each operator (see~\eqref{eq:funcs_S}, \eqref{eq:funcs_double_layer}, and~\eqref{eq:funcs_H}), it can be easily shown that there exists a constant $C_{\qord}>0$ such that 
 $$
  \|\Phi\|_{C^{2\qord}([0,kd_\Gamma])}\leq C_{\qord} \left(\sum_{\ell=0}^p (kd_\Gamma)^\ell\right).
 $$

Similarly, from Proposition~\ref{lem:rad_funcs} together with the boundedness of $u_p$ and all its derivatives on $\mathbb{R}$, we get
$$
\|u_p\left(\tfrac{|x-\,\cdot\,|}{\delta}\right)\|_{C^{\qord}(\Gamma)}\lesssim \delta^{-\qord}\|u_p(|\,\cdot\,|)\|_{C^{\qord}(\overline{B_{\delta^{-1}d_\Gamma}(0)})}\lesssim \delta^{-\qord}\|u_p\|_{C^{2\qord}([0,\delta^{-1}d_\Gamma])}\lesssim \delta^{-\qord}\|u_p\|_{C^{2\qord}(\R)}.
$$

Therefore, letting 
\begin{align*} 
C_\Phi(\qord): =\max\{1,k^{\qord}\} \left(\sum_{\ell=0}^p (kd_\Gamma)^\ell\right)\quad\text{and}\quad C_\Psi(\qord):=\| \widetilde\Psi\|_{C^{\qord}(\Gamma\times\Gamma)}\end{align*}
we obtain 
\begin{equation}\label{eq:norm_bound_S}
\|f_\delta\|_{C^{\qord}(\Gamma)}\lesssim C_\Phi(\qord) C_\Psi(\qord)\|u_p\|_{C^{2\qord}(\mathbb{R})}
\|\varphi\|_{C^{\qord}(\Gamma)}.
\end{equation}

Combining~\eqref{eq:norm_bound_S} with the quadrature error estimate
\eqref{eq:global_estimate} for integration over $\Gamma_\varrho$ yields
\begin{equation} \label{eq:bound_inside_ball}\begin{split}
\left|
\int_{\Gamma_\varrho} f_\delta\de s- Q_{\Gamma_\varrho}[f_\delta] \right|
&\lesssim h^{\qord}\|f_\delta\|_{C^{\qord}(\Gamma)} \int_{\Gamma_\varrho}\de s\\
 &\lesssim C_\Phi(\qord) C_\Psi(\qord) \|u_p\|_{C^{2\qord}(\R)} \|\varphi\|_{C^{\qord}(\Gamma)} \varrho^2 h^{\qord}\delta^{-\qord-2(p-s)-1} ,
\end{split}\end{equation}
where we used that the surface area of $\Gamma_\varrho$ can be bounded by 
$\varrho^2$.

To estimate the discretization error associated with numerical integration over each
$\Gamma_{j\varrho}$ for $j\geq 2$, we write the integrand~\eqref{eq:reg_integral_analysis} as
$$
f_\delta(y)=\frac{1}{\delta^{2(p-s)+1}}\sigma_p\!\left(\frac{|x-y|}{\delta}\right) v_p\!\left(\frac{|x-y|}{\delta}\right)\Phi(k|x-y|)\widetilde\Psi(x,y)\varphi(y),
$$
where 
$$
 v_p(t) := \frac{1}{t^{2(p-s)+1}},\quad t\neq 0.
$$

Since the domains $\Gamma_{j\varrho}$ are uniformly separated from the target point $x$, we do not need to rely on Proposition~\ref{lem:rad_funcs} to estimate the norm of $f_\delta$ in this case. Indeed, bounding the $C^{\qord}$-norm of the radial factor by the corresponding norm of its radial profile and letting
\begin{equation*}
I_j:=\left[\frac{(j-1)\varrho}{\delta},\frac{d_\Gamma}{\delta}\right]\subset[1,\infty),\quad j\in\{2,\ldots,N_\varrho\},
\end{equation*}
 we obtain 
\begin{equation}\label{eq:proto_bound}\begin{split}
\|f_\delta\|_{C^{\qord}(\Gamma_{j\varrho})}&\lesssim C_\Phi(\qord) C_\Psi(\qord) \|\sigma_p\,v_p\|_{C^{\qord}(I_j)}\|\varphi\|_{C^{\qord}(\Gamma)}\delta^{-\qord-2(p-s)-1}.
\end{split}\end{equation}

 Note that we have used the fact that for all $y\in\Gamma_{j\varrho}$, $j\in\{2,\ldots,N_\varrho\}$, it holds $|x-y|\geq (j-1)\varrho$, so $|x-y|/\delta\geq (j-1)\varrho/\delta\in I_j$. In particular $|x-y|/\delta\geq 1$. Consequently, all the  derivatives of the mapping $y\mapsto |x-y|$ up to order $\qord$, which would appear in the bound above due to the chain rule,  are bounded on this set by constants depending only on $\qord$. The maximum of these constants is also absorbed into the implicit constant hidden in the symbol $\lesssim$ in~\eqref{eq:proto_bound}, which may depend on $\qord$ (and on the geometry of $\Gamma$), but is independent of $h$, $\delta$, and $k$.
 
We next use the inequality
\begin{equation}\label{eq:proto_bound_2}
\|\sigma_p\,v_p\|_{C^{\qord}(I_j)}\lesssim \|\sigma_p\|_{C^{\qord}(I_j)}\|v_p\|_{C^{\qord}(I_j)},\end{equation} 
with the constant depending only on $\qord$, and estimate $\|v_p\|_{C^{\qord}(I_j)}$. To do so, we note that for all $t\in I_j$ and $\ell\in\{0,\ldots,\qord\}$ it holds that  $|\frac{\de^\ell}{\de t^\ell} t^{-2(p-s)-1}|\leq \frac{(\ell+2(p-s)+1)!}{(2(p-s)+1)!}t^{-2(p-s)-1}$. It hence follows that  
\begin{equation}\label{eq:proto_bound_3}
\|v_p\|_{C^{\qord}(I_j)}=\max_{\ell\in\{0,\ldots,\qord\}}\|v_p\|_{C^{\ell}(I_j)}\leq\frac{(\qord+2(p-s)+1)!}{(2(p-s)+1)!}\left(\frac{(j-1)\varrho}{\delta}\right)^{-2(p-s)-1}.
\end{equation}
Using~\eqref{eq:proto_bound_2} and~\eqref{eq:proto_bound_3} to bound $\|\sigma_pv_p\|_{C^{\qord}(I_j)}$ in~\eqref{eq:proto_bound}, we arrive at
\begin{equation}\label{eq:bound_out}\begin{split}
\|f_\delta\|_{C^{\qord}(\Gamma_{j\varrho})}\lesssim C_\Phi(\qord) C_\Psi(\qord)\|\sigma_p\|_{C^{\qord}(\R)} \|\varphi\|_{C^{\qord}(\Gamma)} 
\left(\frac{(j-1)\varrho}{\delta}\right)^{-2(p-s)-1}\delta^{-\qord-2(p-s)-1}.
\end{split}\end{equation}

Applying the quadrature error estimate~\eqref{eq:global_estimate} on $\Gamma_{j\varrho}$ for $j\in\{2,\ldots,N_\varrho\}$ using the bound~\eqref{eq:bound_out}, we get
\begin{equation}\label{eq:bound_outside_ball}
\begin{split}
\left|\int_{\Gamma_{j\varrho}} f_\delta\de s -
Q_{\Gamma_{j\varrho}}\big[f_\delta\big]\right| 
&\lesssim C_\Phi(\qord) C_\Psi(\qord) \|\sigma_p\|_{C^{\qord}(\R)} \|\varphi\|_{C^{\qord}(\Gamma)} \frac{\delta^{2(p-s)+1}}{\varrho^{2(p-s)-1}}\frac{(2j-1)}{(j-1)^{2(p-s)+1}} h^{\qord}\delta^{-\qord-2(p-s)-1},
\end{split}
\end{equation}
where we used the fact that the area of $\Gamma_{j\varrho}$ is bounded by $(2j-1)\varrho^2$.

Therefore, combining~\eqref{eq:bound_inside_ball} and~\eqref{eq:bound_outside_ball} via the triangle inequality, it follows that the discretization error can be bounded as
\begin{equation}\label{eq:bound_discretization_error_proto}\begin{split}
\left|\int_{\Gamma} f_\delta\de s -
Q_{\Gamma}[f_\delta]\right|&\leq \sum_{j=1}^{N_\varrho}\left|\int_{\Gamma_{j\varrho}}f_\delta\de s-Q_{\Gamma_{j\varrho}}[f_\delta]\right|\\
 &\lesssim C_\sigma(\qord) C_\Phi(\qord) C_\Psi(\qord) \|\varphi\|_{C^{\qord}(\Gamma)} \left\{\varrho^2+\frac{\delta^{2(p-s)+1}}{\varrho^{2(p-s)-1}}\sum_{j=2}^{N_\varrho}
\frac{2j-1}{(j-1)^{2(p-s)+1}}
\right\}h^{\qord}\delta^{-\qord-2(p-s)-1},
\end{split}\end{equation}
 where  we have let
$$
C_\sigma(\qord) :=\max\{\|u_p\|_{C^{2\qord}(\R)},\|\sigma_p\|_{C^{\qord}(\R)}\}.
$$
Note that the sum in~\eqref{eq:bound_discretization_error_proto} can bounded as follows:
$$
\sum_{j=2}^{N_\varrho}
\frac{2j-1}{(j-1)^{2(p-s)+1}}\leq \begin{cases}\displaystyle 3\frac{d_\Gamma}{\varrho},&p=s,\\
6,& p>s.
 \end{cases}
$$
Since $\varrho$ is not a parameter of the method, it must be selected appropriately.
We choose $\varrho = \delta^{\varepsilon}$,
where the exponent $\varepsilon=\varepsilon(p,s)$ is chosen so as to balance the two terms inside the curly brackets above.
By doing so we get
$$
\varepsilon(p,s) = \begin{cases} \displaystyle\frac{1}{2}, &p=s,\\
1,& p>s,
\end{cases}
$$
from here we arrive at the following quadrature error estimate:
\begin{equation}\label{eq:bound_discretization_error}\begin{split}
\left|\int_{\Gamma} f_\delta\de s -
Q_{\Gamma}[f_\delta]\right|
 &\lesssim C_\sigma(\qord) C_\Phi(\qord) C_\Psi(\qord) \|\varphi\|_{C^{\qord}(\Gamma)} h^{\qord}\begin{cases}\delta^{-\qord},&p=s,\\
 \delta^{-\qord-2(p-s)+1},&p>s.\\
 \end{cases}
\end{split}\end{equation}

In practice, to achieve a discretization error independent of $\delta$, the regularization parameter is chosen as $\delta \propto h^{\mu}$ for $\mu=\mu(p,s)>0$. 
With this choice, we obtain the following estimate for the discretization error in~\eqref{eq:bound_discretization_error}:
$$
E_d(h,k) \lesssim C_\sigma(\qord) C_\Phi(\qord) C_\Psi(\qord) \|\varphi\|_{C^{\qord}(\Gamma)}\begin{cases}h^{\qord(1-\mu)},&p=s,\\
 h^{\qord(1-\mu)-2\mu(p-s)+\mu},&p>s.\\
 \end{cases}
$$
Clearly, the fastest decay of the discretization error as $h\to 0$ is obtained by selecting $\qord$ as large as possible, so we set $\qord=\mathfrak{q}+1$.

To balance the discretization error $E_d(h,k)$ with the regularization error $E_r(\delta,k)=\mathcal{O}(\delta^{\mathfrak{m}})$ as $\delta\to 0$, we must choose the exponent $\mu>0$ appropriately. Here the constant implicit in $\mathcal{O}(\delta^{\mathfrak{m}})$ depends on $\varkappa\in(0,k)$ and $\mathfrak{m}$ but is independent of $\delta$. The suitable choice, denoted $\mu^\star = \mu^\star(p,s)$, equates the orders of the two error contributions in~\eqref{eq:error} and yields an overall convergence rate of $\mathcal{O}(h^{\mathfrak o^\star})$, where
\begin{equation}\label{eq:convergence_orders}
\mathfrak o^\star(p,s) = \mu^\star(p,s)\,\mathfrak{m},\quad
\mu^\star(p,s) = \begin{cases}
\dfrac{\mathfrak{q}+1}{\mathfrak{q}+1+\mathfrak{m}}, & (p,s) = (0,0),\\[6pt]
\dfrac{\mathfrak{q}+1}{\mathfrak{q}+2(p-s)+\mathfrak{m}}, & (p,s)\in\{(1,0),(2,1),(2,0)\}.
\end{cases}
\end{equation}

\begin{figure}[htbp]
  \centering
  \includegraphics[width=0.24\textwidth]{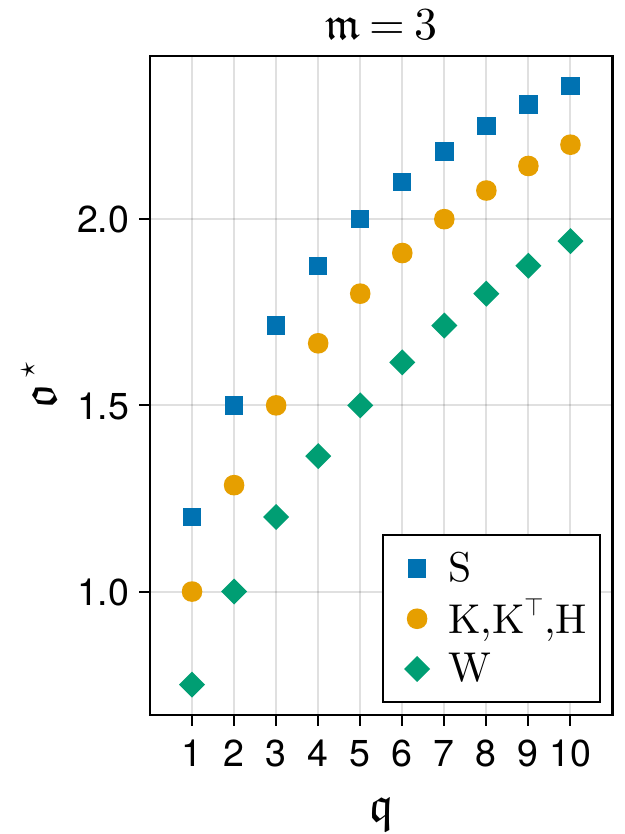} 
  \includegraphics[width=0.24\textwidth]{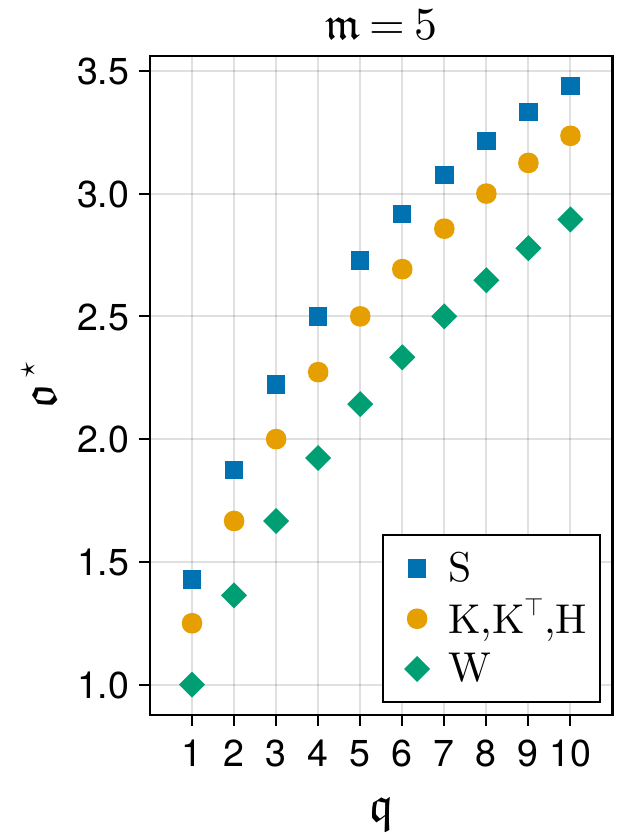} 
  \includegraphics[width=0.24\textwidth]{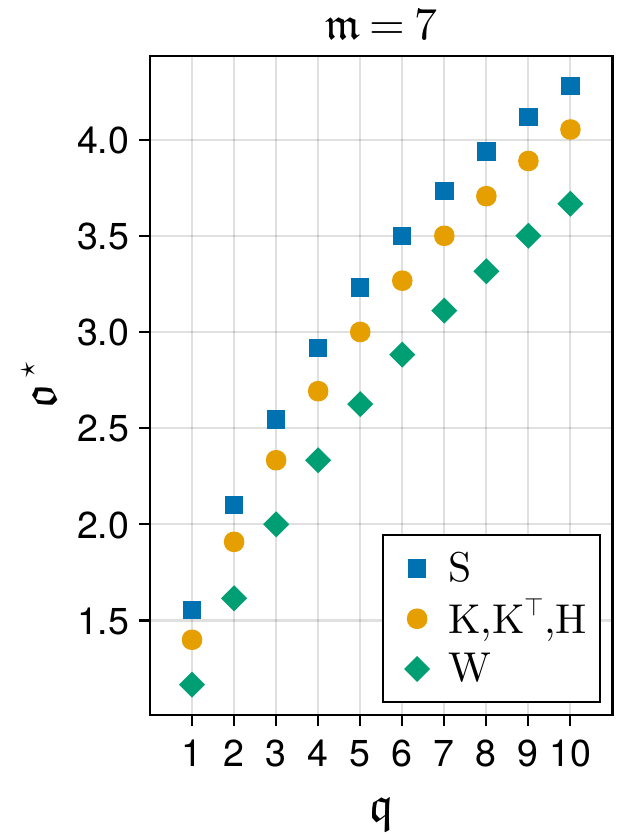} 
  \includegraphics[width=0.24\textwidth]{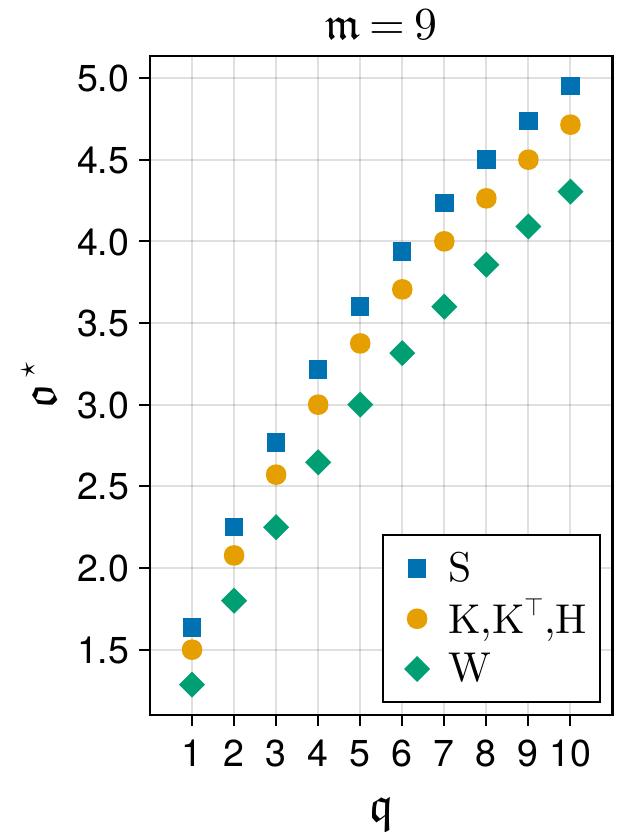}
  \caption{Convergence order $\mathfrak{o}^\star$ as given in~\eqref{eq:convergence_orders} for each integral operator considered in this work—$\mathsf S$, $\mathsf K$, $\mathsf K^\top$, $\mathsf H$, and $\mathsf W$—corresponding to $(p,s)=(0,0)$, $(1,0)$, $(1,0)$, $(2,1)$, and $(2,0)$, respectively, for various quadrature degrees of exactness $\mathfrak{q}$ and regularization orders $\mathfrak m$.}
  \label{fig:effective_orders}
\end{figure}

The following theorem summarizes the theoretical results of this section:
\begin{theorem}[Quadrature and discretization error]\label{thm:discretization_error}
Let $\Gamma$ be a smooth closed surface discretized by a shape-regular family of curved meshes $\{\mathcal{T}_h\}_{h>0}$ with mesh size $h$, and let $Q_\Gamma$ be a composite quadrature rule of degree of exactness $\mathfrak{q}$. Let $\varphi \in C^{\mathfrak{q}+1}(\Gamma)$, $x \in \Gamma$, $k > 0$, and $\mathsf{B} \in \{\mathsf{S}, \mathsf{K}, \mathsf{K}^\top, \mathsf{H}, \mathsf{W}\}$ with associated parameters $(p,s)$ as given in Table~\ref{tab:ops_params}. Let $f_\delta$ denote the regularized integrand for $\mathsf{B}[\varphi](x)$ with regularization parameter $\delta \in (0,1)$, and suppose the regularization error satisfies $E_r(\delta,k) = \mathcal{O}(\delta^{\mathfrak{m}})$ as $\delta \to 0$.

\begin{enumerate}
\item \textbf{Discretization error bound.} There exist constants $C_\sigma$, $C_\Phi$, $C_\Psi$, depending on $k$, $\Gamma$, and $\mathfrak{q}$ but independent of $h$ and $\delta$, such that
\begin{equation}\label{eq:disc_error_thm}
\left|\int_\Gamma f_\delta\,\de s - Q_\Gamma[f_\delta]\right| \lesssim C_\sigma C_\Phi C_\Psi \|\varphi\|_{C^{\mathfrak{q}+1}(\Gamma)}\, h^{\mathfrak{q}+1} \times \begin{cases} \delta^{-(\mathfrak{q}+1)}, & p = s, \\[4pt] \delta^{-(\mathfrak{q}+2(p-s))}, & p > s. \end{cases}
\end{equation}

\item \textbf{Optimal regularization parameter.} Setting $\delta \propto h^{\mu^\star}$ with
\begin{equation}\label{eq:optimal_mu}
\mu^\star(p,s) = \begin{cases} \dfrac{\mathfrak{q}+1}{\mathfrak{q}+1+\mathfrak{m}}, & p = s, \\[8pt] \dfrac{\mathfrak{q}+1}{\mathfrak{q}+2(p-s)+\mathfrak{m}}, & p > s, \end{cases}
\end{equation}
balances the discretization and regularization errors, yielding a total error $E = E_r + E_d = \mathcal{O}(h^{\mathfrak{o}^\star})$ with overall convergence rate
\begin{equation}\label{eq:optimal_order}
\mathfrak{o}^\star(p,s) = \mu^\star(p,s)\cdot\mathfrak{m}.
\end{equation}
\end{enumerate}

\begin{table}[h!]
\centering
\begin{tabular}{ccccc}
\hline
Operator & $p$ & $s$ & $\mu^\star$ & $\mathfrak{o}^\star$ \\
\hline
$\mathsf{S}$ & $0$ & $0$ & $\dfrac{\mathfrak{q}+1}{\mathfrak{q}+1+\mathfrak{m}}$ & $\dfrac{(\mathfrak{q}+1)\mathfrak{m}}{\mathfrak{q}+1+\mathfrak{m}}$ \\[8pt]
$\mathsf{K},\,\mathsf{K}^\top$ & $1$ & $0$ & $\dfrac{\mathfrak{q}+1}{\mathfrak{q}+2+\mathfrak{m}}$ & $\dfrac{(\mathfrak{q}+1)\mathfrak{m}}{\mathfrak{q}+2+\mathfrak{m}}$ \\[8pt]
$\mathsf{H}$ & $2$ & $1$ & $\dfrac{\mathfrak{q}+1}{\mathfrak{q}+2+\mathfrak{m}}$ & $\dfrac{(\mathfrak{q}+1)\mathfrak{m}}{\mathfrak{q}+2+\mathfrak{m}}$ \\[8pt]
$\mathsf{W}$ & $2$ & $0$ & $\dfrac{\mathfrak{q}+1}{\mathfrak{q}+4+\mathfrak{m}}$ & $\dfrac{(\mathfrak{q}+1)\mathfrak{m}}{\mathfrak{q}+4+\mathfrak{m}}$ \\[4pt]
\hline
\end{tabular}
\caption{Parameters $(p,s)$ and optimal exponents $\mu^\star$, $\mathfrak{o}^\star$ for each boundary integral operator. The convergence rate $\mathfrak{o}^\star$ is independent of the wavenumber $k$ and satisfies $\mathfrak{o}^\star \to \mathfrak{q}+1$ as $\mathfrak{m} \to \infty$ as well as $\mathfrak o^\star\to\mathfrak m$ as $\mathfrak{q}\to\infty$.}
\label{tab:ops_params}
\end{table}
\end{theorem}

\begin{remark}\label{rem:curvature}
\rev{The constants in the regularization error bound~\eqref{eq:bound_reg_error} and in Theorem~\ref{thm:discretization_error} depend on the target point $x$ through the derivatives of the local parametrization of $\Gamma$ and of $\widetilde\Psi$, which grow with the curvature of $\Gamma$ near $x$. The rates $\mathfrak o^\star$ are therefore independent of the geometry, but the values of $\delta$ and $h$ at which the asymptotic regime is reached are not, and they decrease where the surface is more strongly curved. Since the estimates hold at each target separately, they also remain valid when $\delta$ varies with $x$.}
\end{remark}

\begin{remark}
The analysis in this section applies to a specific surface discretization based on tessellation and composite interpolatory quadrature rules. Alternative approaches—such as global parametrizations or overlapping patch representations combined with \emph{spectrally accurate quadrature rules} (e.g., the trapezoidal rule), as in~\cite{beale2024extrapolated,beale2025high,beale2016simple,Bruno:2001ima,GaneshGraham2004}, or non-overlapping quad-patch representations with tensor-product Fejér-type quadrature rules, as in~\cite{perez2019harmonic,perez2019planewave,BrunoGarza:20}—can, in practice, achieve higher convergence rates than those predicted by Theorem~\ref{thm:discretization_error}.
\end{remark}

\section{Numerical results}\label{sec:numerics}
This section presents numerical examples validating the proposed high-order regularization procedures. All experiments are performed using the open-source Julia package \texttt{Inti.jl}. \rev{The experiments of Sections~\ref{sec:validation} and~\ref{sec:combined} are carried out on an ellipsoid represented exactly by the parametric patches of \texttt{Inti.jl} divided into curved triangles; the torus and bean-shaped surfaces of Section~\ref{sec:scattering} are discretized into curved triangular elements generated with Gmsh~\cite{geuzaine2009gmsh}, each parametrized by a high-order polynomial isoparametric map from a reference triangle.} Numerical integration over each element is carried out using the Vioreanu--Rokhlin quadrature rule~\cite{Vioreanu:14}, which places nodes strictly in the interior of each element. We consider degrees of exactness $\mathfrak{q}\in\{2,4,5\}$, though other high-order rules may equally be employed.

\subsection{Regularization error}\label{sec:validation}
Our first numerical example validates the regularized boundary integral operators 
constructed in Section~\ref{sec:general}. Specifically, we verify that the 
regularization error bound~\eqref{eq:bound_reg_error} is sharp for the operators 
$\mathsf{S}_\delta$, $\mathsf{K}_\delta$, $\mathsf{K}^\top_\delta$, and 
$\mathsf{T}_\delta$, constructed in Sections~\ref{sec:single_layer},
\ref{sec:double_layer}, and~\ref{sec:hyper}, respectively.

\rev{To this end we take $\Gamma$ to be the ellipsoid $x_1^2+x_2^2/0.8^2+x_3^2/0.6^2=1$, with semi-axes $(1,0.8,0.6)$, a standard test surface for kernel regularization methods~\cite{beale2016simple,beale2024extrapolated}. Its exact parametrization by six patches is subdivided into quadrilaterals, each split into two triangles in parameter space, so that there is no geometric error.
}

As input density we take a linear combination of spherical harmonics
\begin{equation}
    \varphi\rev{(x)} = \sum_{\ell=0}^{L}\sum_{n=-\ell}^{\ell} c_{\ell n}\, Y_{\ell}^n\rev{(x/|x|)},
    \label{eq:f_input}
\end{equation}
with $L = 5$ and coefficients  given by 
$$c_{\ell n} = \begin{cases}2^{-n+2},& \ell=n\in\{0,1,2,3,4,5\},\\
  0,&\text{ otherwise,}\end{cases}$$
\rev{where $Y_\ell^n$, with $\ell\in\N_0$ and $n\in\{-\ell,\ldots,\ell\}$, denote the spherical harmonics. Since no closed-form expression for $\mathsf B[\varphi]$ is available on the ellipsoid, the reference values $\mathsf B[\varphi](x)$ are computed by adaptive cubature (the adaptive quadrature available in \texttt{Inti.jl}, provided by \texttt{HAdaptiveIntegration.jl}~\cite{faria2026hadaptive}) over the same exactly parametrized triangles, applied to the kernel regularized with $\mathfrak m=9$ and $\delta'=0.02$, whose regularization error $\mathcal O(\delta'^9)$ is negligible. Indeed, the change in the reference values when $\delta'=0.03$ is used instead is below $10^{-9}$ in relative terms for all four operators and both wavenumbers considered. The reference values are computed at $N_E=200$ randomly chosen points of $\Gamma$. Since the geometry is exact on every mesh, they do not depend on the mesh and are reused in Section~\ref{sec:combined}.}

We measure the approximate relative $L^2$ error in the operator evaluation via
\begin{equation}
e^{(\mathsf B)}(\delta,h,k):=\frac{\displaystyle\sqrt{\sum_{j=1}^{\rev{N_E}} \bigl|\mathsf{B}[\varphi](x_j) - 
\widetilde{\mathsf{B}}_{\delta}[\varphi](x_j)\bigr|^2}}
{\displaystyle\sqrt{\sum_{j=1}^{\rev{N_E}} \bigl|\mathsf{B}[\varphi](x_j)\bigr|^2}},
\label{eq:rel_error}
\end{equation}
where \rev{$\{x_j\}_{j=1}^{N_E}$ are evaluation points on $\Gamma$, specified with each experiment, and $\widetilde{\mathsf{B}}_{\delta}$ is the quadrature approximation of the regularized operator $\mathsf{B}_\delta$.}

Since the regularization error bound~\eqref{eq:bound_reg_error} derived in Section~\ref{sec:general} takes the form $|I^{(\mathsf{B})}_{p,m}(\varkappa)|\delta^{\mathfrak m}$, where $\varkappa = \delta k$ is the scaled wavenumber and $\mathfrak{m} = 2(m-p)+1$ is the regularization order, we introduce the normalized error measure
\begin{equation}\label{eq:rel_norm_error}
e^{(\mathsf{B})}_{\rm nrm}(\delta,h,k) = \frac{e^{(\mathsf B)}(\delta,h,k)}{|I^{(\mathsf{B})}_{p,m}(\varkappa)|}.
\end{equation}
Dividing by $|I^{(\mathsf{B})}_{p,m}(\varkappa)|$ removes the $\varkappa$-dependent prefactor and isolates the algebraic factor $\delta^{\mathfrak{m}}$, making the convergence rate more clearly visible in log-log plots. The normalization constant $I^{(\mathsf{B})}_{p,m}(\varkappa)$ can be computed numerically via the recurrence relations or series expansions in Appendix~\ref{sec:coeff}, or alternatively by applying an adaptive quadrature rule to near machine-precision accuracy.

\rev{Figure~\ref{fig:delta_convergence_ellipsoid} displays the normalized errors~\eqref{eq:rel_norm_error} for regularization orders $\mathfrak m\in\{3,5,7,9\}$ and wavenumbers $k\in\{0,\pi\}$, on a mesh with $N_Q=115200$ quadrature nodes, largest element edge $h=0.064$, quadrature degree of exactness $\mathfrak q=4$, and $\delta\in[h,1]$, with the evaluation points taken as $N_E=200$ quadrature nodes chosen at random. As can be seen in the figure, in the range where the regularization error dominates the observed error~\eqref{eq:rel_norm_error} is in good agreement with the predicted rate $\mathcal O(\delta^{\mathfrak m})$ of~\eqref{eq:reg_error_S}, \eqref{eq:reg_error_K} and~\eqref{eq:reg_error_T}, for both wavenumbers. The oscillations at $k=\pi$ come from the normalization by $|I^{(\mathsf B)}_{p,m}(\varkappa)|$, and the quadrature error becomes visible for the smallest values of $\delta$ when $\mathfrak m\geq7$.}

\begin{figure}[htbp]
  \centering\color{revffcolor}
  \includegraphics[width=0.24\textwidth]{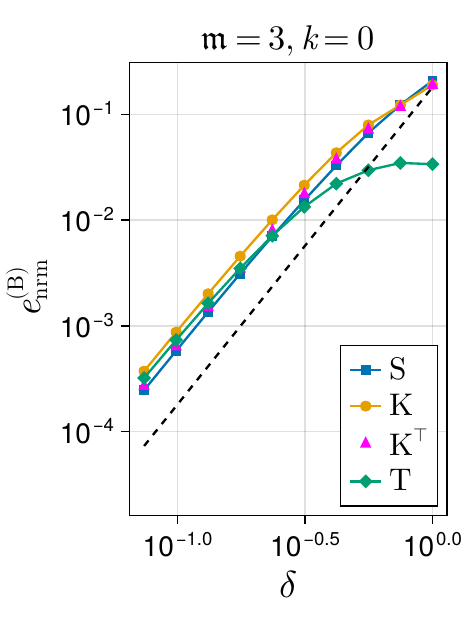}
  \includegraphics[width=0.24\textwidth]{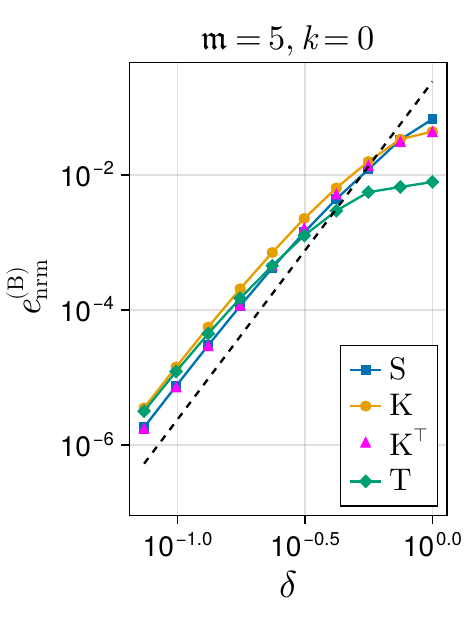}
  \includegraphics[width=0.24\textwidth]{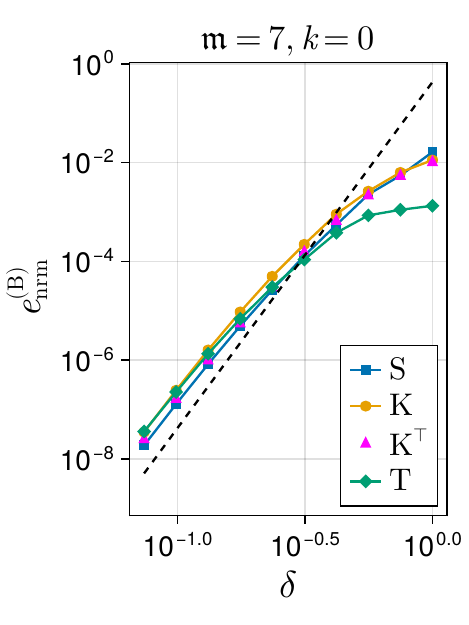}
  \includegraphics[width=0.24\textwidth]{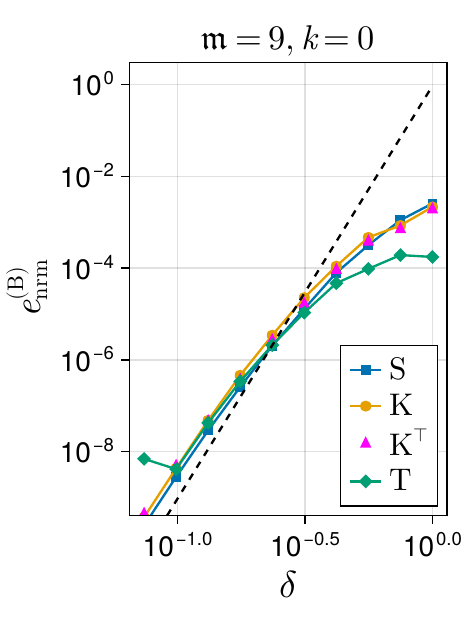}\\
  \includegraphics[width=0.24\textwidth]{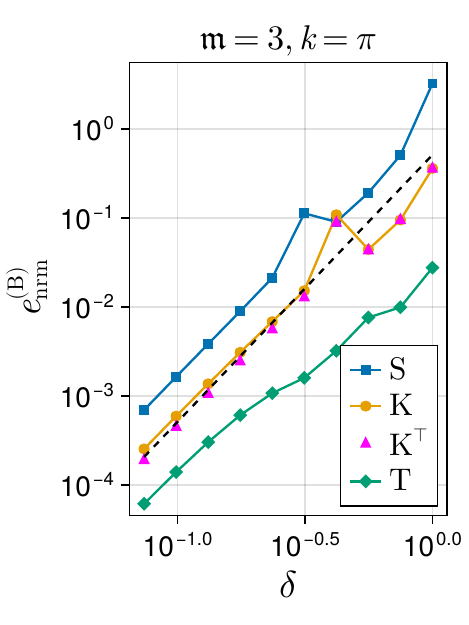}
  \includegraphics[width=0.24\textwidth]{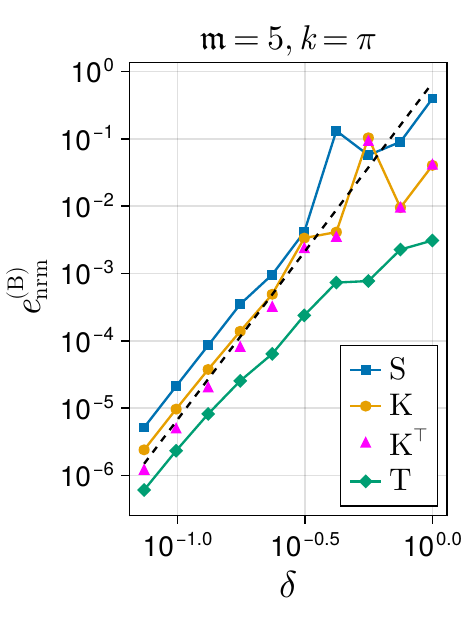}
  \includegraphics[width=0.24\textwidth]{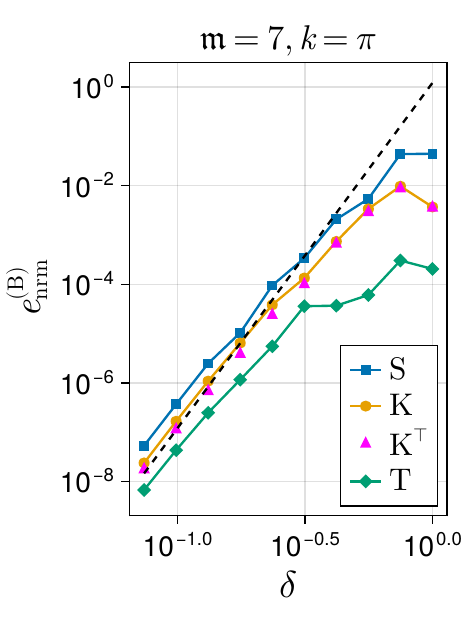}
  \includegraphics[width=0.24\textwidth]{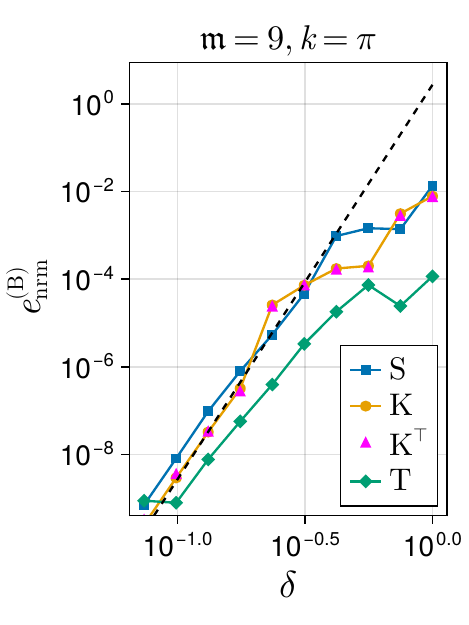}
  \caption{\rev{Normalized error $e^{(\mathsf{B})}_{\rm nrm}$ for $\mathsf B\in\{\mathsf S,\mathsf K,\mathsf K^\top,\mathsf T\}$ defined in~\eqref{eq:rel_norm_error} as a function of the regularization parameter $\delta$, for regularization orders $\mathfrak{m}\in\{3,5,7,9\}$ (left to right), on the ellipsoid with semi-axes $(1,0.8,0.6)$ (exact geometry, $h=0.064$, $\mathfrak q=4$), with wavenumber $k=0$ (top row) and $k=\pi$ (bottom row). The dashed black line shows the $\mathcal O(\delta^{\mathfrak m})$ rate, anchored at $\delta\approx0.1$.}}
  \label{fig:delta_convergence_ellipsoid}
\end{figure}

\subsection{Combined regularization and discretization error}\label{sec:combined}
 \rev{
We begin this section by performing a numerical experiment to isolate the quadrature error $E_d(h,k)$ in~\eqref{eq:error}. For fixed $\delta$ the regularized integrand is smooth, so that $E_d=\mathcal O(h^{\mathfrak q+1})$ is what a convergent quadrature rule of degree of exactness $\mathfrak q$ produces. The total error, on the other hand, then plateaus at the regularization error $E_r(\delta,k)$, which does not decrease with $h$ and acts as a floor. What the estimate~\eqref{eq:bound_discretization_error} provides is how the ``constant" multiplying $h^{\mathfrak q+1}$ depends on $\delta$. Indeed,  with $\qord=\mathfrak q+1$, this estimate predicts
\begin{equation}\label{eq:delta_scaling_prediction}
\frac{E_d(h,k)}{h^{\mathfrak q+1}}\;\lesssim\;
\begin{cases}
\delta^{-(\mathfrak q+1)}, & p=s,\\[2pt]
\delta^{-(\mathfrak q+1)-2(p-s)+1}, & p>s,
\end{cases}
\end{equation}
that is, using the parameters $(p,s)$ of Table~\ref{tab:ops_params}, an exponent $-(\mathfrak q+1)$ for $\mathsf S$, $-(\mathfrak q+2)$ for $\mathsf K$, $\mathsf K^\top$ and $\mathsf H$, and $-(\mathfrak q+4)$ for $\mathsf W$. These are upper bounds, and whether the $\delta$-dependence is attained determines whether the coupling $\delta\propto h^{\mu^\star}$ is the appropriate one.

To test~\eqref{eq:delta_scaling_prediction} we hence take $\mathfrak m=5$, $\mathfrak q\in\{2,4,5\}$, $k\in\{0,\pi\}$ and $\delta\in[0.05,1]$, and compare, at $100$ points $x\in\Gamma$ chosen at random on the ellipsoid, the quadrature approximation $\widetilde{\mathsf B}_\delta[\varphi](x)$\rev{, with $\varphi$ the density~\eqref{eq:f_input},} with the integral $\mathsf B_\delta[\varphi](x)$ of the same regularized kernel over $\Gamma$. Since the exact triangles cover $\Gamma$ without geometric error, this reference value is the same for every mesh and is computed only once, by adaptive cubature on the triangles near the target point $x$ and a rule of degree $17$ elsewhere. To further check its accuracy we repeat the computation on a second, unrelated triangulation of $\Gamma$ with the two results agreeing to $2\cdot10^{-13}$ (relative). The difference $\mathsf B_\delta[\varphi](x)-\widetilde{\mathsf B}_\delta[\varphi](x)$ is the quadrature error at $x\in\Gamma$, in principle free of regularization error since both terms involve the same regularized kernel. We measure it by the relative error $e^{(\mathsf B)}_d(\delta,h,k)$, defined as~\eqref{eq:rel_error} with $\mathsf B[\varphi]$ replaced by $\mathsf B_\delta[\varphi]$ and the $100$ points as evaluation points. Since $\mathsf B_\delta[\varphi]$ depends only weakly on $\delta$, $e^{(\mathsf B)}_d$ obeys the same scaling~\eqref{eq:delta_scaling_prediction}. We use meshes with  $h\in\{0.039,0.027,0.019\}$. 

Figure~\ref{fig:delta_quad_error_ellipsoid} then displays $e^{(\mathsf B)}_d/h^{\mathfrak q+1}$ as a function of $\delta$, for the points with $h/\delta\leq1/2$. To quantify the slopes we fit $e^{(\mathsf B)}_d\approx C h^{\gamma}\delta^{-\alpha}$ at $k=0$ over all three meshes and $\delta\leq0.4$. For $\mathsf S$ this gives $\alpha=2.7$ and $4.9$ for $\mathfrak q=2$ and $4$, close to the predicted $\mathfrak q+1$, so the $\delta$-dependence of the bound is essentially attained. For $\mathsf K$ and $\mathsf K^\top$ it gives $\alpha=2.9$ and $4.8$, close to the values for $\mathsf S$ and about one power milder than the predicted $\mathfrak q+2$. This is consistent with $\nu(y)\cdot(x-y)=\mathcal O(|x-y|^2)$ on a smooth surface, which makes the kernel of $\mathsf K$ effectively no more singular than that of $\mathsf S$. It is worth noting here that the analysis of Section~\ref{sec:discretization_error} cannot use this, since $\nu(y)\cdot(x-y)/|x-y|^2$ is bounded but not smooth, so that $s=0$ is the admissible choice there (except on the sphere, Appendix~\ref{sec:sphere}). For $\mathsf T$, $\alpha=5.0$, $6.7$ and $7.7$ for $\mathfrak q=2$, $4$ and $5$ lie between the exponents predicted for its $\mathsf H$ part ($\mathfrak q+2$) and its $\mathsf W$ part ($\mathfrak q+4$). (For $\mathfrak q=5$ the errors of $\mathsf S$, $\mathsf K$ and $\mathsf K^\top$ are below the accuracy of the reference values for almost all $\delta$ with $h/\delta\leq1/2$.) In $h$, errors decay closer to $h^{\mathfrak q+2}$ rather than $h^{\mathfrak q+1}$ so the curves for the three meshes in Figure~\ref{fig:delta_quad_error_ellipsoid} do not coincide. In summary, the measured quadrature errors are consistent with~\eqref{eq:delta_scaling_prediction}, give no indication that the $\delta$-dependence of $\mathsf S$ is milder than predicted, and show that for $\mathsf K$, $\mathsf K^\top$ and $\mathsf T$ the coupling $\delta\propto h^{\mu^\star}$ derived from the bound might be conservative, in agreement with the faster than predicted convergence of $\mathsf T$ reported below. At $k=\pi$ (bottom row of Figure~\ref{fig:delta_quad_error_ellipsoid}) the curves follow the same slopes for small $\delta$, but at larger $\delta$ they show local plateaus followed by drops, caused by the variation of the coefficients $a_\ell(\varkappa)$ of the regularizing functions with $\varkappa=k\delta$. 
}

\begin{figure}[htbp]
  \centering\color{revffcolor}
  \includegraphics[width=0.31\textwidth]{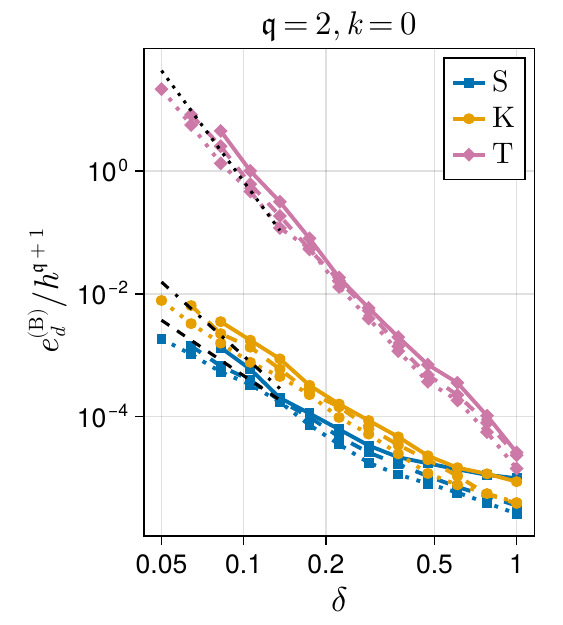}
  \includegraphics[width=0.31\textwidth]{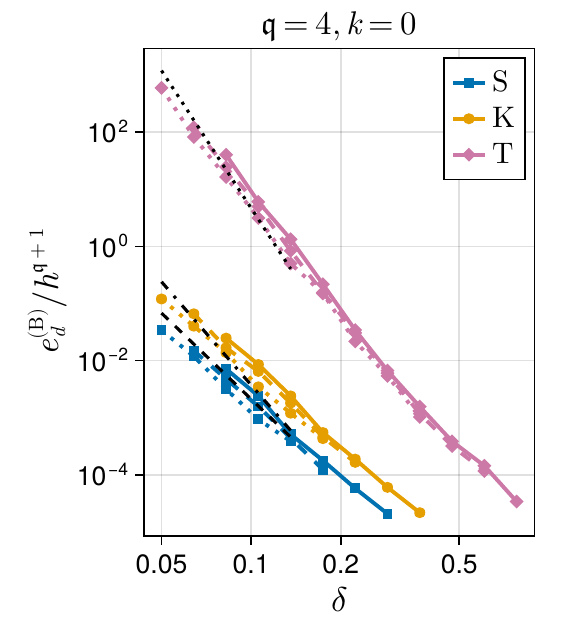}
  \includegraphics[width=0.31\textwidth]{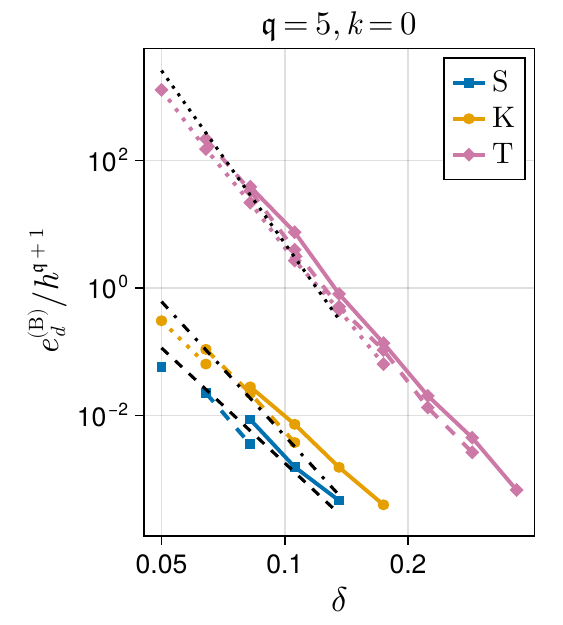}\\
  \includegraphics[width=0.31\textwidth]{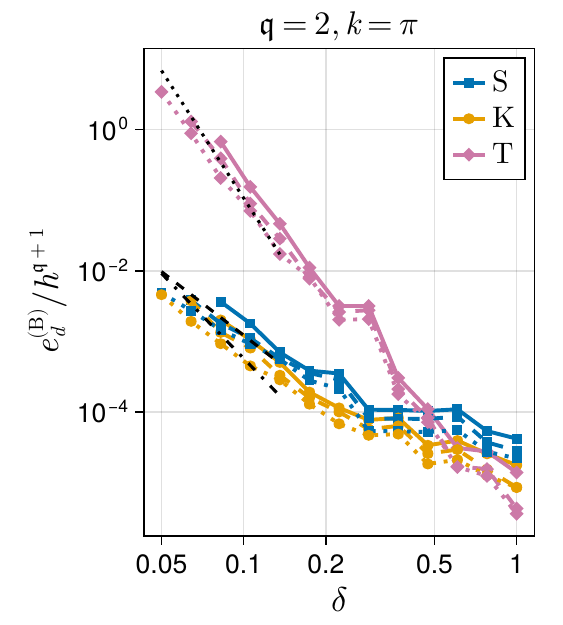}
  \includegraphics[width=0.31\textwidth]{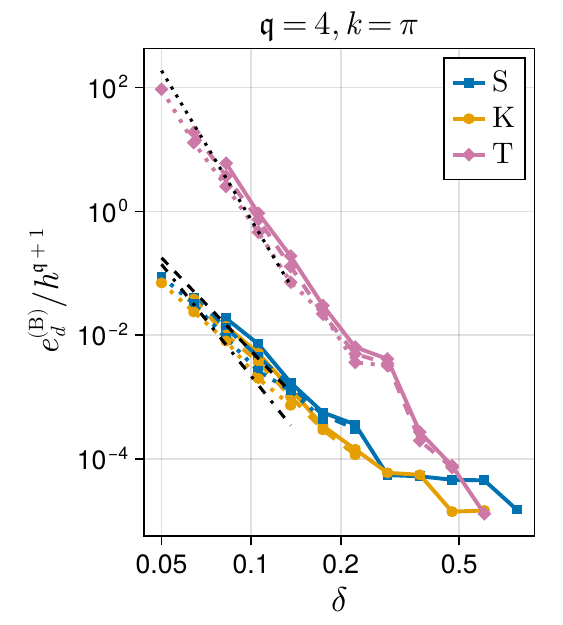}
  \includegraphics[width=0.31\textwidth]{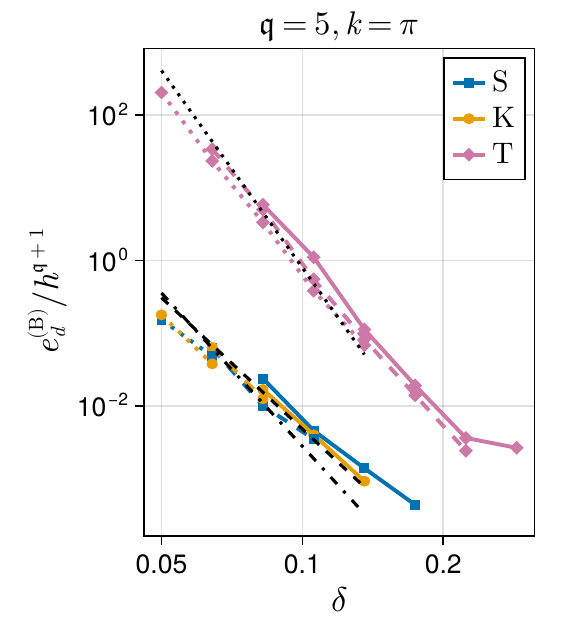}
  \caption{\rev{Relative quadrature error $e^{(\mathsf B)}_d$ alone, normalized as $e^{(\mathsf B)}_d/h^{\mathfrak q+1}$, as a function of $\delta$ for $\mathsf S$, $\mathsf K$ and $\mathsf T$ on the ellipsoid (exact geometry) with $\mathfrak m=5$, for $\mathfrak q\in\{2,4,5\}$ (left to right) and $k=0$ (top row) and $k=\pi$ (bottom row), on meshes with $h=0.039$ (solid), $0.027$ (dashed) and $0.019$ (dotted); only points with $h/\delta\leq1/2$ and above the accuracy of the reference values are shown, and $\mathsf K^\top$ coincides with $\mathsf K$ to plotting accuracy. Black lines: slopes $\delta^{-(\mathfrak q+1)}$ (dashed), $\delta^{-(\mathfrak q+2)}$ (dash-dotted) and $\delta^{-(\mathfrak q+4)}$ (dotted) of~\eqref{eq:delta_scaling_prediction}.}}
  \label{fig:delta_quad_error_ellipsoid}
\end{figure}

In order to further validate the combined regularization and discretization error analysis of Section~\ref{sec:discretization_error}, we use the same setup as in the previous example and examine the relative error $e^{(\mathsf{B})}$ defined in~\eqref{eq:rel_error} for various mesh sizes $h$, with the regularization parameter set to $\delta = c\, h^{\mu^\star}$, where $\mu^\star = \mu^\star(p,s)$ is given by~\eqref{eq:convergence_orders} \rev{with the general-surface values of Table~\ref{tab:ops_params}} and $c > 0$ is a fixed constant in each numerical experiment.
\rev{Note that the theory of Section~\ref{sec:discretization_error} fixes only the exponent $\mu^\star$ in $\delta=c\,h^{\mu^\star}$, so the prefactor $c$ sets the level of the error curve, not its slope. Two requirements constrain it. The analysis of the quadrature error assumes that the mesh resolves the regularized kernel, $\delta\gtrsim h$ (the condition $\varrho\geq\delta\geq h$ in~\eqref{eq:1st_ring}). If $c$ is too small this fails and the error deteriorates, as in the left part of the curves of Figure~\ref{fig:delta_convergence_ellipsoid}. If $c$ is too large, the regularization error, which grows like $c^{\mathfrak m}h^{\mu^\star\mathfrak m}$, moves the whole curve upwards. In the convergence studies that follow, a rule we have found robust is to fix $c$ by requiring $\delta=\lambda_\delta h_{\max}$ on the coarsest mesh $h_{\max}$ of the sequence, with $\lambda_\delta$ of order one, i.e.\ $c=\lambda_\delta h_{\max}^{1-\mu^\star}$, as in~\cite[Sect.~4.2]{beale2025high}. On all meshes $\delta=c\,h^{\mu^\star}$ then follows the scaling prescribed by the theory, and since $\mu^\star<1$ the ratio $\delta/h$ grows under refinement, so that the resolution condition holds throughout.}

\rev{A sweep of $\lambda_\delta$ from $0.19$ to $3$ on two fixed meshes ($h\approx0.098$ and $h\approx0.049$), for $(\mathfrak m,\mathfrak q)\in\{(5,4),(7,5)\}$ and $k\in\{0,\pi\}$, confirms this picture. Indeed, the observed error is V-shaped, with its minimum at $\lambda_\delta\approx0.5$--$0.65$, and grows like $\lambda_\delta^{\mathfrak m}$ to the right of the minimum but much more rapidly to its left. Overestimating $c$ therefore costs a predictable factor, whereas underestimating it can cost orders of magnitude, so one should err on the large side. In general, the optimal value depends on the geometry, mainly through the curvature (see the end of this section), and for an unseen geometry a short sweep in $\lambda_\delta$ on a coarse mesh is a practical way of selecting it. Accordingly, in the experiments that follow we take $\lambda_\delta=0.75$ ($\lambda_\delta=0.95$ for $\mathsf S$, $\mathsf K$, $\mathsf K^\top$ and $\lambda_\delta=1.05$ for $\mathsf H$, $\mathsf W$ when $\mathfrak m=9$), slightly to the right of the minimum, where the errors are within a factor $2$--$10$ of it.
 Both larger and smaller values degrade the observed rates.}

\rev{Figure~\ref{fig:h_convergence_k_0_ellipsoid} displays the relative error~\eqref{eq:rel_error} for the Laplace case $k=0$ on ten meshes with $0.049\leq h\leq0.196$, using the reference values described in Section~\ref{sec:validation}. For $\mathfrak m\leq7$ the observed slopes are in good agreement with the predicted rates.
For $\mathfrak m=9$ they are lower, since $\delta$ then remains mostly above the range in which the regularization error follows $\delta^{\mathfrak m}$ (Figure~\ref{fig:delta_convergence_ellipsoid}).}

\begin{figure}[htbp]
  \centering\color{revffcolor}
  \includegraphics[width=0.24\textwidth]{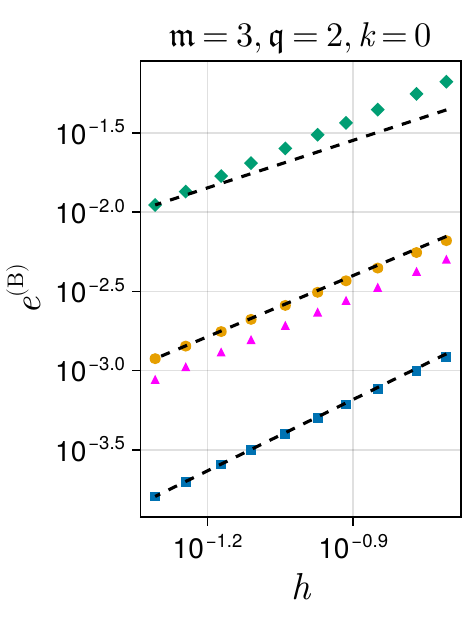}
  \includegraphics[width=0.24\textwidth]{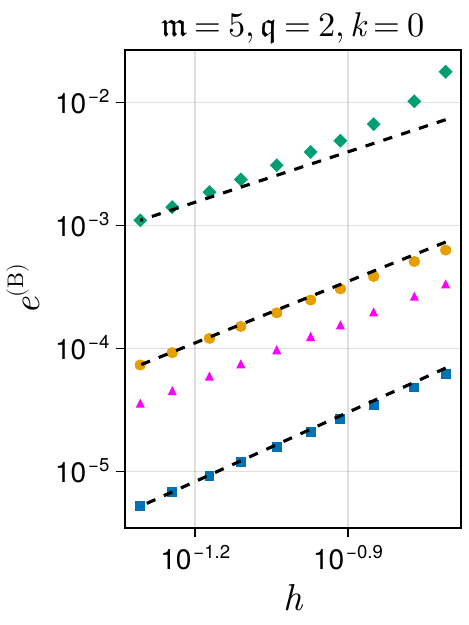}
  \includegraphics[width=0.24\textwidth]{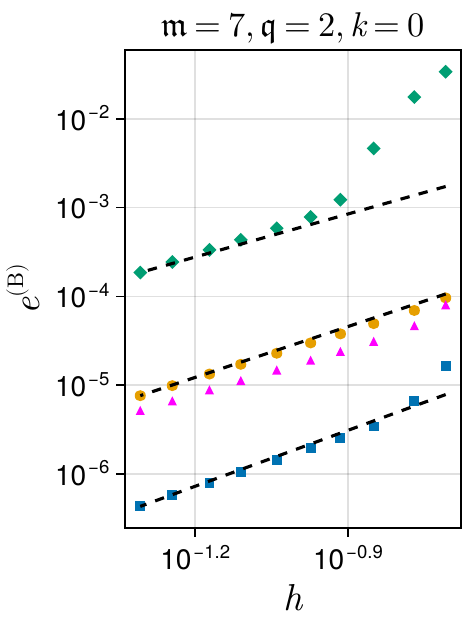}
  \includegraphics[width=0.24\textwidth]{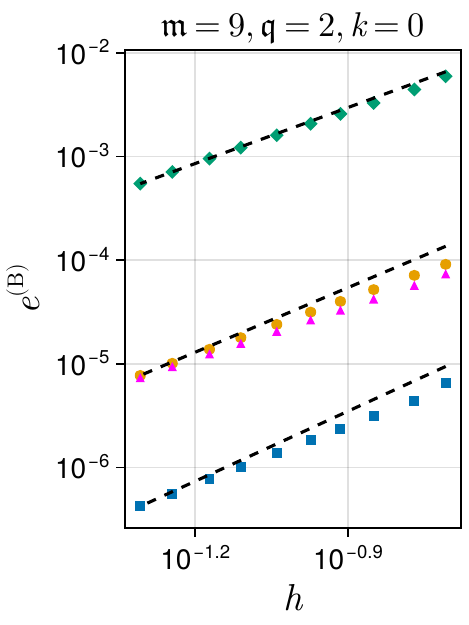}\\
  \includegraphics[width=0.24\textwidth]{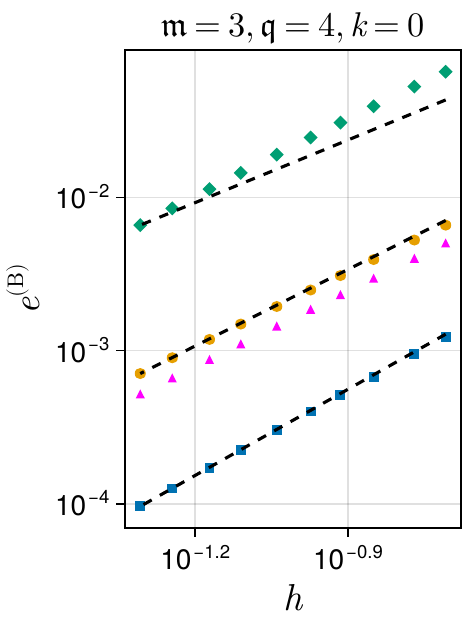}
  \includegraphics[width=0.24\textwidth]{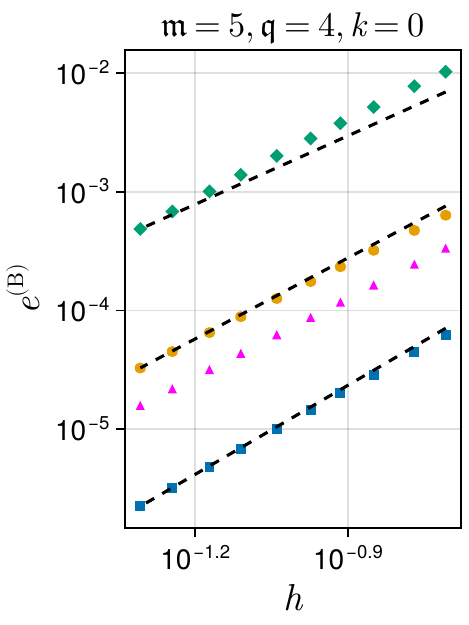}
  \includegraphics[width=0.24\textwidth]{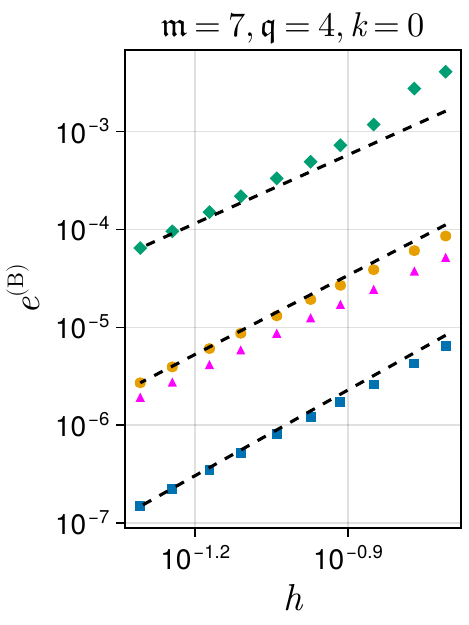}
  \includegraphics[width=0.24\textwidth]{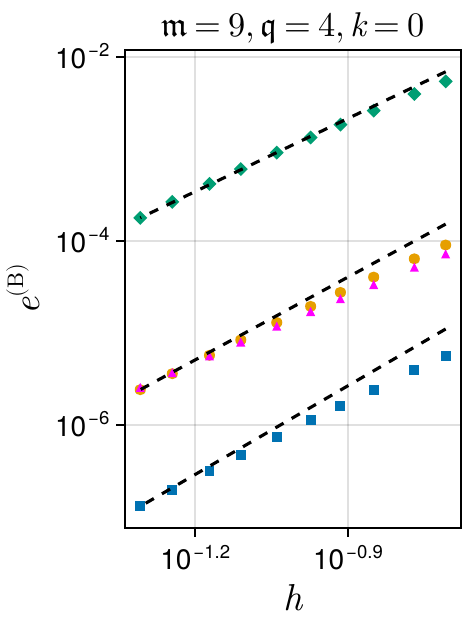}\\
  \includegraphics[width=0.24\textwidth]{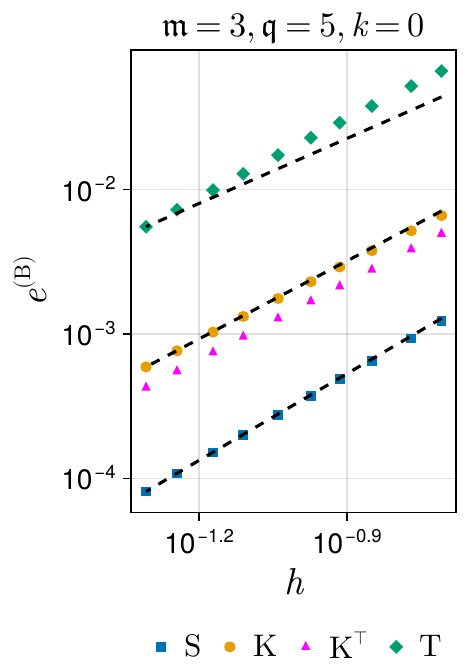}
  \includegraphics[width=0.24\textwidth]{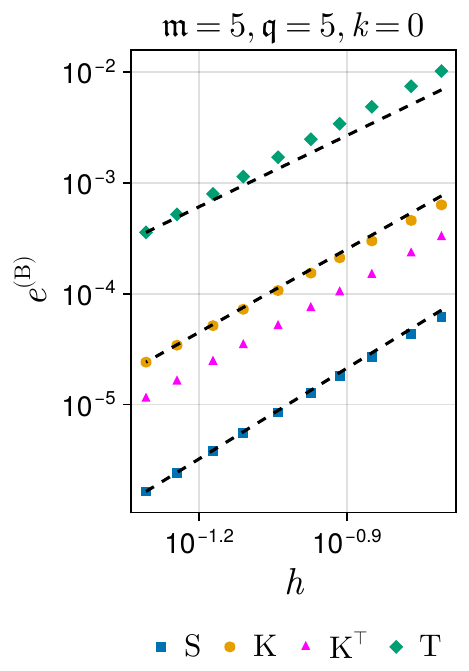}
  \includegraphics[width=0.24\textwidth]{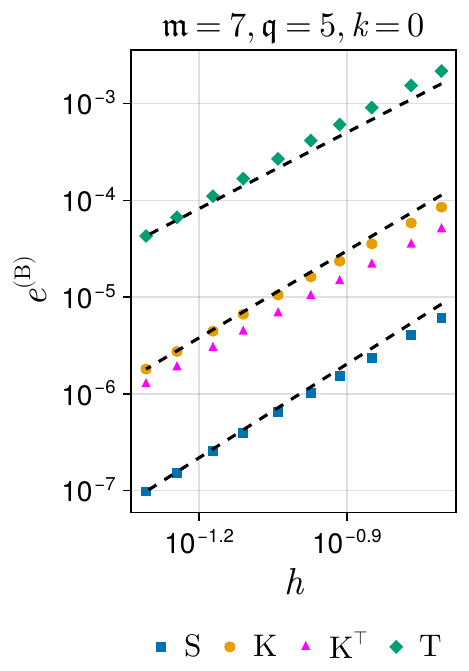}
  \includegraphics[width=0.24\textwidth]{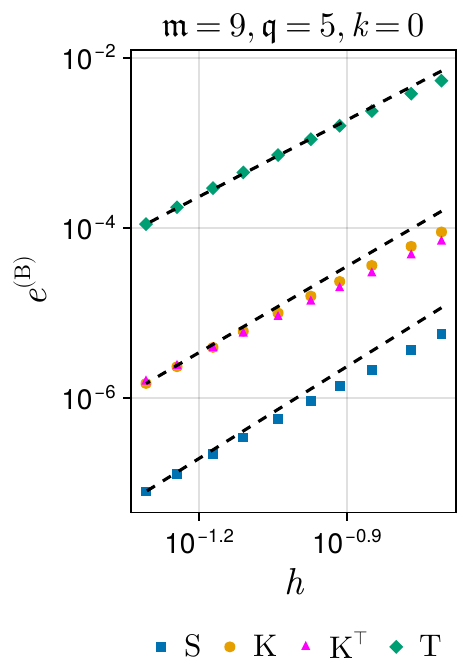}
  \caption{\rev{Overall relative error $e^{(\mathsf{B})}$ for $\mathsf B\in\{\mathsf S,\mathsf K,\mathsf K^\top,\mathsf T\}$ defined in~\eqref{eq:rel_error} as a function of the mesh size $h$ on the ellipsoid with semi-axes $(1,0.8,0.6)$, for regularization orders $\mathfrak{m}\in\{3,5,7,9\}$ (left to right) and quadrature rules of degree of exactness $\mathfrak{q}\in\{2,4,5\}$ (top to bottom), with wavenumber $k=0$ and $\delta=c\,h^{\mu^\star}$ ($\lambda_\delta=0.75$). The dashed black lines show the theoretical $\mathcal{O}(h^{\mathfrak{o}^\star})$ convergence rates of Table~\ref{tab:ops_params} for $\mathsf S$, $\mathsf K$ ($\mathsf K^\top$) and the $\mathsf W$ part of $\mathsf T$.}}
  \label{fig:h_convergence_k_0_ellipsoid}
\end{figure}

For non-zero wavenumbers, however, the raw relative error~\eqref{eq:rel_error} does not exhibit a clean $\mathcal{O}(h^{\mathfrak{o}^\star})$ behaviour as $h\to 0$. This is because the combined error bound satisfies
\begin{equation}\label{eq:error_model}
e^{(\mathsf{B})}(\delta,h,k) \lesssim \left(c_0 + c_1 |I_{p,m}^{(\mathsf{B})}(\varkappa)|\right) h^{\mathfrak{o}^\star}
\end{equation}
for sufficiently small $h>0$, here $c_0, c_1 > 0$ are constants independent of $k$ and $h$, and $I_{p,m}^{(\mathsf{B})}(\varkappa)$, introduced in~\eqref{eq:important_integral_gen}, is a moment integral that depends on the scaled wavenumber $\varkappa = \delta k \propto h^{\mu}k$, introducing an $h$-varying prefactor that obscures the algebraic convergence rate. \rev{To remove this dependence we divide the error by $|I_{p,m}^{(\mathsf{B})}(\varkappa)|$, the factor through which the prefactor in~\eqref{eq:error_model} varies with $h$, and rescale to the finest mesh:}
\begin{equation}\label{eq:norm_model_error}
\rev{\widetilde{e}^{(\mathsf{B})}_{\rm nrm}(\delta,h,k) := e^{(\mathsf{B})}(\delta,h,k)\,\frac{|I_{p,m}^{(\mathsf{B})}(\varkappa_{\rm f})|}{|I_{p,m}^{(\mathsf{B})}(\varkappa)|},}
\end{equation}
\rev{where $\varkappa_{\rm f}$ is the value of $\varkappa$ on the finest mesh, so that $\widetilde e^{(\mathsf B)}_{\rm nrm}$ coincides with $e^{(\mathsf B)}$ there. This normalization is effective when the regularization error, whose constant is proportional to $|I_{p,m}^{(\mathsf B)}(\varkappa)|$, dominates the prefactor in~\eqref{eq:error_model}, i.e.\ $c_0\ll c_1|I_{p,m}^{(\mathsf B)}(\varkappa)|$. Figure~\ref{fig:h_convergence_k_1_ellipsoid} displays $\widetilde{e}^{(\mathsf{B})}_{\rm nrm}$ for $k=\pi$ on the ellipsoid, where we take $|I^{(\mathsf T)}_{p,m}|=|I^{(\mathsf H)}_{p,m}|+|I^{(\mathsf W)}_{p,m}|$ for the hypersingular operator. As at $k=0$, for $\mathfrak m\leq7$ the errors of $\mathsf S$, $\mathsf K$ and $\mathsf K^\top$ approach the predicted slopes as $h$ decreases, while $\mathsf T$ converges faster than the rate predicted for its $\mathsf W$ part, both here and at $k=0$ (Figure~\ref{fig:h_convergence_k_0_ellipsoid}).
For $\mathfrak m=9$ the convergence of $\mathsf S$, $\mathsf K$ and $\mathsf K^\top$ is again slower, and for $\mathfrak q=2$ the error of $\mathsf S$ retains some scatter, as the coefficients of its regularizing function vary appreciably with $\varkappa$ when $\mathfrak m=9$. Overall, for both wavenumbers the observed convergence is consistent with Theorem~\ref{thm:discretization_error}. Once $\delta$ is small enough for the regularization error to follow $\delta^{\mathfrak m}$, the errors approach or exceed the predicted rates.}

\begin{figure}[htbp]
  \centering\color{revffcolor}
  \includegraphics[width=0.24\textwidth]{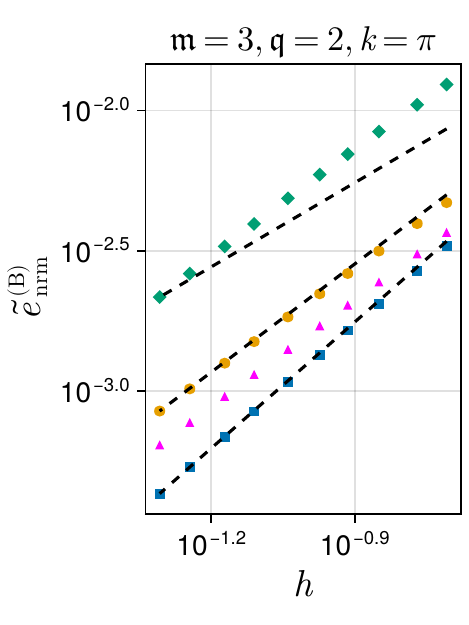}
  \includegraphics[width=0.24\textwidth]{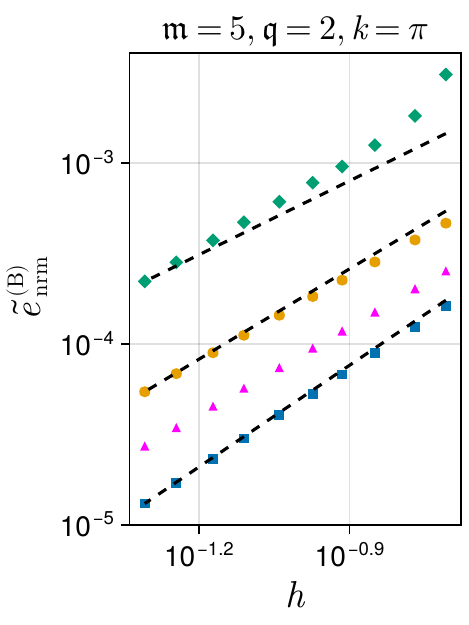}
  \includegraphics[width=0.24\textwidth]{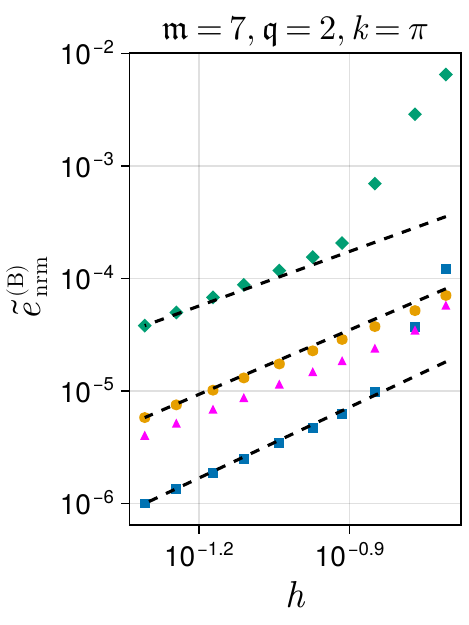}
  \includegraphics[width=0.24\textwidth]{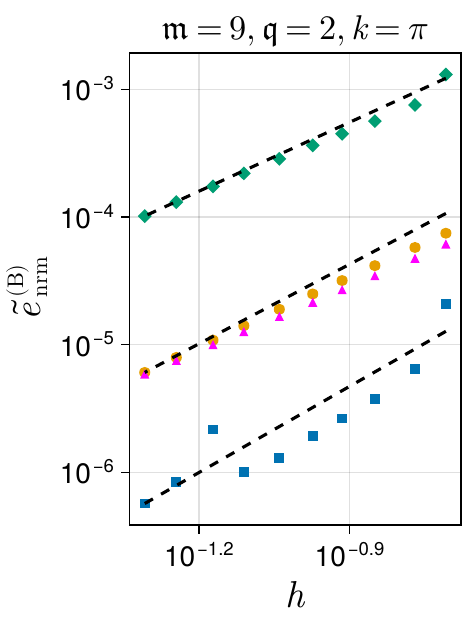}\\
  \includegraphics[width=0.24\textwidth]{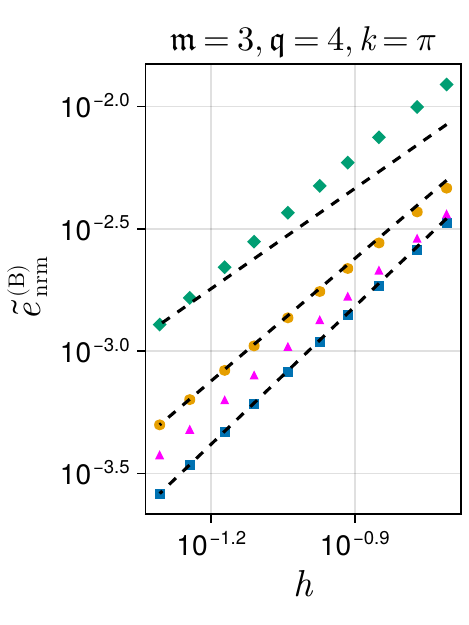}
  \includegraphics[width=0.24\textwidth]{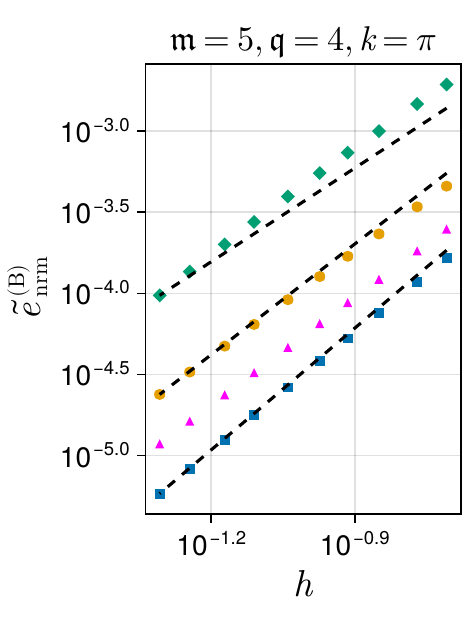}
  \includegraphics[width=0.24\textwidth]{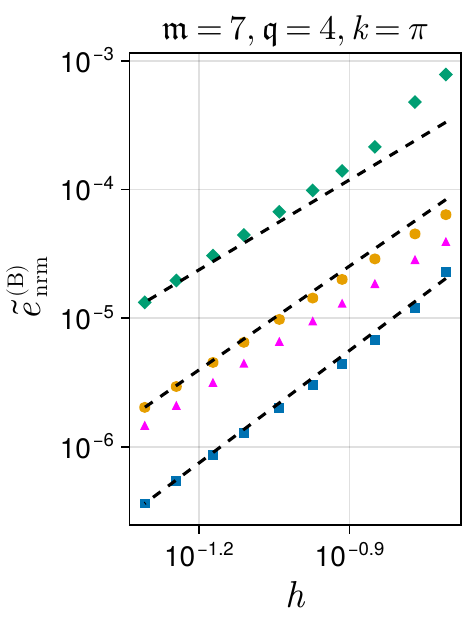}
  \includegraphics[width=0.24\textwidth]{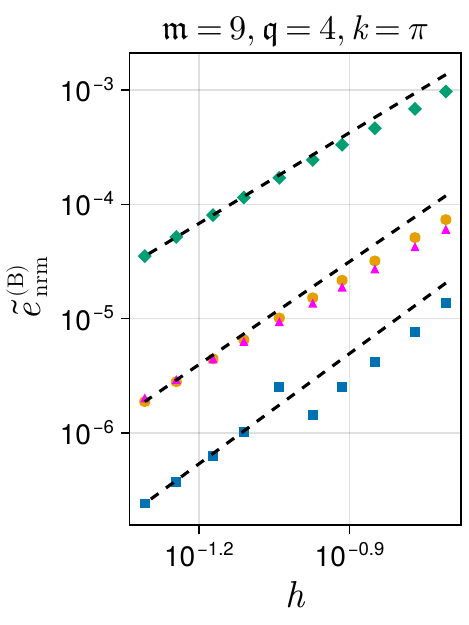}\\
  \includegraphics[width=0.24\textwidth]{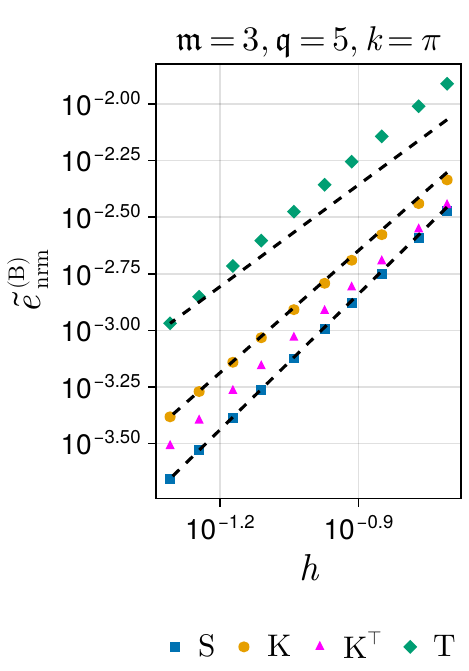}
  \includegraphics[width=0.24\textwidth]{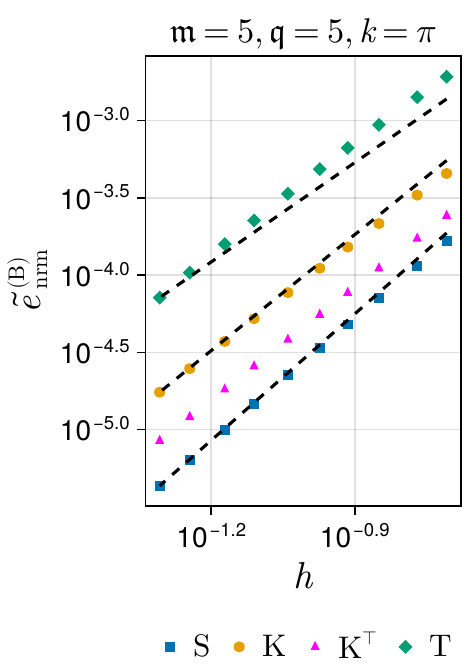}
  \includegraphics[width=0.24\textwidth]{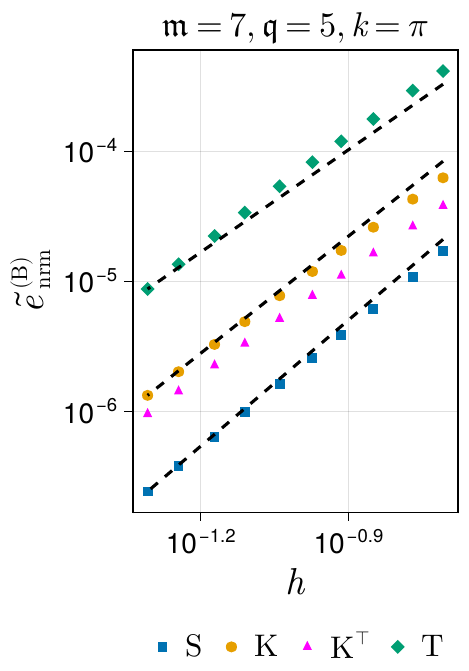}
  \includegraphics[width=0.24\textwidth]{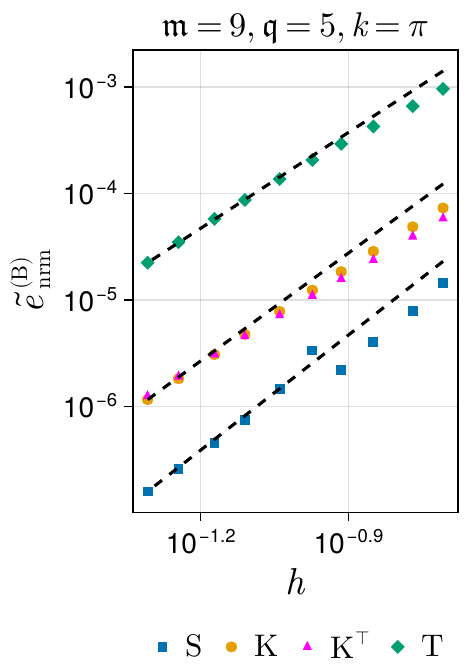}
  \caption{\rev{Normalized relative error $\widetilde e^{(\mathsf{B})}_{\rm nrm}$ for $\mathsf B\in\{\mathsf S,\mathsf K,\mathsf K^\top,\mathsf T\}$ defined in~\eqref{eq:norm_model_error} as a function of the mesh size $h$ on the ellipsoid with semi-axes $(1,0.8,0.6)$, for regularization orders $\mathfrak{m}\in\{3,5,7,9\}$ (left to right) and quadrature rules of degree of exactness $\mathfrak{q}\in\{2,4,5\}$ (top to bottom), with wavenumber $k=\pi$ and $\delta=c\,h^{\mu^\star}$ ($\lambda_\delta=0.75$). The errors are divided by $|I^{(\mathsf B)}_{p,m}(\varkappa)|$ and rescaled to coincide with $e^{(\mathsf B)}$ on the finest mesh; no constant is fitted. The dashed black lines show the theoretical $\mathcal{O}(h^{\mathfrak{o}^\star})$ convergence rates of Table~\ref{tab:ops_params} for $\mathsf S$, $\mathsf K$ ($\mathsf K^\top$) and the $\mathsf W$ part of $\mathsf T$.}}
  \label{fig:h_convergence_k_1_ellipsoid}
\end{figure}

\rev{We finish this section with an observation on the role of the surface curvature, which by Remark~\ref{rem:curvature} affects the onset of the asymptotic regime but not the rates. On the ellipsoid, whose largest principal curvature is $\kappa_{\max}=1/0.6^2\approx2.8$, attained at $(\pm1,0,0)$, the regularization error approaches the rate $\delta^{\mathfrak m}$ only once $\delta\kappa_{\max}$ is small, roughly below $0.4$ in Figure~\ref{fig:delta_convergence_ellipsoid}. The regularization parameter should therefore be smaller where the surface is more strongly curved. A curvature-adapted choice of $\delta$ is left for future work and briefly outlined in Section~\ref{sec:conclusion}.}

\subsection{Nyström discretization of Helmholtz boundary integral equations}\label{sec:scattering}
In this section we consider two exterior Helmholtz boundary value problems: the sound-soft (Dirichlet) problem
\begin{subequations}\begin{equation}\label{eq:dir_prob}
\left\{\begin{split}\Delta u+k^2 u = 0\quad&\text{in}\quad\R^3\setminus\overline\Omega,\\
  u = f\quad&\text{on}\quad\Gamma,\\
  \lim_{|x|\to\infty} |x|\left\{\frac{\p u}{\p |x|}-ik u\right\}=0,
\end{split}\right.
\end{equation}
and the sound-hard (Neumann) problem
\begin{equation}\label{eq:neu_prob}
\left\{\begin{split}\Delta v+k^2 v = 0\quad&\text{in}\quad\R^3\setminus\overline\Omega,\\
  \dfrac{\p v}{\p \nu} = g\quad&\text{on}\quad\Gamma,\\
  \lim_{|x|\to\infty} |x|\left\{\frac{\p v}{\p |x|}-ik v\right\}=0.
\end{split}\right.
\end{equation}\end{subequations}
The boundary data $f,g:\Gamma\to\C$ are assumed to be smooth and $\Omega\subset\R^3$ is an open and bounded domain with smooth boundary $\Gamma=\partial\Omega$ and outward unit normal denoted by $\nu$. 

Both problems are solved via an indirect combined-field integral equation formulation~\cite{Kress:1995}, using the ansätze
\begin{subequations}\begin{align}
u(x) &= \int_{\Gamma}\left\{\frac{\p G(x,y)}{\p \nu(y)}-ikG(x,y)\right\}\varphi(y)\de s(y)\quad\text{and}\label{eq:rep_for_dir}\\
v(x) &= \int_{\Gamma}\left\{\frac{\p G(x,y)}{\p \nu(y)}\mathsf S[\psi](y)-ikG(x,y)\psi(y)\right\}\de s(y),\quad x\in\R^3\setminus\overline{\Omega},\label{eq:rep_for_neu}
\end{align}\label{eq:field_rep}\end{subequations}
where $G(x,y) = \frac{\e^{ik|x-y|}}{4\pi|x-y|}$, $x\neq y$, denotes the free-space Green's function for the Helmholtz equation.
Enforcing the respective boundary conditions and taking the appropriate boundary limits of~\eqref{eq:field_rep}, we obtain the following Fredholm second-kind boundary integral equations:
\begin{subequations}\begin{align}
\left(\frac{1}{2}\mathsf I + \mathsf K - ik\mathsf S\right)\varphi &= f,\label{eq:CFIE_soft}\\
\left(\frac{ik}{2}\mathsf I + \mathsf T\mathsf S - ik\mathsf K^\top\right)\psi &= g,\label{eq:CFIE_hard}
\end{align}\end{subequations}
for the unknown smooth surface densities $\varphi,\psi:\Gamma\to\C$. Both problems are uniquely solvable in $C^{r,\alpha}(\Gamma)$, $r\in\mathbb{N}_0$, $\alpha\in(0,1)$, for all $k>0$~\cite{Kress:1995}.

The Nyström discretization of the combined-field  BIEs~\eqref{eq:CFIE_soft} and~\eqref{eq:CFIE_hard} proceeds as follows. The surface~$\Gamma$ is discretized into curved triangular elements $\mathcal T_h=\{T\}$ and a composite quadrature rule of degree of exactness $\mathfrak{q}$ with nodes $\{x_j\}_{j=1}^{N_Q}$ and weights $\{w_j\}_{j=1}^{N_Q}$ is fixed on $\Gamma$, as described in Section~\ref{sec:numerics}. Each boundary integral operator $\mathsf{B} \in \{\mathsf{S}, \mathsf{K}, \mathsf{K}^\top, \mathsf{T}\}$ is replaced by its regularized counterpart $\mathsf{B}_\delta$, which features a smooth integrand and is therefore amenable to high-order numerical integration using the same quadrature rule. Collocating the resulting equations at every quadrature node $\{x_j\}_{j=1}^{N_Q}$ and approximating the surface integrals by the composite quadrature rule yields a dense linear system of the form
\begin{equation*}
A\,\bold{c} = \bold{f},
\end{equation*}
where $A\in\C^{N_Q\times N_Q}$ is the Nyström discretization matrix, $N_Q$ is the total number of quadrature nodes, $\bold{c} = (\varphi(x_1),\ldots,\varphi(x_{N_Q}))^\top$ or $(\psi(x_1),\ldots,\psi(x_{N_Q}))^\top$ is the vector of unknown density values at the quadrature nodes, and $\bold{f}$ collects the right-hand side evaluations at those points. 

As a first test, we place a point source at $x_0\in\Omega$ and solve both problems~\eqref{eq:dir_prob} and~\eqref{eq:neu_prob} with boundary data
$$
f(x) = G(x,x_0) \qquad\text{and}\qquad g(x) = \nabla_x G(x,x_0)\cdot\nu(x),\qquad x\in\Gamma,
$$
respectively, for which the exact solution is $u(x)=v(x) = G(x,x_0)$ in $\R^3\setminus\overline{\Omega}$. The accuracy is assessed via the relative far-field error
\begin{equation}\label{eq:far_field_error}
e_{\rm ff}(h) = \frac{\max_{j\in\{1,\ldots,100\}} |w(z_j)-\widetilde{w}_h(z_j)|}{\max_{j\in\{1,\ldots,100\}}|w(z_j)|},\qquad w\in\{u,v\},
\end{equation}
where $\{z_j\}_{j=1}^{100}$ are target points distributed approximately uniformly on a sphere of radius $10$ enclosing~$\Gamma$, and $\widetilde{u}_h$, $\widetilde{v}_h$ denote the numerical approximations obtained by Nyström discretization and quadrature evaluation of the representation formulae~\eqref{eq:rep_for_dir} and~\eqref{eq:rep_for_neu}, respectively.

Figure~\ref{fig:h_convergence_torus_k_1} displays the far-field errors~\eqref{eq:far_field_error} for CFIE-D~\eqref{eq:CFIE_soft} and CFIE-N~\eqref{eq:CFIE_hard}, discretized via the Nystr\"om method, as a function of mesh size $h$ for the torus (major radius $1$, minor radius $1/2$) and the bean-shaped surface from~\cite{Bruno:2001ima}, at wavenumber $k=\pi$. The incident point source is located at $x_0=(1,1,0)$ for the torus and $x_0=(0.1,0.1,0.1)$ for the bean. Three parameter combinations $(\mathfrak{m},\mathfrak{q})\in\{(3,2),(5,4),(7,5)\}$ are considered. The regularization parameter is $\delta = \rev{c}\,h^{\mu^\star}$, with $\mu^\star$ as in~\eqref{eq:optimal_mu} and \rev{$c=\lambda_\delta h_{\max}^{1-\mu^\star}$ as in Section~\ref{sec:combined}, with $h_{\max}=0.2$ and $\lambda_\delta=0.55$ for $\mathsf S$ and $\mathsf K$ and $1.2$ for $\mathsf K^\top$, $\mathsf H$ and $\mathsf W$ on the torus, and $0.5$ and $1.0$, respectively, on the bean ($1.0$ for all operators when $(\mathfrak m,\mathfrak q)=(3,2)$); the resulting values of $\delta$ are listed in Table~\ref{tab:delta_values}}; the resulting predicted convergence rates $\mathfrak{o}^\star$ of Theorem~\ref{thm:discretization_error}\rev{, eq.~\eqref{eq:optimal_order},} are indicated by dashed black lines, corresponding to the double-layer operator for CFIE-D and the hypersingular operator for CFIE-N, as these govern the slowest convergence rate within each formulation. All linear systems are solved via GMRES~\cite{saad1986gmres} with relative tolerance $10^{-8}$, accelerated by $\mathcal{H}$-matrix compression~\cite{Hackbusch:2015} via the \texttt{HMatrices.jl} Julia package~\cite{HMatrices_jl}, requiring a nearly constant iteration count of \rev{15--20} (torus) and 16--17 (bean) across all examples (surface diameter $\approx$ one wavelength). \rev{The $\mathcal H$-matrices use a geometric cluster tree with at most $50$ points per leaf, the standard strong admissibility condition with admissibility parameter $3$, and partial adaptive cross approximation with an absolute tolerance of $10^{-8}$.} \rev{The observed convergence rates match or exceed the predicted rates $\mathfrak{o}^\star$ in all cases. The results therefore confirm the estimates of Section~\ref{sec:discretization_error}} and demonstrate that the regularization does not adversely affect GMRES convergence.

\begin{table}[h!]
\centering
\begin{tabular}{lcccc}
\hline
Experiment & $h$ & $\lambda_\delta$ & $\delta$ & $\delta/h$\\
\hline
Ellipsoid, Figs.~\ref{fig:h_convergence_k_0_ellipsoid}, \ref{fig:h_convergence_k_1_ellipsoid} & $0.049$--$0.196$ & $0.75$--$1.05$ & $0.058$--$0.21$ & $0.75$--$3.2$\\
Torus, Fig.~\ref{fig:h_convergence_torus_k_1} & $0.075$--$0.2$ & $0.55$--$1.2$ & $0.067$--$0.24$ & $0.55$--$2.3$\\
Bean, Fig.~\ref{fig:h_convergence_torus_k_1} & $0.075$--$0.2$ & $0.5$--$1.0$ & $0.061$--$0.20$ & $0.5$--$1.9$\\
Plane waves, Fig.~\ref{fig:scattering} & $0.1$ & --- ($c=0.2$) & $0.057$--$0.075$ & $0.57$--$0.75$\\
\hline
\end{tabular}
\caption{\rev{Values of the regularization parameter $\delta=c\,h^{\mu^\star}$, with $c=\lambda_\delta h_{\max}^{1-\mu^\star}$ and $h_{\max}$ the largest mesh size of each sequence, used in the experiments. The ranges are over all mesh sizes, operators, regularization orders and quadrature rules of each experiment. The value of $\lambda_\delta$ depends on the operator, and on the ellipsoid also on $\mathfrak m$, as specified in the text. For the torus, bean and plane-wave experiments, $h$ is the nominal Gmsh mesh size.}}
\label{tab:delta_values}
\end{table}

\begin{figure}[htbp]
  \centering
  \includegraphics[width=0.49\textwidth]{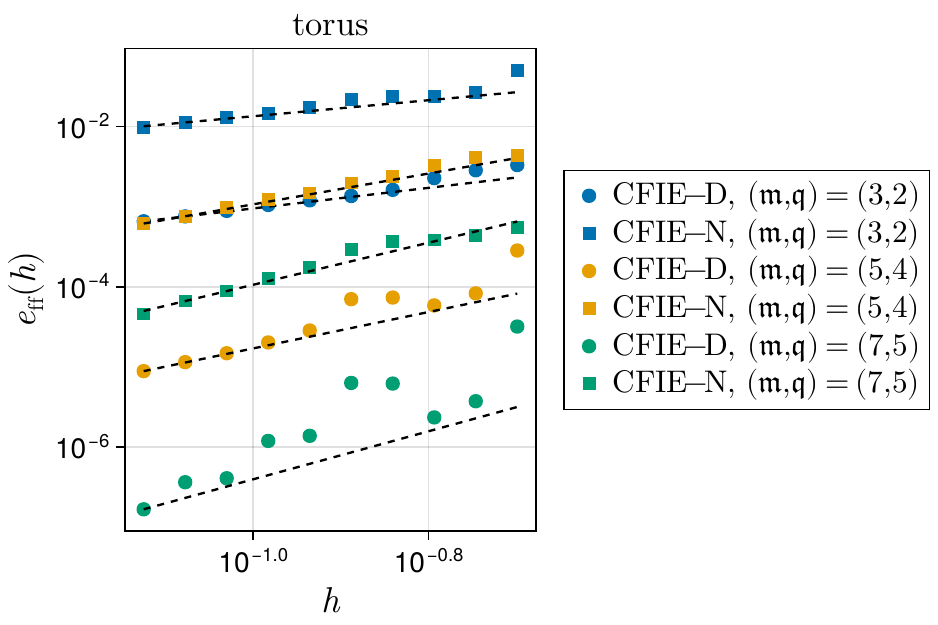}\hfill
  \includegraphics[width=0.49\textwidth]{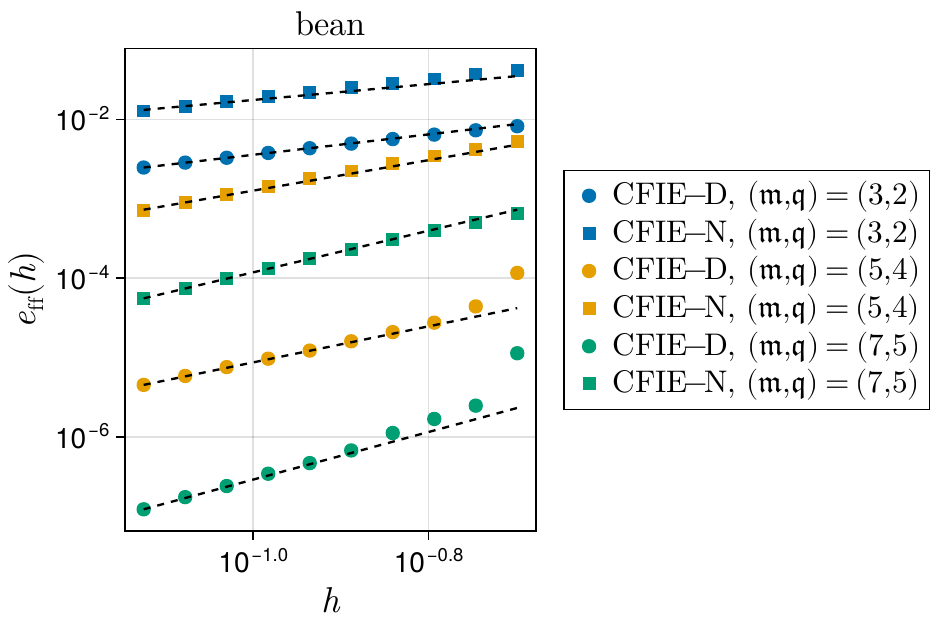} 
 
   \caption{Far-field errors $e_{\mathrm{ff}}$\eqref{eq:far_field_error} as a function of mesh size $h$ for CFIE-D\eqref{eq:CFIE_soft} and CFIE-N~\eqref{eq:CFIE_hard}, applied to an interior point-source problem on the torus (left) and bean-shaped (right) geometries at wavenumber $k=\pi$, for parameter combinations $(\mathfrak{m},\mathfrak{q})\in\{(3,2),(5,4),(7,5)\}$. Dashed black lines indicate the predicted convergence rates $\mathfrak{o}^\star$ of Theorem~\ref{thm:discretization_error}, governed by the double-layer operator (CFIE-D) and the hypersingular operator (CFIE-N).}
\label{fig:h_convergence_torus_k_1}
\end{figure}

Finally, to further illustrate the method's performance at higher frequencies, Figure~\ref{fig:scattering} displays the real part of the total fields $u^{\rm tot} = u + u^{\rm inc}$ and $v^{\rm tot} = v + u^{\rm inc}$ resulting from the scattering of a plane wave $u^{\rm inc}(x) = \e^{ik x\cdot d}$ ($|d|=1$) by the torus and bean geometries at wavenumber $k = 5\pi$ (surface diameter $\approx$ ten wavelengths for the torus and five wavelengths for the bean). Both sound-soft~\eqref{eq:dir_prob} and sound-hard~\eqref{eq:neu_prob} problems are solved with data $f(x) = -u^{\rm inc}(x)$ and $g(x) = -\nabla u^{\rm inc}(x)\cdot\nu(x)$ for $x\in\Gamma$, respectively, via CFIE-D~\eqref{eq:CFIE_soft} and CFIE-N~\eqref{eq:CFIE_hard} with parameters $(\mathfrak{m},\mathfrak{q}) = (5,5)$. Surface discretizations with mesh size $h=0.1=\lambda/4$ are used, yielding $N_Q =48460$ and $N_Q = \rev{31520}$ quadrature points for the torus and bean respectively. \rev{The regularization parameter is $\delta=0.2\,h^{\mu^\star}$, i.e.\ $\delta=0.057$ for $\mathsf S$, $0.063$ for $\mathsf K$, $\mathsf K^\top$ and $\mathsf H$, and $0.075$ for $\mathsf W$ (Table~\ref{tab:delta_values}).} GMRES iteration counts of 29 and 24 were needed for convergence of CFIE-D, and of \rev{201 and 148} for CFIE-N, with a relative tolerance of $10^{-8}$, for the torus and bean respectively; the corresponding linear system solve times were \rev{3.8 and 2.0 seconds for CFIE-D, and 42 and 17 seconds for CFIE-N (wall-clock times using 14 threads of an Apple M3 Max processor with 36~GB of RAM; shortest of two runs)}. To estimate the error, we solve the same problems with the point source from the example above and obtain far-field relative errors~\eqref{eq:far_field_error} of \rev{$2.4\cdot 10^{-3}$ and $1.2\cdot 10^{-3}$} for the torus and bean in the case of CFIE-D, and of \rev{$2.7\cdot10^{-3}$ and $2.9\cdot 10^{-3}$} in the case of CFIE-N. These results confirm that the regularization yields accurate solutions on modestly refined geometries without degrading solver efficiency, even at higher frequencies.

\begin{figure}[htbp]
  \centering
  \includegraphics[width=0.4\linewidth]{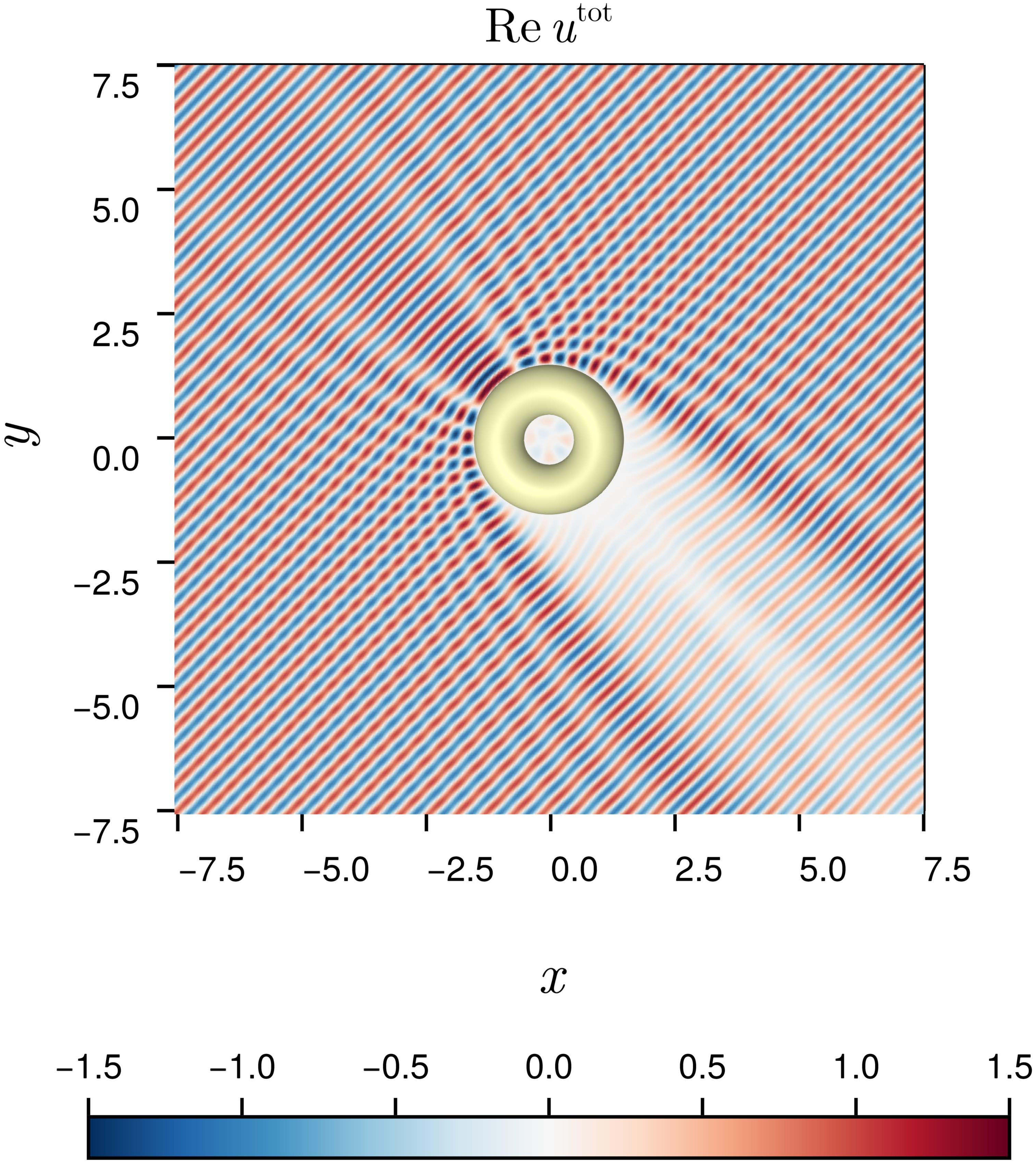} \qquad \includegraphics[width=0.4\linewidth]{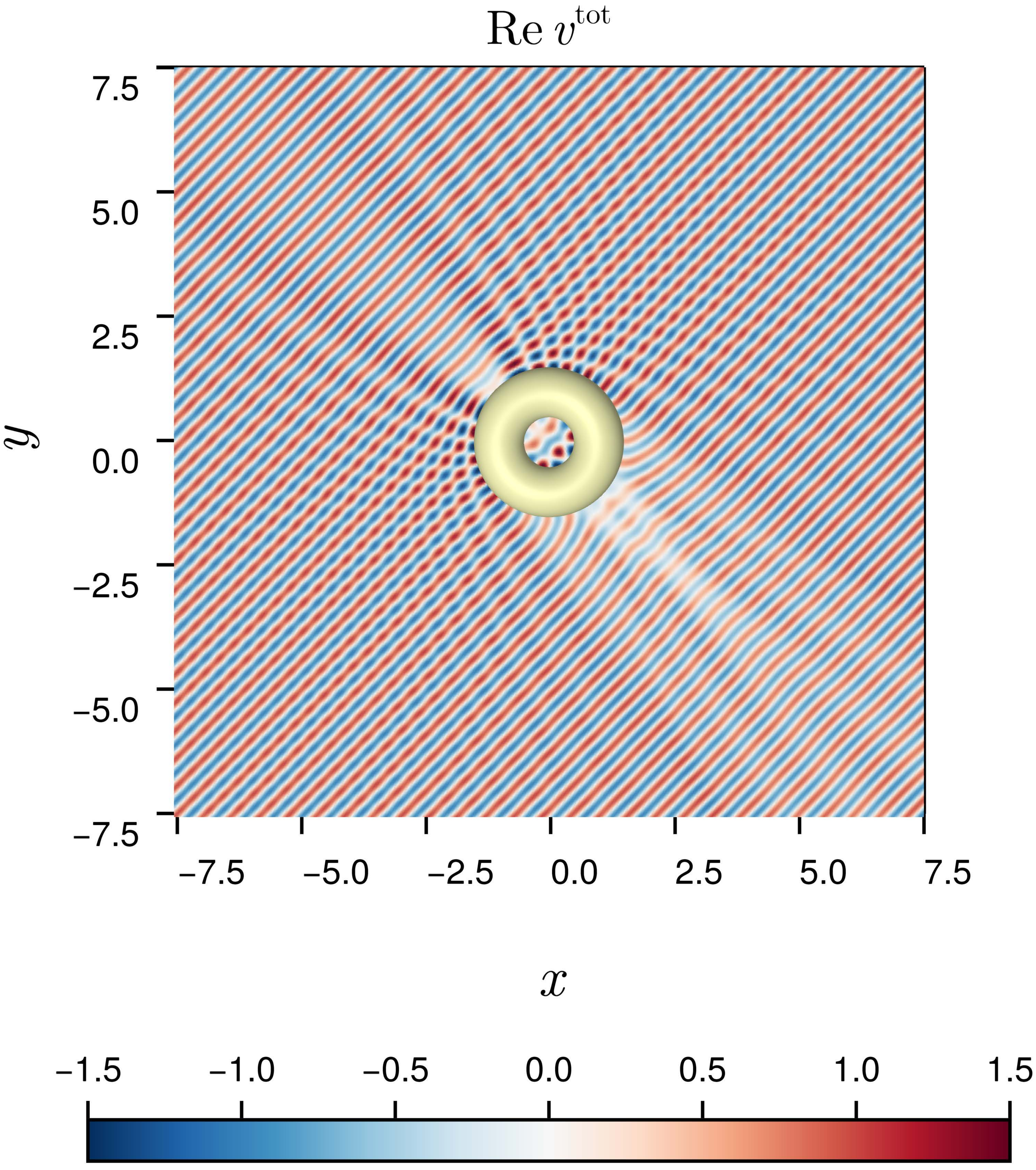}\bigskip\\
    \includegraphics[width=0.4\linewidth]{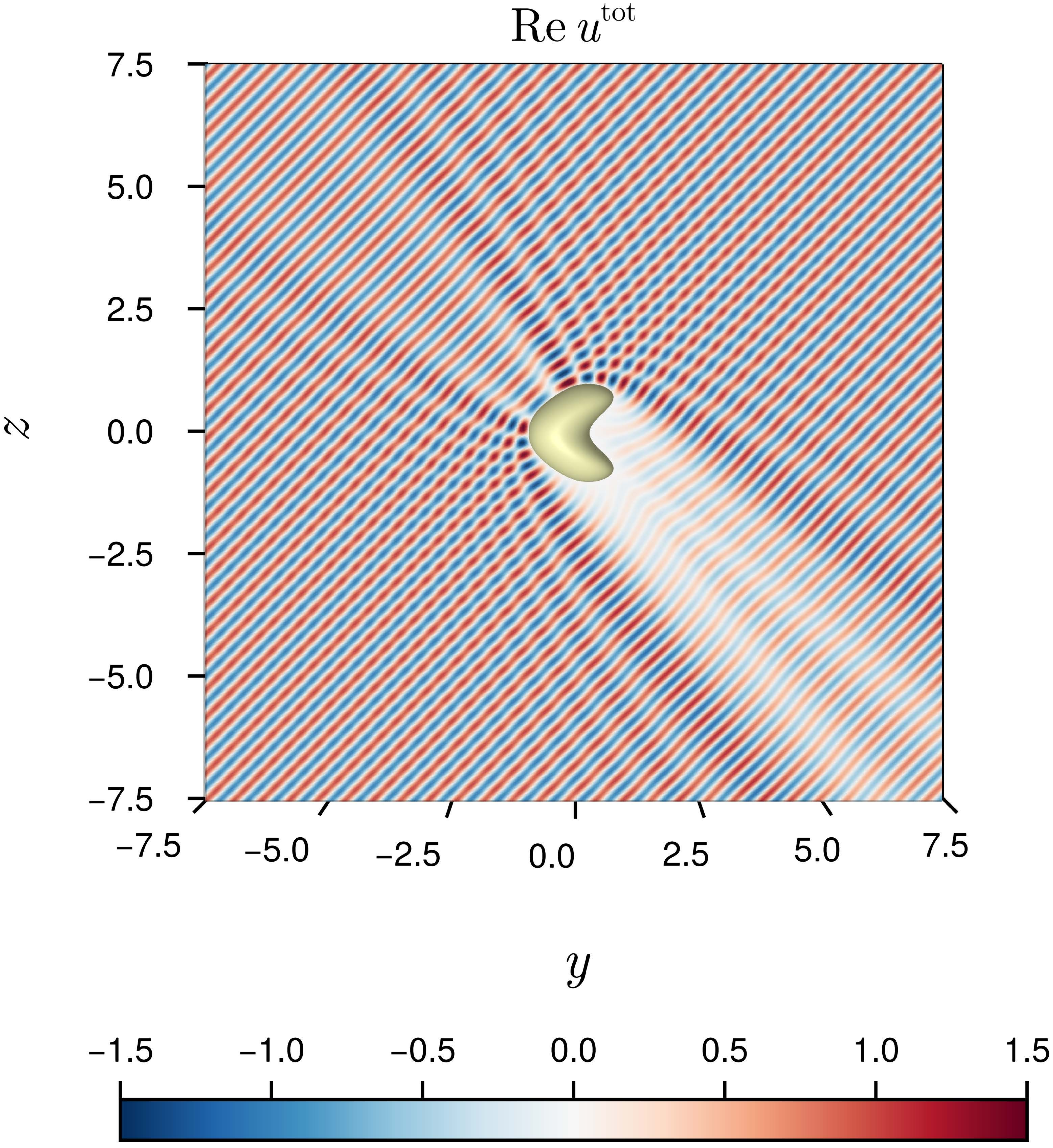}\qquad \includegraphics[width=0.4\linewidth]{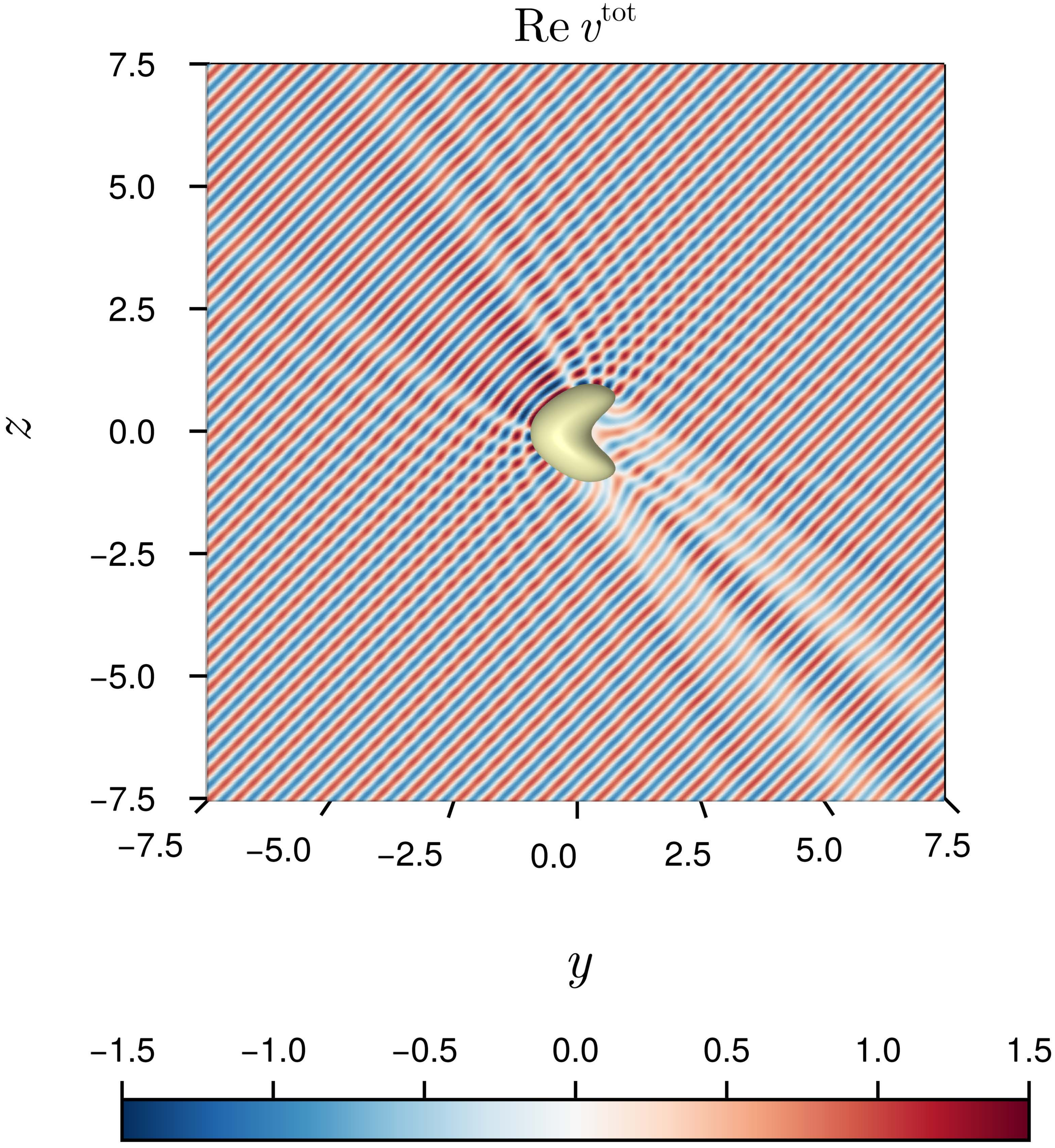}\\
  \caption{Real part of the total fields $u^{\rm tot}=u+u^{\inc}$ and $v^{\rm tot}=v+u^{\inc}$ for plane-wave scattering $u^{\inc}(x)=\e^{ik x\cdot d}$ at wavenumber $k=5\pi$, computed via the Nystr\"om method with regularization order $\mathfrak{m}=5$ and quadrature degree $\mathfrak{q}=5$. Top row: torus (major radius $1$, minor radius $1/2$); bottom row: bean-shaped surface~\cite{Bruno:2001ima}. Left column: sound-soft (Dirichlet) problem~\eqref{eq:dir_prob}; right column: sound-hard (Neumann) problem~\eqref{eq:neu_prob}.}
\label{fig:scattering}
\end{figure}

\rev{\subsection{Cost of the $\mathcal H$-matrix compression}\label{sec:hmatrix_cost}}
\rev{Since $\mathcal H$-matrix compression is used in place of a kernel-specific fast method, it is worth making explicit how its cost depends on $\delta$, the one parameter of the method with no counterpart in a standard discretization. 

The regularized kernel differs from the Helmholtz kernel by the correction $-\,\Phi(k|x-y|)|x-y|^{-(2p+1)}w_p(|x-y|/\delta)\Psi(x,y)$, which is $\mathcal O\big(P_p(t)\e^{-t^2}\big)$ with $t=|x-y|/\delta$. Admissible cluster pairs satisfy $\operatorname{dist}(\tau,\sigma)\gtrsim\operatorname{diam}\tau$, so whenever $\operatorname{dist}(\tau,\sigma)\gg\delta$ the correction is below the compression tolerance and the block coincides, to that tolerance, with the corresponding block of the unregularized operator. Only blocks whose separation is comparable to $\delta$ are affected, of which there are $\mathcal O(N_Q)$ under the quasi-uniformity assumption, so the memory and the block ranks should be essentially independent of $\delta$ for small $\delta$ and increase mildly as $\delta$ grows. To test this, we take the largest problem of Section~\ref{sec:scattering}, the torus of the plane-wave example ($h=0.1$, $N_Q=48460$, $(\mathfrak m,\mathfrak q)=(5,5)$, $k=5\pi$), for which the compressed operators $\mathsf S$, $\mathsf K$, $\mathsf K^\top$ and $\mathsf T$ occupy $5.3$, $6.2$, $6.2$ and $7.4$~GB, respectively (compression ratios between $4.7$ and $6.6$ with respect to the dense matrices), and multiply the values of $\delta$ used there by a factor $s$ ranging from $1/4$ to $8$, so that $\delta/h$ ranges from $0.1$ to $3.3$ for $\mathsf S$. 

Table~\ref{tab:hmatrix_cost} reports the total memory of the four compressed operators, the mean and maximum ranks of the admissible blocks of $\mathsf T$, the time of one matrix-vector product with the four operators, and the number of GMRES iterations and far-field errors~\eqref{eq:far_field_error} for the point-source problem, together with the same quantities for the unregularized Helmholtz kernels. For $s\leq1/2$ the memory and block ranks coincide with those of the unregularized kernels; at the value of $\delta$ used in the experiments ($s=1$) the memory increases by $4.5\%$ and the time per matrix-vector product by $8\%$, and for $\delta$ eight times larger by $35\%$ and $29\%$, respectively. The assembly time is about $1.5$ to $2$ times that of the unregularized kernels for all $\delta$ ($28$~s instead of $16$~s for the four operators at $s=1$, with $14$ threads), which reflects the cost of evaluating the regularized kernels rather than a dependence on $\delta$. The number of GMRES iterations is independent of $\delta$ for CFIE-D, whereas for CFIE-N it is largest when $\delta$ is too small to be resolved by the mesh ($s=1/4$); the accuracy is acceptable only for $1/2\leq s\leq1$, in agreement with the dependence on $\lambda_\delta$ discussed in Section~\ref{sec:combined}. For the values of $\delta$ that are relevant in practice, the regularization thus has a negligible effect on the cost of the $\mathcal H$-matrix compression and of the matrix-vector products.}

\begin{table}[h!]
\centering
\setlength{\tabcolsep}{4.5pt}
\begin{tabular}{cccccccccc}
\toprule
 & & memory & \multicolumn{2}{c}{ranks of $\mathsf T$} & mat-vec & \multicolumn{2}{c}{GMRES its} & \multicolumn{2}{c}{$e_{\rm ff}$}\\
\cmidrule(lr){4-5}\cmidrule(lr){7-8}\cmidrule(lr){9-10}
$s$ & $\delta_{\mathsf S}/h$ & (GB) & mean & max & (s) & D & N & D & N\\
\midrule
--- & ---  & $24.0$ & $54.6$ & $293$ & $0.24$ & --- & ---   & --- & ---\\
$1/4$ & $0.10$ & $24.0$ & $54.6$ & $295$ & $0.25$ & $33$ & $294$ & $1.2\cdot10^{-2}$ & $2.8\cdot10^{-1}$\\
$1/2$ & $0.21$ & $24.0$ & $54.9$ & $295$ & $0.26$ & $32$ & $254$ & $5.2\cdot10^{-4}$ & $2.1\cdot10^{-3}$\\
$1$   & $0.42$ & $25.1$ & $64.1$ & $295$ & $0.26$ & $33$ & $198$ & $2.4\cdot10^{-3}$ & $2.7\cdot10^{-3}$\\
$2$   & $0.83$ & $29.2$ & $80.9$ & $329$ & $0.30$ & $33$ & $196$ & $1.2\cdot10^{-2}$ & $2.9\cdot10^{-2}$\\
$4$   & $1.67$ & $30.1$ & $67.8$ & $470$ & $0.30$ & $29$ & $234$ & $3.3\cdot10^{-1}$ & $3.4\cdot10^{-1}$\\
$8$   & $3.33$ & $32.4$ & $59.0$ & $493$ & $0.31$ & $28$ & $92$  & $3.9\cdot10^{-1}$ & $6.1\cdot10^{-1}$\\
\bottomrule
\end{tabular}
\caption{\rev{Cost of the $\mathcal H$-matrix representation as a function of $\delta$, for the torus with $h=0.1$ ($N_Q=48460$), $(\mathfrak m,\mathfrak q)=(5,5)$ and $k=5\pi$, where $\delta$ is $s$ times the value used in the experiments (Table~\ref{tab:delta_values}). Memory: total for $\mathsf S$, $\mathsf K$, $\mathsf K^\top$ and $\mathsf T$; mat-vec: total time of one product with the four operators (14 threads); GMRES iterations (relative tolerance $10^{-8}$) and far-field errors for CFIE-D and CFIE-N with the point source of Figure~\ref{fig:h_convergence_torus_k_1}. First row: unregularized Helmholtz kernels compressed with the same parameters.}}
\label{tab:hmatrix_cost}
\end{table}

\section{Conclusions and future work}\label{sec:conclusion}

This paper has established that the kernel regularization approach extends cleanly to the full Helmholtz Calder\'{o}n calculus, including the hypersingular operator, within a unified theoretical and computational framework. \rev{A local error analysis of the combined regularization and quadrature error yields, for each operator, the scaling $\delta\propto h^{\mu^\star}$ and the resulting convergence order, both confirmed by the experiments.} The numerical experiments \rev{also} suggest that the method remains robust at higher frequencies and that $\mathcal{H}$-matrix acceleration effectively compensates for the incompatibility of the modified kernel with FMM-based solvers\rev{, at a cost that depends only weakly on $\delta$}.

\rev{We see several directions for future work. First, $\delta$ need not be uniform. Since the error estimates hold at each target separately (Remark~\ref{rem:curvature}), $\delta$ can be chosen per target, e.g.\ $\delta(x)=c\,h(x)^{\mu^\star}$ with a local length scale $h(x)=h\min\{1,\kappa_0/\kappa(x)\}$ that decreases with the largest absolute principal curvature $\kappa(x)$. Preliminary experiments indicate that this significantly reduces the error at highly curved points. Second, nearly singular evaluation near the surface. Beale and Tlupova~\cite{beale2025high} (see also~\cite{beale2024extrapolated}) have done this for the Laplace and Stokes kernels by centering the regularization at the nearest surface point and letting the coefficients depend on the scaled target-to-surface distance, and we expect the idea to carry over. For Helmholtz the coefficients would depend on two parameters, $\varkappa=\delta k$ and that distance, and the isolated values of $\varkappa$ at which the moment matrix of Remark~\ref{rem:ill-cond_mat} loses rank would become curves in the parameter plane, making the rectangular/pseudoinverse formulation essential. Finally, the electric- and magnetic-field integral operators of the time-harmonic Maxwell equations and the volume potentials of Lippmann--Schwinger-type formulations share the singularity structure treated here, which makes both natural candidates for extending the method.}
\rev{\section*{Acknowledgments}
We are grateful to Prof.~J.~Thomas Beale (Duke University) for helpful discussions and for his note~\cite{beale2026lagrange}, which allowed us to relax the smoothness assumptions of the analysis. We also thank the two anonymous referees for their careful reading and useful comments, which substantially improved the paper.}

\appendix

\section{Recurrence relations for integrals}\label{sec:coeff}
In this section, we derive recurrence relations for the evaluation of the coefficients involved in the linear systems~\eqref{eq:system_SL_coeff}, \eqref{eq:system_K}, and \eqref{eq:system_for_hyper}, which are used in the construction of the regularizing functions of the Helmholtz boundary integral operators considered in this work. 

\begin{proposition} \label{prop:prop_1}
The functions  $S_{j}(\varkappa)$ and $C_{j}(\varkappa)$,  defined by 
\begin{subequations}\label{eq:aux_funcs}\begin{align}
 S_{j}(\varkappa):=&~\frac{2}{\sqrt\pi}\int_{0}^\infty \sin(\varkappa t)\e^{-t^2}t^{2j}\de t\quad\text{and}\label{eq:def_coeff_C}\\
 C_{j}(\varkappa):=&~\frac{2}{\sqrt\pi}\int_{0}^\infty \cos(\varkappa t)\e^{-t^2}t^{2j+1}\de t
\quad\text{for } \varkappa\in\R\text{ and }  j\in\N_0,\label{eq:def_coeff_S}
\end{align}\end{subequations}
 satisfy the following recurrence relations:
\begin{subequations}\begin{align}
S_{j}(\varkappa) =&~\left(j-\frac12\right)S_{j-1}(\varkappa)+\frac{\varkappa}{2} C_{j-1}(\varkappa)\quad\text{and}\label{eq:rec_rel_S}\\
C_{j}(\varkappa) =&~jC_{j-1}(\varkappa)-\frac{\varkappa}{2} S_{j}(\varkappa)\qquad\qquad\qquad\text{for}\quad j\in\N, \label{eq:rec_rel_C}
\end{align}\end{subequations}
with initial conditions
 \begin{align}\label{eq:init_conds_aux_funcs}
C_0(\varkappa) = \frac{1-\varkappa F\left(\frac{\varkappa}{2}\right)}{\sqrt{\pi }}\qquad\text{and}\qquad
 S_{0}(\varkappa)=\frac{2}{\sqrt\pi}F\left(\frac{\varkappa}{2}\right),\end{align}
where
\begin{equation*}\label{eq:Dawson}
F(x):=\e^{-x^2} \int_0^x \e^{t^2} \de t,\qquad x\in\R,
\end{equation*}
denotes the Dawson function~\cite{abramowitz1965handbook}. 
Moreover, at $\varkappa=0$, we have
\begin{equation}\label{eq:exact_coeff_0}
C_{j}(0)
=  \frac{j!}{\sqrt{\pi}}
\quad\text{for}\quad j\in\N_0.
\end{equation}

\end{proposition}

\begin{proof}

The initial conditions~\eqref{eq:init_conds_aux_funcs} follow directly from the evaluation of the integrals in~\eqref{eq:aux_funcs} corresponding to $j=0$. 
Let $\ell\geq 1$ and $f_\ell(t):=t^{\ell}\e^{-t^2}$. 

By integrating by parts,
noting that the boundary terms vanish for $\ell\ge 1$, we arrive at
\begin{equation}\label{eq:prop_1}
i\varkappa\int_0^\infty f_\ell(t) \e^{i\varkappa t}\de t = -\int_0^\infty f_\ell'(t)\e^{i\varkappa t}\de t 
=~-\int_0^\infty \e^{-t^2}\left\{\ell t^{\ell-1}-2t^{\ell+1}\right\}\e^{i\varkappa t}\de t.
\end{equation}
Taking real part in~\eqref{eq:prop_1}, setting $\ell = 2j$, and using the definitions~\eqref{eq:aux_funcs},  we obtain 
$$
-\varkappa S_{j}(\varkappa) = -2jC_{j-1}(\varkappa)+2C_{j}(\varkappa),
$$
from where the recurrence relation~\eqref{eq:rec_rel_C} follows.
Similarly, taking imaginary part in~\eqref{eq:prop_1} and setting $\ell = 2j-1$ we arrive at 
$$
\varkappa C_{j-1}(\varkappa) = -(2j-1)S_{j-1}(\varkappa)+2S_{j}(\varkappa),
$$
from where the recurrence relation~\eqref{eq:rec_rel_S} follows.

Finally, formula~\eqref{eq:exact_coeff_0} follows from setting $\varkappa=0$ in~\eqref{eq:rec_rel_C} and applying the initial condition~\eqref{eq:init_conds_aux_funcs} which yields $C_0(0) = 1/\sqrt{\pi}$. 
\end{proof}

\begin{proposition}\label{prop:prop_2}
The functions   $\widetilde C_{j}(\varkappa)$ and $\widetilde S_{j}(\varkappa)$,  defined by
\begin{subequations}\label{eq:def_tilde}\begin{align}
\widetilde C_{j}(\varkappa):=&~\int_{0}^\infty \cos(\varkappa t)\erfc(t)t^{2j}\de t\quad\text{and}\label{eq:def_coeff_C_erf}\\
\widetilde  S_{j}(\varkappa):=&~\int_{0}^\infty \sin(\varkappa t)\erfc(t)t^{2j+1}\de t\quad\text{for }\varkappa\in\R\text{ and }j\in\N_0,\label{eq:def_coeff_S_erf}
 \end{align}\end{subequations}
 satisfy  the following recurrence relations:
\begin{subequations}\label{eq:rec_tilde}\begin{align}
\widetilde C_{j}(\varkappa) =&~~\frac{S_{j}(\varkappa)-2j \widetilde S_{j-1}(\varkappa)}{\varkappa} \quad\text{and}\quad \label{eq:rec_rel_c_erf}\\
\widetilde S_{j}(\varkappa) =&~~\frac{(2j+1) \widetilde C_{j}(\varkappa)-C_{j}(\varkappa)}{\varkappa}\quad\text{for}\quad \varkappa\in\R\setminus\{0\}\text{ and } j\in\N,\label{eq:rec_rel_s_erf}
\end{align}\end{subequations}
with the initial conditions
 \begin{align}\label{eq:init_cond_tilde}
\widetilde C_{0}(\varkappa)=\frac{2}{\sqrt\pi}\frac{F\left(\frac{\varkappa}{2}\right)}{ \varkappa}\qquad\text{and}\qquad
\widetilde S_{0}(\varkappa)=\frac{\left(\varkappa^2+2\right) F\left(\frac{\varkappa}{2}\right)-\varkappa}{\sqrt{\pi } \varkappa^2},
\end{align}
 where $\{C_{j}(\varkappa)\}_{j\in\N_0}$ and~$\{S_{j}(\varkappa)\}_{j\in\N_0}$ are defined in~\eqref{eq:def_coeff_C} and~\eqref{eq:def_coeff_S}, respectively. Moreover, at $\varkappa=0$, one has 
\begin{equation}\label{eq:zero_erf}
 \widetilde C_{j}(0) = \frac{1}{\sqrt\pi}\frac{j!}{2j+1},\qquad j\in\N_0.
\end{equation}
\end{proposition}

\begin{proof} We proceed as in the proof of the proposition above. The initial conditions~\eqref{eq:init_cond_tilde} are obtained by direct evaluation of the improper integrals~\eqref{eq:def_tilde} at $j = 0$.  Let $\ell\geq 1$ and $f_\ell(t) := \erfc(t)t^{\ell}$. Integrating by parts and using the fact that for $\ell\geq 1$ the boundary terms vanish, we readily get
\begin{equation}\label{eq:interm_prop_2}\begin{split}
i\varkappa\int_0^\infty f_\ell(t)\e^{i\varkappa t}\de t =&-\int_{0}^\infty f'_\ell(t)\e^{i\varkappa t}\de t\\
=&-\ell\int_{0}^\infty \erfc(t)t^{\ell-1}\e^{i\varkappa t}\de t+\frac{2}{\sqrt{\pi}}\int_{0}^\infty \e^{-t^2}t^\ell\e^{i\varkappa t}\de t.
\end{split}\end{equation}
Taking the imaginary part of the identity~\eqref{eq:interm_prop_2}, setting $\ell=2j$, and applying the definitions~\eqref{eq:def_tilde} and~\eqref{eq:aux_funcs}, we obtain
$$
\varkappa \widetilde C_{j}(\varkappa) = -2j\widetilde S_{j-1}(\varkappa) +S_{j}(\varkappa),
$$
from which the recurrence relation~\eqref{eq:rec_rel_c_erf} follows.
Similarly, taking the real part of the identity~\eqref{eq:interm_prop_2} and setting $\ell =2j+1$, we obtain
$$
-\varkappa \widetilde S_{j}(\varkappa) = -(2j+1)\widetilde C_{j}(\varkappa) + C_{j}(\varkappa),
$$
from where both, the recurrence~\eqref{eq:rec_rel_s_erf} and 
$$
 \widetilde C_{j}(0) = \frac{C_{j}(0)}{2j+1},\qquad j\in\N_0,
 $$
 follow.
 Finally,~\eqref{eq:zero_erf} is obtained by substituting~\eqref{eq:exact_coeff_0} in the identity above.
\end{proof}

\begin{remark} The recurrence relations~\eqref{eq:rec_tilde} exhibit a clear numerical instability for small values of $\varkappa$. To address this issue, one can, of course, directly evaluate the integrals numerically. Yet an alternative approach is to employ the following Taylor expansions 
\begin{align}
    \widetilde{C}_j(\varkappa) &= \sum_{\ell=0}^\infty \frac{(-1)^\ell \widetilde{C}_{\ell+j}(0)}{(2\ell)!} \varkappa^{2\ell} 
    =\frac{1}{\sqrt\pi} \sum_{\ell=0}^\infty \frac{(-1)^\ell}{2(\ell+j)+1} \frac{(\ell+j)!}{(2\ell)!} \varkappa^{2\ell} , \label{eq:C_expansion} \\
    \widetilde{S}_j(\varkappa) &= \sum_{\ell=0}^\infty \frac{(-1)^\ell \widetilde{C}_{\ell+j+1}(0)}{(2\ell+1)!} \varkappa^{2\ell+1} 
    =\frac{1}{\sqrt\pi}  \sum_{\ell=0}^\infty \frac{(-1)^\ell}{2(\ell+j)+3} \frac{(\ell+j+1)!}{(2\ell+1)!} \varkappa^{2\ell+1},\quad j\in\N_0, \label{eq:S_expansion}
\end{align}
which can be used safely for $\varkappa \in (0, 1)$. Indeed, to maintain machine precision (relative error $\approx 10^{-16}$) in this regime, the series should be truncated at approximately $25$ terms. This provides a sufficient safety margin even for moderately large values of~$j$.  

Moreover, while not strictly necessary in practice, a similar approach can be used to numerically evaluate the functions defined in \eqref{eq:aux_funcs}. Specifically, the following series expansions can be employed:
\begin{align*}
C_j(\varkappa) &= \sum_{\ell=0}^\infty\frac{(-1)^\ell C_{\ell+j}(0)}{(2\ell)!} \varkappa^{2\ell} = \frac{1}{\sqrt\pi}\sum_{\ell=0}^\infty\frac{(-1)^\ell (\ell+j)!}{(2\ell)!} \varkappa^{2\ell}, \\
S_j(\varkappa) &= \sum_{\ell=0}^\infty\frac{(-1)^\ell C_{\ell+j}(0)}{(2\ell+1)!} \varkappa^{2\ell+1} = \frac{1}{\sqrt\pi}\sum_{\ell=0}^\infty\frac{(-1)^\ell (\ell+j)!}{(2\ell+1)!} \varkappa^{2\ell+1}, \quad j\in\mathbb{N}_0.
\end{align*}
As before, approximately 25 terms are sufficient to achieve near machine-precision accuracy.
\end{remark}

To compute the coefficients of the regularizing function $\sigma_2^{(H)}$ associated with the hypersingular operator, we must evaluate a few singular integrals that cannot be directly expressed in terms of the functions introduced in Propositions~\ref{prop:prop_1} and~\ref{prop:prop_2}. The following result defines these missing terms and provides representations of the Hadamard finite-part integrals appearing in their definitions as improper integrals suitable for numerical evaluation. Unfortunately, to the best of our knowledge, these integrals admit closed-form expressions only in terms of derivatives of confluent hypergeometric functions, which are rather cumbersome for practical use.
\begin{proposition}\label{prop:prop_3} Let
\begin{subequations}\begin{align*}
\widetilde C_{-1}(\varkappa) := {\rm p.f.}\!\int_{0}^\infty \frac{\cos(\varkappa t)}{t^2}\erfc(t)\de t,\qquad C_{-1}(\varkappa):=\frac{2}{\sqrt{\pi}}{\rm p.f.\!}\int_0^\infty \frac{\cos(\varkappa t)}{t}\e^{-t^2}\de t,
\end{align*}
and
\begin{align*} \widetilde S_{-1}(\varkappa) :=&\int_0^\infty \frac{\sin(\varkappa t)}{t}\erfc(t)\de t.\end{align*}\end{subequations}
Then, the following identities hold:
\begin{subequations}\begin{align}
\widetilde C_{-1}(\varkappa) =& -\frac{2}{\sqrt{\pi}}-\varkappa \widetilde S_{-1}(\varkappa) -C_{-1}(\varkappa),\label{eq:rec_neg_1}\\
C_{-1}(\varkappa) =&\frac{2}{\sqrt\pi} \int_0^\infty \left\{\varkappa\sin(\varkappa t)+2t\cos(\varkappa t)\right\}\e^{-t^2}\ln(t)\de t .\label{eq:rec_neg_2}
\end{align}\end{subequations}
Moreover, in the special case $\varkappa=0$, we have 
\begin{equation}\label{eq:sing_case}
\widetilde C_{-1} (0) =\frac{\gamma-2 }{\sqrt{\pi}} \qquad\text{and}\qquad C_{-1}(0) = -\frac{\gamma }{\sqrt{\pi}},
\end{equation}
where $\gamma\approx 0.57722$ is Euler's gamma constant~\cite{abramowitz1965handbook}.

\end{proposition}

\begin{proof}
Let $\varepsilon>0$. Integration by parts yields
\begin{align}\label{eq:int_parts_1}
\int_{\varepsilon}^\infty \frac{\cos(\varkappa t)}{t^2}\erfc(t)\de t =& \frac{\cos(\varkappa \varepsilon)\erfc(\varepsilon)}{\varepsilon} - \varkappa\int_{\varepsilon}^\infty\frac{\sin(\varkappa t)}{t}\erfc(t)\de t -\frac{2}{\sqrt\pi}\int_{\varepsilon}^\infty\frac{\cos(\varkappa t)}{t}\e^{-t^2}\de t,\\
\int_{\varepsilon}^\infty\frac{\cos(\varkappa t)}{t}\e^{-t^2}\de t =& -\ln(\varepsilon)\cos(\varkappa \varepsilon)\e^{-\varepsilon^2}  + \int_{\varepsilon}^\infty \{ \varkappa\sin(\varkappa t)+2t\cos{\varkappa t}\}\e^{-t^2}\ln(t)\de t.\label{eq:int_parts_2}
\end{align}

It follows from~\eqref{eq:int_parts_2}, the expansion
\begin{equation*}
 \ln(\varepsilon)\cos(\varkappa\varepsilon)\e^{-\varepsilon^2} = -\ln(\varepsilon) + \mathcal{O}(\varepsilon)\quad\text{as}\quad\varepsilon\to0 +,
\end{equation*}
and the definition of Hadamard's finite part integral~\cite[sec.~3.2]{hsiao2021boundary}, 
that
\begin{equation}\label{eq:AAA}\begin{split}
{\rm p.f.}\!\int_{0}^\infty\frac{\cos(\varkappa t)}{t}\e^{-t^2}\de t :=& \lim_{\varepsilon\to0+}\left\{\int_{\varepsilon}^\infty\frac{\cos(\varkappa t)}{t}\e^{-t^2}\de t+\ln(\varepsilon)\right\} \\
=& \int_0^\infty\{ \varkappa\sin(\varkappa t)+2t\cos(\varkappa t)\}\e^{-t^2}\ln(t)\de t ,\end{split}
\end{equation}
where the last integral above is well-defined as an improper integral. 

Similarly, by combining~\eqref{eq:int_parts_1} and~\eqref{eq:int_parts_2} and employing the  asymptotic expansion 
$$
\frac{\cos (\varkappa \varepsilon) \operatorname{erfc}(\varepsilon)}{\varepsilon}= \frac{1}{\varepsilon} -\frac{2}{\sqrt{\pi}} + \mathcal{O}(\varepsilon)\quad\text{as}\quad\varepsilon\to0+,
$$
 we obtain
\begin{equation}\label{eq:BBB}\begin{split}
{\rm p.f.}\!\int_{0}^\infty\frac{\cos(\varkappa t)}{t^2}\erfc(t)\de t  :=& \lim_{\varepsilon\to 0+}\left\{\int_{\varepsilon}^\infty\frac{\cos(\varkappa t)}{t^2}\erfc(t)\de t -\frac{2}{\sqrt{\pi}}\ln(\varepsilon)-\frac{1}{\varepsilon}\right\}\\
=-\frac{2}{\sqrt{\pi}}-\varkappa\int_{0}^\infty&\frac{\sin(\varkappa t)}{t}\erfc(t)\de t - \frac{2}{\sqrt{\pi}}\int_0^\infty\{ \varkappa\sin(\varkappa t)+2t\cos(\varkappa t)\}\e^{-t^2}\ln(t)\de t.
\end{split}\end{equation}

Therefore, clearly, \eqref{eq:rec_neg_2} follows from \eqref{eq:AAA}, while~\eqref{eq:rec_neg_1} follows from~\eqref{eq:BBB}.

Finally,~\eqref{eq:sing_case} is obtained by setting $\varkappa=0$ in~\eqref{eq:AAA} and~\eqref{eq:BBB} and using the identity
$$
\int_{0}^\infty t\ln(t)\e^{-t^2}\de t = -\frac{\gamma}{4}.
$$
  \end{proof}

\rev{For practical reasons, in our implementation the functions~\eqref{eq:aux_funcs} and~\eqref{eq:def_tilde} are simply evaluated for $\varkappa>0$ by adaptive Gauss--Kronrod quadrature with relative tolerance $10^{-12}$, which is accurate to about $10^{-13}$ for $\varkappa\le7.5$. Since the coefficients depend on $\varkappa$ alone, they are computed only a few times per operator, and the cost of the quadrature is negligible. A well-tuned combination of the Taylor expansions\rev{~\eqref{eq:C_expansion}--\eqref{eq:S_expansion}} for small $\varkappa$, the recurrences for moderate $\varkappa$ and quadrature elsewhere would be considerably cheaper. This will be needed when $\delta$ varies with the target point, as for a curvature-adapted $\delta$ or for nearly singular evaluation (Section~\ref{sec:conclusion}), where the coefficients must be recomputed at every target. The closed forms derived in this appendix also provide the values at $\varkappa=0$ used in the Laplace case and in the balancing of the moment systems, show that the entries are entire functions of $\varkappa$ (Appendix~\ref{sec:coeff_systems}), and reduce the finite-part integrals required for $\mathsf H$ to the improper integrals by which they are computed.}

\section{Bounds for \(C^\mathfrak{q}\)-norms of radial functions}
The following result will be used repeatedly to bound the norm of  radial factors making up the integrand~\eqref{eq:reg_integral_analysis}, and we therefore state it as a proposition.
\begin{proposition}\label{lem:rad_funcs}
Let $u:\R\to\R$ be an analytic even function and define
$$
f(x) := u(|x|), \qquad x\in \R^3.
$$
Then $f:\R^3\to\R$ is analytic and for all $r>0$ and $\mathfrak{q}\in\N$ there exists a constant $C=C(\mathfrak{q})>0$ such that
\begin{equation}\label{eq:sought_enequality}
\|f\|_{C^{\mathfrak{q}}(\overline{B_r(0)})}\le C\,\|u\|_{C^{2\mathfrak{q}}([0,r])}.
\end{equation}
\end{proposition}
\begin{proof} Since $u$ is even and analytic, there exists an analytic function $\tilde u:[0, \infty) \to \mathbb{R}$  such that $u(t)=\tilde u(t^2)$ (see, e.g.~\cite{whitney1943differentiable}). The mapping $x\mapsto |x|^2 = x\cdot x$ is a polynomial of degree two. Therefore, $f$ is a composition of analytic functions and hence analytic. In particular,  $f\in C^{\mathfrak{q}}(\overline{B_r(0)}))$ for all $r>0$ and $\mathfrak{q}\in\N$.

For any multi-index $\alpha=(\alpha_1,\alpha_2)\in\N_0^2$ with $|\alpha|=\alpha_1+\alpha_2=L\le \mathfrak{q}$,  Faà di Bruno's formula yields
$$
\der^\alpha f(x) = \der^\alpha\, \tilde u(|x|^2)=\sum_{\ell=\lceil L/2\rceil}^{L} P_{\alpha,\ell}(x)\,\tilde u^{(\ell)}(|x|^2),
$$
where $P_{\alpha,\ell}$ are homogeneous polynomials of degree  $\ell$. Since $|P_{\alpha,\ell}(x)|\le C_{\alpha, \ell}\,|x|^{\ell}$ for some constants $C_{\alpha,\ell}>0$ and since the functions $\tilde u^{(\ell)}$ are bounded on $[0,r]$, we obtain
\begin{equation}\label{eq:interm_lemma_0}
\sup_{x\in\overline{B_r(0)}}|\der ^\alpha f(x)|\leq \widetilde C_{\alpha}\max\{1,r^{L}\}\max_{\ell\in\{0,\ldots L\}}\sup_{t\in[0,r]}|\tilde u^{(\ell)}(t^2)|,\quad \widetilde C_{\alpha} := \sum_{\ell=\lceil L/2\rceil}^L C_{\alpha,\ell},\quad 0\leq|\alpha|= L\leq \mathfrak{q}.
\end{equation}

On the other hand, using induction, it can be shown that
\begin{equation}\label{eq:rep_der}
\tilde u^{(\ell)}(t^2)=\frac{1}{(2 t)^{2\ell-1}(\ell-1)!} \int_0^t\left(t^2-s^2\right)^{\ell-1} u^{(2\ell)}(s) \de s,\quad t\geq0,\quad\ell\in\N.
\end{equation}

Indeed, for $\ell=1$, formula~\eqref{eq:rep_der} gives
$$
\tilde u^{(1)}(t^2)=\frac{1}{(2 t)^{1}0!} \int_0^t\left(t^2-s^2\right)^0 u^{(2)}(s) \de s=\frac{u^{(1)}(t)-u^{(1)}(0)}{2t}=\frac{u^{(1)}(t)}{2t},
$$
where we used that $u^{(1)}$ is odd, and hence $u^{(1)}(0)=0$. On the other hand, differentiating the identity $u(t)=\tilde u(t^2)$ yields $u^{(1)}(t)=2t\tilde u^{(1)}(t^2)$ which is consistent with the expression above. Note that the singularity at $t=0$ is removable. This verifies the formula for $\ell=1$ and establishes the base case of the induction.

Assume now that~\eqref{eq:rep_der} holds for  $\ell=n\in\N$. Differentiating both sides of~\eqref{eq:rep_der} with $\ell=n$,  we get 
\begin{equation}\label{eq:interm_lemma}
2t\,\tilde u^{(n+1)}(t^2) = \frac{1}{2^{2n-1}t^{2n}(n-1)!}\int_0^t\left\{s^2(2n-1)-t^2\right\}\left(t^2-s^2\right)^{n-2}  u^{(2n)}(s) \de s.
\end{equation}

On the other hand, integrating by parts twice and using the fact that $u^{(2n+1)}(0)=0$, which follows from the evenness of $u$, we obtain
\begin{gather*}
\int_0^t\left(t^2-s^2\right)^{n} u^{(2n+2)}(s) \de s = -(t^2-s^2)^nu^{(2n+1)}(0)\Big|_{s=0}^{s=t}+2n\int_0^t s(t^2-s^2)^{n-1}u^{(2n+1)}(s)\de s\\
=2n t(t^2-s^2)^{n-1}u^{(2n)}(s)\Big|_{s=0}^{s=t}+2n\int_0^t (s^2(2n-1)-t^2)(t^2-s^2)^{n-2}u^{(2n)}(s)\de s,
\end{gather*}
where the boundary terms in the last integration by parts also vanish. Using this last identity in~\eqref{eq:interm_lemma}, we arrive at
$$
\tilde u^{(n+1)}(t^2)=\frac{1}{(2t)^{2n+1}n!}\int_0^t\left(t^2-s^2\right)^{n} u^{(2n+2)}(s) \de s.
$$
Since the formula~\eqref{eq:rep_der} holds for $\ell=1$ and the induction step shows that, if it holds for $\ell=n$, it also holds for $\ell=n+1$, the result follows by induction for all $\ell\in\N$.

Expanding the kernel $(t^2-s^2)^{\ell-1}$ in~\eqref{eq:rep_der} using the binomial theorem, we readily obtain
$$
|\tilde u^{(\ell)}(t^2)|\leq \frac{1}{2^{2\ell-1}(\ell-1)!}\left(\sum_{i=0}^{\ell-1}{{\ell-1}\choose{i}}\frac{(-1)^i}{2i+1}\right) \sup_{s\in[0,t]}|u^{(2\ell)}(s)|.
$$
This together with the identity $\tilde u(t^2)=u(t)$ yields that there exist  constants $C'_\ell>0$ such that
$$
\sup_{t\in[0,r]}|\tilde u^{(\ell)}(t^2)|\leq  C'_\ell \|u\|_{C^{2\ell}([0,r])},\quad \ell\in \N_0.
$$
Therefore, 
$$
 \max_{\ell\in\{0,\ldots,L\}}\sup_{t\in[0,r]}|\tilde u^{(\ell)}(t^2)|\leq  \widehat C_L \|u\|_{C^{2L}([0,r])},\quad 0\leq L\leq \mathfrak{q},
$$
with $\widehat C_L := \max_{\ell\in\{0,\ldots,L\}} C'_\ell.$
Using this last inequality in~\eqref{eq:interm_lemma_0} we  arrive at 
$$
\sup_{x\in\overline{B_r(0)}}|\der ^\alpha f(x)|\leq \widetilde C_\alpha \widehat C_L\max\{1,r^{L}\}  \|u\|_{C^{2L}([0,r])},\quad 0\leq |\alpha|=L\leq \mathfrak{q},
$$
from where, taking the maximum over all the multi-indices $\alpha\in\N_0^2$ such that $|\alpha|=L\in\{0,\ldots, \mathfrak{q}\}$, yields 
\begin{equation}\label{eq:almost_there}
\|f\|_{C^{\mathfrak{q}}(\overline{B_r(0)})}\le C'(\mathfrak{q})\,\max\{1,r^{\mathfrak{q}}\}\,\|u\|_{C^{2\mathfrak{q}}([0,r])},
\end{equation}
where $C'(\mathfrak{q}) := \max_{\alpha\in\N^2_0:0\leq|\alpha|\leq \mathfrak{q}}\{\widetilde C_\alpha\widehat C_{|\alpha|}\}$. 

We now distinguish two cases: $r \leq 1$ and $r > 1$. In the former case, the desired inequality~\eqref{eq:sought_enequality} follows directly from~\eqref{eq:almost_there} with $C(\mathfrak{q}):= C'(\mathfrak{q})$.
In the latter case, it follows from Faà di Bruno’s formula, together with the smoothness of the mapping $x \mapsto |x|$ and the boundedness of all its derivatives on the set $\{x \in \R^3 : |x| \geq 1\}$, that there exists $C''(\mathfrak{q})>0$ such that
$$\|f\|_{C^{\mathfrak{q}}\!\left(\overline{B_r(0)\setminus B_1(0)}\right)} \leq C''(\mathfrak{q})\, \|u\|_{C^{\mathfrak{q}}([1,r])} \leq C''(\mathfrak{q})\, \|u\|_{C^{2\mathfrak{q}}([1,r])}.$$
Combining this estimate with~\eqref{eq:almost_there} for $r=1$, yields the desired inequality with $C(\mathfrak{q}) :=\max\{C'(\mathfrak{q}),C''(\mathfrak{q})\}$. The proof is now complete.
\end{proof}

\rev{\section{Moment systems for the regularizing coefficients}\label{sec:coeff_systems}}

\rev{The coefficients $a_\ell$ in~\eqref{eq:reg_funcs}, determined by the moment systems~\eqref{eq:good_system}, depend on $\delta$ and $k$ only through the scaled wavenumber $\varkappa=\delta k$\rev{ of~\eqref{eq:scaled_wn}}. By Propositions~\ref{prop:prop_1} and~\ref{prop:prop_2}, the entries of $\mathrm A(\varkappa)$ and $\bold b(\varkappa)$ are entire functions of $\varkappa$, and at $\varkappa=0$ the systems reduce to the Laplace systems of Remarks~\ref{rem:laplace_SL} and~\ref{rem:laplace_DL_coeff} and~\eqref{eq:coeff_hyper}.}

\rev{In the square case, $\det\mathrm A(\varkappa)$ is therefore an entire function of $\varkappa$, and it does not vanish identically since $\det\mathrm A(0)\neq0$ by the explicit Laplace formulas. Its zeros are thus isolated and the square system fails only on a discrete set of values of $\varkappa$, but there the failure is a genuine loss of rank of the moment map, not a numerical artifact, and cannot be avoided within the square formulation.}

\rev{In the rectangular formulation of Remark~\ref{rem:ill-cond_mat}, $m$ moment conditions are imposed on $n=m+1$ coefficients and the minimum-norm solution $\bold a(\varkappa)=\mathrm A(\varkappa)^{\dagger}\bold b(\varkappa)$ is used; since $\mathfrak m=2(m-p)+1$ depends on $m$ alone, the extra coefficient does not lower the order. The number of conditions is $(\mathfrak m-1)/2$ for $\mathsf S$, $\mathsf K$, $\mathsf K^\top$ and $\mathsf W$, and $(\mathfrak m+1)/2$ for $\mathsf H$ (Table~\ref{tab:system_sizes}). Before the pseudoinverse is applied, each system is multiplied from the left by the closed-form inverse of the leading square block of $\mathrm A(0)$, which balances the rows without changing the minimum-norm solution.}

\begin{table}[h!]
\centering
\begin{tabular}{ccccc}
\hline
$\mathfrak m$ & $\mathsf S$ & $\mathsf K,\,\mathsf K^\top$ & $\mathsf H$ & $\mathsf W$ \\
\hline
$3$ & $1\times2$ & $1\times2$ & $2\times3$ & $1\times2$ \\
$5$ & $2\times3$ & $2\times3$ & $3\times4$ & $2\times3$ \\
$7$ & $3\times4$ & $3\times4$ & $4\times5$ & $3\times4$ \\
$9$ & $4\times5$ & $4\times5$ & $5\times6$ & $4\times5$ \\
\hline
\end{tabular}
\caption{\rev{Dimensions (number of moment conditions $\times$ number of coefficients) of the rectangular moment systems used in Section~\ref{sec:numerics}.}}
\label{tab:system_sizes}
\end{table}

\rev{In order to compare the conditioning of the two formulations, Figure~\ref{fig:moment_conditioning} displays the smallest singular value $\sigma_{\min}$ of the square and rectangular matrices $\mathrm A(\varkappa)$ of $\mathsf S$, $\mathsf K$ ($\mathsf K^\top$) and $\mathsf W$ for $\mathfrak m\in\{5,9\}$. The matrix of $\mathsf H$ is not shown, since for order $\mathfrak m$ it coincides with that of $\mathsf K$ for order $\mathfrak m+2$. The square matrices lose rank at isolated values of $\varkappa$, several of which lie in the range $\varkappa\le\pi$ used in Section~\ref{sec:numerics}. There $\sigma_{\min}$ drops to round-off level, although the dips appear shallower in the figure because $\varkappa$ is sampled at finitely many points. The rectangular matrices, in contrast, remain well conditioned, with $\sigma_{\min}\geq7.9\cdot10^{-4}$ for $\varkappa\le\pi$.}

\begin{figure}[htbp]
  \centering\color{revffcolor}
  \includegraphics[width=0.4\textwidth]{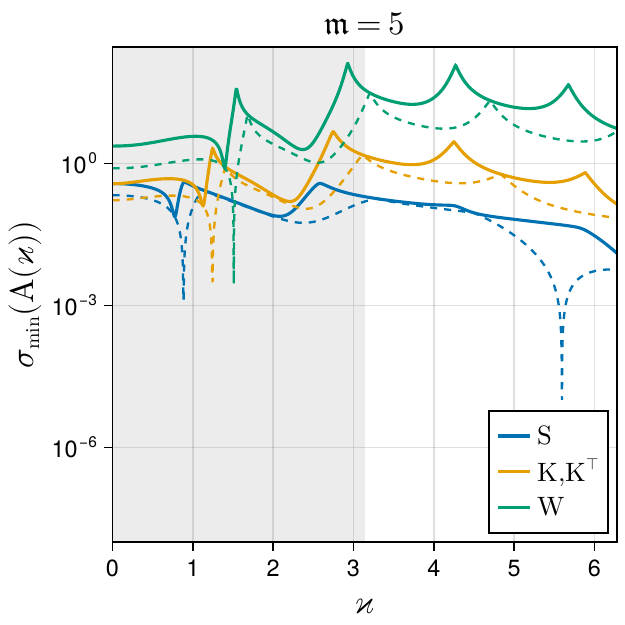}\qquad
  \includegraphics[width=0.4\textwidth]{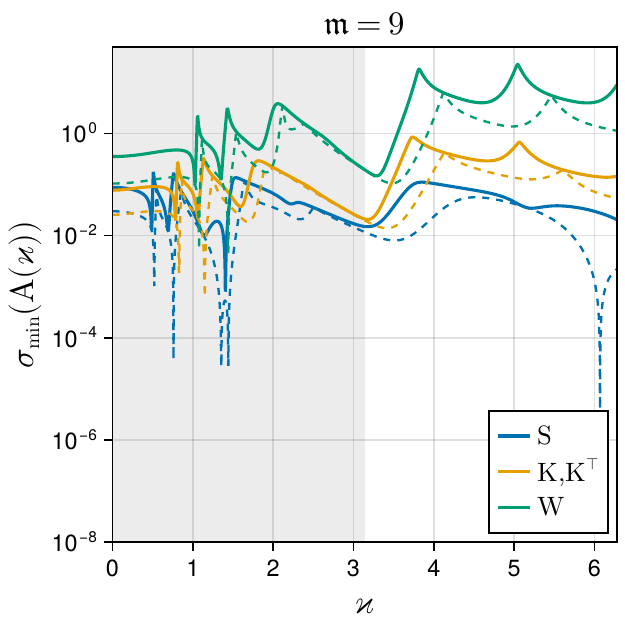}
  \caption{\rev{Smallest singular value of the moment matrices $\mathrm A(\varkappa)$ of $\mathsf S$, $\mathsf K$ ($\mathsf K^\top$) and $\mathsf W$ as a function of the scaled wavenumber $\varkappa=\delta k$, for $\mathfrak m=5$ (left) and $\mathfrak m=9$ (right). Dashed lines: square systems; solid lines: rectangular systems of Table~\ref{tab:system_sizes}. The shaded region $\varkappa\le\pi$ is the range visited in Section~\ref{sec:numerics}.}}
  \label{fig:moment_conditioning}
\end{figure}

\rev{The entries of $\mathrm A(\varkappa)$ and $\bold b(\varkappa)$ are computed as explained at the end of Appendix~\ref{sec:coeff}. Evaluating directly the moment integrals that the computed coefficients should annihilate gives residuals below $3\cdot10^{-11}$ (relative) over $\varkappa\in[0,\pi]$ for all four operators and $\mathfrak m\leq9$, far below the regularization errors of Section~\ref{sec:numerics}.}

\rev{\section{Improved convergence rates on the sphere}\label{sec:sphere}}

Improved convergence rates are expected when $\Gamma$ is a sphere due to its constant curvature. Indeed, suppose without loss of generality that $\Gamma$ is a sphere of radius $r>0$ centered at the origin. Since $\nu(x) = x/r$ for 
$x \in \Gamma$, one has the identities
$$
  \nu(y)\cdot(x-y) = \frac{x\cdot y - r^2}{r} = -\nu(x)\cdot(x-y)
  \qquad\text{and}\qquad
  |x-y|^2 = 2(r^2 - x\cdot y).
$$
It follows that
$$
  \frac{\Psi^{(\mathsf{K})}(x,y)}{|x-y|^2} = -\frac{1}{2r}
  \qquad\text{and}\qquad
  \frac{\Psi^{(\mathsf{W})}(x,y)}{|x-y|^4} = -\frac{1}{4r^2}
$$
are both constant, hence smooth. Consequently, for the sphere one may 
take $(p,s) = (1,1)$ for $\mathsf{K}$ and $\mathsf{K}^\top$, and 
$(p,s) = (2,2)$ for $\mathsf{W}$ in~\eqref{eq:int_form_1}, leading to improved convergence rates for the overall methodology. The resulting sphere-specific 
parameters and rates are displayed in Table~\ref{tab:ops_params_sphere}.

  \begin{table}[h!]
\centering
\begin{tabular}{ccccc}
\hline
Operator & $p$ & $s$ & $\mu^\star$ & $\mathfrak{o}^\star$ \\
\hline
$\mathsf{S}$ & $0$ & $0$ 
  & $\dfrac{\mathfrak{q}+1}{\mathfrak{q}+1+\mathfrak{m}}$ 
  & $\dfrac{(\mathfrak{q}+1)\mathfrak{m}}{\mathfrak{q}+1+\mathfrak{m}}$ \\[8pt]
$\mathsf{K},\,\mathsf{K}^\top$ & $1$ & $1$ 
  & $\dfrac{\mathfrak{q}+1}{\mathfrak{q}+1+\mathfrak{m}}$ 
  & $\dfrac{(\mathfrak{q}+1)\mathfrak{m}}{\mathfrak{q}+1+\mathfrak{m}}$ \\[8pt]
$\mathsf{H}$ & $2$ & $1$ 
  & $\dfrac{\mathfrak{q}+1}{\mathfrak{q}+2+\mathfrak{m}}$ 
  & $\dfrac{(\mathfrak{q}+1)\mathfrak{m}}{\mathfrak{q}+2+\mathfrak{m}}$ \\[8pt]
$\mathsf{W}$ & $2$ & $2$ 
  & $\dfrac{\mathfrak{q}+1}{\mathfrak{q}+1+\mathfrak{m}}$ 
  & $\dfrac{(\mathfrak{q}+1)\mathfrak{m}}{\mathfrak{q}+1+\mathfrak{m}}$ \\[4pt]
\hline
\end{tabular}
\caption{Sphere-specific parameters $(p,s)$ and optimal exponents 
$\mu^\star$, $\mathfrak{o}^\star$ for each boundary integral operator. 
For the sphere, $\mathsf{S}$, $\mathsf{K}$, $\mathsf{K}^\top$, and 
$\mathsf{W}$ all achieve the same rates, with $\mathfrak{o}^\star \to 
\mathfrak{q}+1$ as $\mathfrak{m}\to\infty$ and $\mathfrak{o}^\star\to 
\mathfrak{m}$ as $\mathfrak{q}\to\infty$.}
\label{tab:ops_params_sphere}
\end{table}

\rev{These rates can be verified numerically by repeating the combined regularization and discretization study of Section~\ref{sec:combined} on the unit sphere, represented exactly as the ellipsoid of Section~\ref{sec:validation}, with $\delta=c\,h^{\mu^\star}$ chosen according to the exponents of Table~\ref{tab:ops_params_sphere}. On the sphere the reference values $\mathsf B[\varphi](x)$ are moreover available in closed form, since the spherical harmonics are eigenfunctions of all four operators (see, e.g.,~\cite{maltez2024combined}).}

\bibliographystyle{abbrv}
\bibliography{References}
\end{document}